\documentclass[11pt, reqno]{amsproc} 

\usepackage{geometry}
\usepackage{color}
\usepackage{xcolor}
\usepackage{verbatim}
\usepackage[utf8x]{inputenc}
\usepackage{t1enc}
\usepackage[english]{babel}

\usepackage{amsmath,amssymb,amsthm}
\usepackage{mathrsfs,esint}
\usepackage{graphicx,wrapfig}
\usepackage[format=hang]{caption}
\usepackage{tikz}

\newcommand{\dd}								
	{\mathop{}\!\mathrm{d}}						
\newcommand{\ddn}[1]							
	{\mathop{}\!\mathrm{d^{#1}}}
\newcommand{\abs}[1]							
	{\left| #1 \right|}
\newcommand{\smallabs}[1]						
	{\lvert #1 \rvert}	
\newcommand{\norm}[1]							
	{\left\lVert #1 \right\rVert}	
\newcommand{\smallnorm}[1]						
	{\lVert #1 \rVert}						
\newcommand{\ip}[2]								
	{\left< #1 , #2 \right>}
	
\newcommand{\gromp}[2]							
	{\langle #1 \vert #2 \rangle}
\renewcommand{\H}{\mathbb{H}}

\newcommand*{\R}{{\mathbb R}}

\newcommand*{\N}{{\mathbb N}}

\newcommand*{\eps}{\varepsilon}

\newcommand*{\pip}{\varphi}

\newcommand{\desc}{\text{D}}
\newcommand{\hor}{\text{H}}
\newcommand{\asc}{\text{A}}
\newcommand{\pa}{$\sqcup$\hspace{1pt}}

\providecommand*{\vint}[1]{\mathchoice
          {\mathop{\vrule width 5pt height 3 pt depth -2.5pt
                  \kern -9pt \kern 1pt\intop}\nolimits_{\kern -5pt{#1}}}
          {\mathop{\vrule width 5pt height 3 pt depth -2.6pt
                  \kern -6pt \intop}\nolimits_{\kern -3pt{#1}}}
          {\mathop{\vrule width 5pt height 3 pt depth -2.6pt
                  \kern -6pt \intop}\nolimits_{\kern -3pt{#1}}}
          {\mathop{\vrule width 5pt height 3 pt depth -2.6pt
                  \kern -6pt \intop}\nolimits_{\kern -3pt{#1}}}}

\DeclareMathOperator{\Lip}{Lip}

\DeclareMathOperator{\diam}{diam}

\DeclareMathOperator{\len}{len}

\numberwithin{equation}{section}
\theoremstyle{plain}
\newtheorem{theorem}[equation]{Theorem}
\newtheorem{proposition}[equation]{Proposition}
\newtheorem{corollary}[equation]{Corollary}
\newtheorem{lemma}[equation]{Lemma}

\theoremstyle{definition}

\newtheorem{defn}[equation]{Definition}

\newtheorem{remark}[equation]{Remark}

\newtheorem{stassm}[equation]{Standing assumptions}

\begin{document}

\title[Warped products, solid hyperbolic fillings, and $D^{1,p} = N^{1,p} + \mathbb{R}$.]
{Warped products, solid hyperbolic fillings, and the identity $D^{1,p} = N^{1,p} + \mathbb{R}$.} 

\author[I. Kangasniemi]{Ilmari Kangasniemi}
\address{Department of Mathematical Sciences, P.O.~Box 210025, University of Cincinnati, Cincinnati, OH~45221-0025, U.S.A.}
\email{kangaski@ucmail.uc.edu}

\author[J. Kline]{Josh Kline}
\address{Department of Mathematical Sciences, P.O.~Box 210025, University of Cincinnati, Cincinnati, OH~45221-0025, U.S.A.}
\email{klinejp@ucmail.uc.edu}

\author[N. Shanmugalingam]{Nageswari Shanmugalingam}
\address{Department of Mathematical Sciences, P.O.~Box 210025, University of Cincinnati, Cincinnati, OH~45221-0025, U.S.A.}
\email{shanmun@uc.edu}

\author[G. Speight]{Gareth Speight}
\address{Department of Mathematical Sciences, P.O.~Box 210025, University of Cincinnati, Cincinnati, OH~45221-0025, U.S.A.}
\email{gareth.speight@uc.edu}

\thanks{I. Kangasniemi was partially supported by the NSF grant DMS~\#2247469. N. Shanmugalingam was partially supported by the NSF grant DMS~\#2348748. G. Speight was partially supported by the NSF grant DMS~\#2348715.}
\maketitle

\begin{abstract}
	We construct a large class 
	of metric measure spaces $Z$ which satisfy the identity
	$D^{1,p}(Z) = N^{1,p}(Z) + \mathbb{R}$, i.e.\ any measurable function $u \colon Z \to \R$ with an $L^p$-integrable upper gradient is a constant term away from being $L^p$-integrable. To do so, we construct a family of 
	hyperbolic fillings $\H_{\alpha, \beta}(Y)$, $\alpha, \beta \in (0, \infty)$, of a metric measure space 
	$Y$, via a warped product of $Y$
	with an exponentially weighted positive real line. We then show that for certain classes of $Y$,
	the above identity is satisfied for $Z=\H_{\alpha,\beta}(Y)$
	when $1\le p\le \beta/\alpha$.  
	We also show that under mild assumptions on $Y$, the warped product
$\H_{\alpha, \beta}(Y)$ is Gromov hyperbolic as a metric space and the Gromov boundary of $\H_{\alpha, \beta}(Y)$ is quasisymmetric to $Y$. \end{abstract}

\vskip .3cm


\medskip

\noindent
    {\small Mathematics Subject Classification (2020):
Primary: 46E36.
Secondary: 30L15, 53C23.
}

\tableofcontents

\section{Introduction}

Based on~\cite{CKKSS, GKorS, Maly}, the \emph{Dirichlet-Sobolev class} $D^{1,p}(Z)$, $1\le p<\infty$, is
the natural class of Dirichlet boundary data at infinity in the formulation of the Dirichlet problem for $p$-energy
minimizers on an unbounded metric measure space $(Z,d_{Z},\mu_{Z})$.
The class $D^{1,p}(Z)$ 
consists of measurable functions $u:Z\to[-\infty,\infty]$ which are integrable over balls and have an $L^p$-integrable 
upper gradient $g \colon Z \to [0, \infty]$; that is, functions $u \in D^{1,p}(Z)$  have globally finite $p$-energy, but are 
themselves not required to be $L^p$-integrable on the whole space. In contrast, the 
\emph{Newtonian Sobolev class} $N^{1,p}(Z)$, which consists of functions $u \in D^{1,p}(Z)$ that are also globally 
$L^p$-integrable over $Z$, is widely used in the study of Dirichlet problems on bounded domains in metric measure spaces.
	
Motivated by such Dirichlet problems on unbounded spaces, we wish to know
when every 
function $u \in D^{1,p}(Z)$ is $L^p$-integrable over $Z$ after subtracting a suitable constant; that is, when do we have
$D^{1,p}(Z) = N^{1,p}(Z) + \R$. 
This is since if one does have $D^{1,p}(Z) = N^{1,p}(Z) + \R$, then one can only find 
solutions with finite $p$-energy for a given Dirichlet problem on $Z$ when the 
boundary data at infinity is in some sense constant. 
	
It was shown in~\cite[Theorem~1.1]{GKS} that under most circumstances, one 
has $D^{1,p}(Z) \ne N^{1,p}(Z) + \R$ for unbounded spaces $Z$,
ruling out 
many simple properties as a complete characterization of when 
$D^{1,p}(Z)$ and $N^{1,p}(Z) + \R$ agree. However, for the 
standard hyperbolic $n$-space $\mathbb{H}^n$ with 
$n\ge 2$, we have 
$D^{1,p}(\mathbb{H}^n) = N^{1,p}(\mathbb{H}^n) + \R$ if and only if 
$1 \le p \le n-1$; see e.g.\ \cite[Theorem 5.8]{Str}, \cite{Pan}, or 
\cite[Section 8]{GKS}. It is therefore natural to ask whether there 
is a more general class of examples for which this dichotomy occurs, and if so, 
what is the geometric or analytic significance of the cutoff, which for $\mathbb{H}^n$ occurs at $p=n-1$. 
	
 In this article, our objective is to construct a broad class of Gromov hyperbolic spaces for which this dichotomy phenomenon occurs. 
For the examples we construct, the threshold of dichotomy seems to be not related to the
topological dimension of the space nor the local (measure-theoretic) regularity dimension of the space. Instead, the 
threshold appears to depend on the interaction between how fast the volume grows and how fast the metric
diverges. 
	
 On Riemannian manifolds, based on the works of Holopainen and Koskela~\cite{H3, HoloKos}, we know that a sufficiently 
fast volume growth guarantees that the manifold is $p$-hyperbolic, that is, it is hyperbolic from a potential theoretic 
point of view. However, $p$-hyperbolicity of a manifold $M$ does not by itself guarantee that 
$D^{1,p}(M)=N^{1,p}(M)+\R$, as shown in~\cite{GKS}. Even possessing only one end at 
infinity does not guarantee this property, as seen in~\cite[Example~8.1]{BBS2}. The 
examples we construct have more than just rapid volume growth; they are also 
isotropic about a ``base'' point.
	
The main tool for constructing our examples is to consider two well-behaved metric spaces $X, Y$, and construct a 
warped product $X \times_\varphi Y$ out of them. The warping function $\pip \colon X \to [0, \infty)$ acts 
on one of the two spaces,
and dictates how fast in the Cartesian product $X \times Y$ two points in $Y$ move apart as we change 
the $X$-coordinate. There exists a wide body of literature on such warped products in the 
Riemannian setting, and several prior works have considered them in the metric setting, 
such as \cite{AB1, AB2, AB3, Ch, GGN}. The works~\cite{AB1,AB2,AB3, Ch} are concerned
with constructing spaces of non-positive Alexandrov curvature. Spaces of non-positive Alexandrov
curvature, or $\text{CAT}(0)$-spaces, are non-positively curved at all scales. In contrast, our construction
of warped products yields Gromov hyperbolic spaces, which are non-positively curved at large scales but
do not have local curvature control.
	
We now describe our objects of study and the main results of this article in detail.

\subsection{Solid hyperbolic fillings}

Let $(Y, d_Y)$ be a length metric space. For every $\alpha \in (0, \infty)$, we define the 
\emph{solid hyperbolic filling} $\H_\alpha(Y)$, and the \emph{punctured solid hyperbolic filling} 
$\H^{\circ}_\alpha(Y)$ to be the warped product metric spaces
\begin{align*}
	\H_\alpha(Y) &:= [0, \infty) \times_{\psi_\alpha} Y
		&& \text{where } \psi_\alpha(t) = \sinh^\alpha(t),\\
	\H^{\circ}_\alpha(Y) &:= [0, \infty) \times_{\varphi_\alpha} Y
		&& \text{where } \varphi_\alpha(t) = e^{\alpha t}.
\end{align*}
The exact metric on $\H_\alpha(Y)$ and $\H^{\circ}_\alpha(Y)$ also depends on the choice of how the individual 
metrics on $[0, \infty)$ and $Y$ are combined into a warped product metric using a coordinate-increasing 
unitary norm (see Subsection~\ref{subsect:product_metrics})
$\norm{\cdot}$ on $\R^2$.
We provide a more precise definition of warped product spaces in Section~\ref{sect:warped_products}.  

Given a Borel regular measure $\mu_Y$ on $Y$ for which balls have finite measure, for each $\beta\in (0,\infty)$ we construct
measures   
on $\H_\alpha^\circ(Y)$ and $\H_\alpha(Y)$ via 
the product measures 
$d\mu_\beta(t,x) := (e^{\beta t} dt) \times d\mu_Y(x)$ and $d\nu_\beta(t,x) := (\sinh^\beta(t) dt) \times d\mu_Y(x)$,
respectively.
We use $\H_{\alpha, \beta}(Y)$ to denote $\H_\alpha(Y)$ equipped with a measure defined using $\nu_\beta$, and 
$\H^{\circ}_{\alpha, \beta}(Y)$ 
to denote $\H^{\circ}_\alpha(Y)$ equipped with a measure defined using $\mu_\beta$.
See Subsection~\ref{subsect:measures_on_warped_prod}  for the precise definitions. We restrict ourselves to 
considering $\H_{\alpha, \beta}(Y)$ 
only in the case $\mu_Y(Y) < \infty$, as the corresponding measure on $\H_{\alpha, \beta}(Y)$ fails to be locally finite otherwise. 

In~\cite{BoSc}, Bonk and Schramm prove that every Gromov
hyperbolic space of uniformly locally bounded growth is roughly similar to a geodesically convex subset of the 
hyperbolic space $\mathbb{H}^n$.
However, in 
most of the other literature concerning analysis on metric spaces, ``hyperbolic filling'' generally refers to a graph-based 
construction; see e.g.\ the works~\cite{BP,BoSa,BBS,Car}.
The advantage of such an approach is that graph-based hyperbolic fillings of $Y$ 
may support Poincar\'e inequalities even when 
$(Y,d_Y,\mu_Y)$ does not support a Poincar\'e inequality; see e.g.\ \cite{BBS} for details. Our approach with 
$\H_{\alpha, \beta}(Y)$ and $\H^{\circ}_{\alpha, \beta}(Y)$ is instead suitable for studying $Y$ that already 
possess some form of Poincar\'e inequality, and is visually closer to how the standard hyperbolic space $\H^n$ 
fills in its visual
boundary $\mathbb{S}^{n-1}$; in fact, with an $\ell^2$ product structure, the solid hyperbolic filling 
$\H_{1, n-1}(\mathbb{S}^{n-1})$ is a model of $\H^n$. 
In general the graph-like hyperbolic fillings do not satisfy the conclusion of Theorem~\ref{thm:solid_hyp_filling_is_PI},
see for instance~\cite{BBS}.

\subsection{Main results}

To obtain the desired properties of the warped product space, we need some functional inequalities to be satisfied by 
a metric measure
space $(Z,d_Z,\mu_Z)$. We consider the following two inequalities.

\begin{defn}\label{defn:Sobolev-Poincare}
Let $(Z, d_Z, \mu_Z)$ be a metric measure space, and let $p, q \in [1, \infty)$. 
We say that $Z$ is a \emph{$(p,q)$-Sobolev--Poincar\'e 
space (with constant $C$)} if 
there is a constant $C>0$ such that
for all $u \in D^{1,q}(Z)$ and all upper gradients $g \in L^q(Z)$ of $u$, there exists  
$c_u \in \R$ 
for which
\[
	\norm{u-c_u}_{L^p(Z,\mu_Z)} \le C \norm{g}_{L^q(Z,\mu_Z)}.
\]
\end{defn} 
This is different from a classical $(p,q)$-Poincar\'e inequality in the sense that it provides an estimate on all of $Z$, 
while the classical $(p,q)$-Poincar\'e inequality provides an estimate on balls $B_Z \subset Z$.

\begin{defn}\label{defn:weakPI}
Let $(Z, d_Z, \mu_Z)$ be a metric measure space, and let $p, q \in [1, \infty)$. 
Suppose also that $0 < \mu_Z(B_Z) < \infty$ for every ball $B_Z \subset Z$. We then call $Z$ an 
\emph{$\infty$-weak local $(p,q)$-Poincar\'e space} if for every $x \in Z$, there exist constants 
$C_x, r_x > 0$ such that, for every $u$ that is integrable over balls in $Z$ 
and every upper gradient $g$ of $u$
on $Z$, we have
\[
	\norm{u - u_{B_Z(x, r_x)}}_{L^p(B_Z(x, r_x),\mu_Z)} \le C_x \norm{g}_{L^q(Z,\mu_Z)}.
\]
The justification behind our choice of terminology is that a $\lambda$-weak $(p,q)$-Poincar\'e inequality 
as in~\cite{HKST},
with $\lambda > 1$, 
estimates the $L^p$-norm of $u-u_{B_Z(x, r)}$ over a ball $B_Z(x, r)$ with the $L^q$-norm of $g$ over the expanded ball 
$B_Z(x, \lambda r)$; 
in our case, we have an expansion factor $\lambda = \infty$, as our upper gradient $g$ is integrated over all of $Z$.
\end{defn}

We now have enough preliminary definitions to state our main results. Part~(b) of the first of our main results is related to
\cite[Theorem~1.1(ii)]{GKS}. Readers consulting~\cite{GKS} should keep in mind
that the parameter $\beta$ as considered there has a different role from that of $\beta$ in Theorem~\ref{thm:solid_hyp_filling_is_PI}
below.

\begin{theorem}\label{thm:solid_hyp_filling_is_PI}
Let $(Y, d_Y, \mu_Y)$ be a length metric measure space 
which satisfies Standing Assumptions~\ref{stassm:extra_assumptions} with $\mu_Y(Y)<\infty$,
and let $\alpha,\beta$ be two fixed positive real numbers. Suppose also that
$Y$ is an $\infty$-weak local $(p,p)$-Poincar\'e space for some $p$ with $1\le p<\infty$.
\begin{enumerate}
\item[(a)] If $p\le \beta/\alpha$, then the metric measure spaces $\H^{\circ}_{\alpha, \beta}(Y)$ 
and $\H_{\alpha, \beta}(Y)$
are  $(p,p)$-Sobolev--Poincar\'e 
spaces with constant $C = C(\beta, p)$,  and consequently
\begin{align}
	D^{1,p}(\H^{\circ}_{\alpha, \beta}(Y)) &= N^{1,p}(\H^{\circ}_{\alpha, \beta}(Y)) + \R, \label{eq:D_is_N+R_Hcirc1}\\
	D^{1,p}(\H_{\alpha, \beta}(Y)) &= N^{1,p}(\H_{\alpha, \beta}(Y)) + \R. \label{eq:D_is_N+R_H1}
\end{align}
\item[(b)] If $p > \beta/\alpha$ and $Y$ has at least two points, 
then \eqref{eq:D_is_N+R_Hcirc1} 
and \eqref{eq:D_is_N+R_H1} fail to hold.
\end{enumerate}
\end{theorem}

We point out here that a type of \emph{local} $(p,p)$-Poincar\'e inequality on $Y$ is sufficient in 
obtaining a type of \emph{global} $(p,p)$-Sobolev--Poincar\'e inequality on the solid hyperbolic filling. It is also notable that the
constant $C$ of the resulting $(p,p)$-Poincar\'e inequality is independent of $Y$.
In fact, a more general statement holds for warped products of a pair of metric measure spaces
$(X,d_X,\mu_X)$ and $(Y, d_Y,\mu_Y)$, and
acts as the basis for the proof of Theorem~\ref{thm:solid_hyp_filling_is_PI}.
The ideas behind this result are partially inspired by 
Strichartz's proof of~\cite[Theorem 5.8]{Str}.

\begin{theorem}\label{thm:warped_product_PI}
Let $(X, d_X, \mu_X)$ and $(Y, d_Y, \mu_Y)$ be length metric measure spaces
which satisfy Standing Assumptions~\ref{stassm:extra_assumptions} with $\mu_X(X)=\infty$ and $\mu_Y(Y)<\infty$,
and let $1\le p <\infty$. 
Suppose that $\varphi \colon X \to [0, \infty)$ is a continuous function 
for which 
$\varphi(x) > 0$ for $\mu_X$-a.e.\ $x \in X$. 
Moreover, suppose that all of the following conditions hold.
	\begin{enumerate} 
		\item $X$ is a $(p,p)$-Sobolev--Poincar\'e space with constant $C_X$. 
		\item $Y$ is an $\infty$-weak local $(p,p)$-Poincar\'e space.
		\item There exists a measurable set $E \subset X$ with $\mu_X(E)>0$ and
		\[
			\varphi^{-p} \in L^\infty(E, \mu_X) \quad \text{and} \quad \varphi^{-p} \notin L^1(E, \mu_X).
		\]
	\end{enumerate}
	Then $X \times_\varphi Y$, equipped with the measure $\mu_\varphi$ induced by the  product
	measure $\mu_X \times \mu_Y$ as described in Subsection~\ref{subsect:measures_on_warped_prod},
	is a $(p,p)$-Sobolev--Poincar\'e space with constant $C_X$, and therefore
	\[
		D^{1,p}(X \times_\varphi Y) = N^{1,p}(X \times_\varphi Y) + \R.
	\] 
\end{theorem}

To give an idea of the types of functions $\pip$ the above theorem can be applied to, we look at the case of 
$X=[0,\infty)$ equipped with the Euclidean metric 
$d_{\text{Eucl}}$ and $Y=\mathbb{S}^{n-1}$ for some $n\ge 2$. We define $\pip:X\to(0,\infty)$ by 
$\pip(t)=e^{\alpha t}$. Then with $E=X$ itself, we see that $\pip^{-p}$ is bounded on $E$ for each $p\ge 1$.
However, when $X$ is equipped with the measure $\mu_X$ given by $d\mu_X(t):=e^{\beta t}\, dt$,
we see that $\pip^{-p}\in L^1(E,\mu_X)$ if and only if $\beta/\alpha<p$; thus when $\beta/\alpha\ge p$,
the Euclidean space $X$, equipped with the measure $\mu_X$, satisfies the hypothesis~(3) of the above theorem
for this choice of
$\pip$. This is part of the family of examples covered by Theorem~\ref{thm:solid_hyp_filling_is_PI}
above. See the discussion at the beginning of Section~\ref{Sec:Theorem-1.3} below for the verification that this 
example does satisfy condition~(1) of Theorem~\ref{thm:warped_product_PI}.

We show next
that under mild assumptions on $Y$, the spaces $\H_{\alpha}(Y)$ and $\H^{\circ}_\alpha(Y)$ are 
Gromov hyperbolic with Gromov boundary quasisymmetric to $Y$,
which also justifies our terminology of ``solid hyperbolic filling''.

\begin{theorem}\label{thm:solid_filling_is_Gromov_hyp}
	Let $(Y, d_Y)$ be a length space, and let $\alpha > 0$. 
	Then $\H_{\alpha}(Y)$ is Gromov $\delta$-hyperbolic with $\delta = \delta(\alpha)$,
	and if $Y$ is bounded, then $\H^{\circ}_\alpha(Y)$ is Gromov $\delta$-hyperbolic 
	with $\delta = \delta(\alpha, \diam (Y))$. 
	Moreover, if $Y$ is a complete bounded length space, 
	then for small enough values of $\eps > 0$, the Gromov boundaries 
	$\partial_G \H_{\alpha}(Y)$ and $\partial_G \H^{\circ}_\alpha(Y)$ 
	equipped with their standard
	induced distance $d_\eps$ are quasisymmetrically homeomorphic to $Y$.
\end{theorem}

More general versions of Theorems~\ref{thm:solid_hyp_filling_is_PI}, \ref{thm:warped_product_PI},
and~\ref{thm:solid_filling_is_Gromov_hyp}
are given in later sections of this paper, see Theorems \ref{thm:solid_hyp_filling_is_PI-gen},
\ref{thm:warped_product_PI-ext}, and~\ref{thm:solid_filling_is_Gromov_hyp-II}. 
These general versions partially eliminate the requirement that $\mu_Y(Y)<\infty$.
It is possible that this requirement can be fully eliminated if we do not require that balls in the warped 
product space have finite
measure, but we choose not to pursue this direction of investigation here.
The requirement of $(X,d_X)$ and $(Y,d_Y)$ being length spaces is also not 
essential, as Theorems~\ref{thm:solid_hyp_filling_is_PI}, \ref{thm:warped_product_PI} 
and~\ref{thm:solid_filling_is_Gromov_hyp}
hold true even if the metric spaces are merely quasiconvex rather than 
length spaces, with possible extra dependence on quasiconvexity data in some of the 
constants; see Remarks~\ref{rem:metr_derivs_in_length_metrics}, \ref{rem:QC2length}, 
and \ref{rem:Poincare_bilipschitz_invariance} below.

\subsection*{Acknowledgments}

We thank Ilkka Holopainen for bringing to our attention the article \cite{Pan}; this article 
ultimately set us on the path to Theorem~\ref{thm:solid_hyp_filling_is_PI}.

\section{Background}

\subsection{Conventions and standing assumptions}\label{sub:pseudobable}

We fix some notation and terminology conventions that we use throughout this article. 

\begin{stassm}\label{stassm:extra_assumptions}
 In literature, a metric measure space generally refers to a triple $(Z,d_Z,\mu_Z)$ 
where $Z$ is a non-empty set, $d_Z$ a metric on $Z$, and $\mu_Z$ is a measure on $Z$.
In this paper we will use the term \emph{metric measure space} in a more restrictive manner by also
requiring the following three conditions: 
\begin{itemize}
\item $(Z,d_Z)$ is separable.
\item $\mu_Z$ is a Borel regular outer measure, that is, it is defined on \emph{all} subsets of $Z$. 
\item $0<\mu_Z(B)<\infty$ for all balls $B\subset Z$.
\end{itemize}
Here, an outer measure $\mu_Z$ is \emph{Borel regular} if 
Borel sets are $\mu_Z$-measurable and for each set $A\subset Z$ we have that
$\mu_Z(A) = \inf_F \mu_Z(F)$ with the infimum taken over all Borel sets $F\subset Z$ with $A\subset F$.
By saying that a set $A\subset Z$ is \emph{$\mu_Z$-measurable}, we mean that the Carath\'eodory criterion holds, i.e.~that
for every set $S\subset Z$ we have $\mu_Z(S)=\mu_Z(S\setminus A)+\mu_Z(S\cap A)$.
\end{stassm}

Under our conventions, a \emph{pseudometric} on a set $Z$ is a metric that is allowed to have zero distances; i.e.\ a pseudometric is a function $d_Z \colon Z \times Z \to [0, \infty)$ with $d_Z(x,x) = 0$, $d_Z(x,y) = d_Z(y,x)$ and 
$d_Z(x,y) \le d_Z(x,z) + d_Z(z,y)$ for all $x,y,z \in Z$. A pseudometric space $(Z, d_Z)$ induces a metric space $(X, d_X)$, where 
$X = Z/{\sim}$ is the quotient of $Z$ under the equivalence relation where $z \sim w$ if $d_Z(z,w) = 0$ for $z, w \in Z$, 
and the metric on $X$ is given by $d_X([z], [w]) = d_Z(z, w)$ for all $z, w \in Z$; for details, see e.g.\ \cite[Chapter 4 Theorem 15]{Kel}.

Given a metric or pseudometric space $(Z,d_Z)$ and $z\in Z$, $r>0$, 
we denote the open metric ball centered at $z$ with radius $r$ by
$B_Z(z,r)$ or $B_{d_Z}(z, r)$; that is, $B_Z(z,r)= B_{d_Z}(z, r) = \left\{x\in Z\, :\, d_Z(z,x)<r\right\}$. A 
\emph{path} (or \emph{curve}) on ($Z$, $d_Z)$ is a 
continuous map $\gamma\colon[a,b]\to Z$ for some $a,b\in\R$ with $a<b$.

\subsection{Product metrics}\label{subsect:product_metrics}

We say that a norm $\norm{\cdot}$ on $\R^2$ is \emph{coordinate-increasing} if $\norm{(a,c)} \le \norm{(b, d)}$ whenever 
$0 \le a \le b$ and $0 \le c \le d$. We also say that a norm $\norm{\cdot}$ on $\R^2$ is \emph{unitary} if 
$\norm{(1,0)} = \norm{(0, 1)} = 1$. We note that for all $p \in [1, \infty]$, the $\ell^p$-norm on $\R^2$ is coordinate-increasing 
and unitary. For an example of a norm on $\R^2$ that is not coordinate-increasing, 
consider e.g.\ $\norm{(a,b)} = (1/4)\abs{a+b} + (3/4)\abs{a-b}$, for which $\norm{(1,0)} = \norm{(0,1)} = 1$ but $\norm{(1,1)} = 1/2$. 

As evidenced by the following lemma,
these two properties are relevant if one wishes to define a product metric on $X \times Y$ by the formula
\begin{equation}\label{eq:general_product_metric}
	d_{\norm{\cdot}}((x,y), (x', y')) = \norm{(d_X(x,x'), d_Y(y, y'))}
\end{equation}
where $d_X$ and $d_Y$ are metrics on $X$ and $Y$, respectively.

\begin{lemma}\label{lem:product_metrics}
	Let $\norm{\cdot}$ be a norm on $\R^2$. Then the following conditions are equivalent.
	\begin{enumerate}
		\item \label{enum:coord_incr}
		The norm $\norm{\cdot}$ is coordinate-increasing.
		\item \label{enum:valid_prod_metr}
		For all metric spaces $(X, d_X)$ and $(Y, d_Y)$, 
		\eqref{eq:general_product_metric} defines a metric on $X \times Y$.
		\item \label{enum:valid_prod_metr_R2}
		For $(X, d_X) = (Y, d_Y) = (\R^2, d_{\text{Eucl}})$, \eqref{eq:general_product_metric} 
		defines a metric on $X \times Y$.
	\end{enumerate}
	If any of the above hold, then $d_{\norm{\cdot}}$ induces the usual product topology on $X \times Y$. Moreover, if additionally $\norm{\cdot}$ is unitary, then for all $x \in X$ and $y \in Y$, $X$ and $Y$ are isometric to the subspaces $X \times \{y\}$ and $\{x\} \times Y$ of $(X \times Y, d_{\norm{\cdot}})$, respectively.
\end{lemma}

\begin{proof}
	We only prove here that \eqref{enum:coord_incr}--\eqref{enum:valid_prod_metr_R2} are equivalent, as the rest is standard. Here, \eqref{enum:valid_prod_metr}$\implies$\eqref{enum:valid_prod_metr_R2} is clear, and \eqref{enum:coord_incr}$\implies$\eqref{enum:valid_prod_metr} follows by estimating that
	\begin{align*}
		\norm{(d_X(x,x'), d_Y(y, y'))}
		&\le \norm{(d_X(x,x'') + d_X(x'', x'), d_Y(y, y'') + d_Y(y'', y'))}\\
		&\le \norm{(d_X(x,x''), d_Y(y, y''))} + \norm{(d_X(x'',x'), d_Y(y'', y'))}
	\end{align*}
	for $(x,y), (x', y'), (x'', y'') \in X \times Y$, using the coordinate-increasing property of $\norm{\cdot}$ in the first step. For \eqref{enum:valid_prod_metr_R2}$\implies$\eqref{enum:coord_incr}, if $0 \le a \le b, 0 \le c \le d$, then there 
exist points $v_a, v_b, v_c, v_d \in \R^2$ with $\abs{v_a} = a$, $\abs{v_c} = c$, $\abs{v_b} = \abs{v_b - v_a} = b/2$, and 
$\abs{v_d} = \abs{v_d - v_c} = d/2$. Therefore, if $d_{\norm{\cdot}}$ is a metric on $\R^2 \times \R^2$, 
then the triangle inequality of $d_{\norm{\cdot}}$ yields
	\begin{align*}
		\norm{(a,c)}
		&= d_{\norm{\cdot}}((0,0), (v_a, v_c))\\
		&\le d_{\norm{\cdot}}((0,0), (v_b, v_d)) + d_{\norm{\cdot}}((v_b,v_d), (v_a, v_c))
		= 2\norm{(b/2, d/2)} = \norm{(b,d)},
	\end{align*}
	proving \eqref{enum:coord_incr}.
\end{proof}

The following estimate is also convenient to observe.

\begin{lemma}\label{lem:unitary_coord_incr_norms}
	Let $\norm{\cdot}$ and $\norm{\cdot}'$ be unitary coordinate-increasing norms on $\R^2$. Then
	\[
		\frac{1}{2}\norm{(a,b)}' \le \norm{(a,b)} \le 2\norm{(a,b)}'.
	\]
	for all $a, b \in [0, \infty)$. In particular, $(X \times Y, d_{\norm{\cdot}})$ and $(X \times Y, d_{\norm{\cdot}'})$ are 2-bilipschitz equivalent.
\end{lemma}

\begin{proof}
For $a, b \in [0, \infty)$, noting that $\norm{(a,0)} = a = \norm{(a,0)}'$ and $\norm{(0,b)} = b = \norm{(0,b)}'$, we may estimate that
	\[
		\norm{(a,b)} \le \norm{(a, 0)} + \norm{(0, b)} 
		= \norm{(a, 0)}' +  \norm{(0, b)}'
		\le \norm{(a, b)}' +  \norm{(a, b)}'
		= 2\norm{(a,b)}'.
	\]
	The other inequality is proven identically with the roles of the norms reversed.
\end{proof}

We highlight that the estimate of Lemma \ref{lem:unitary_coord_incr_norms} only holds 
when both $a$ and $b$ are non-negative (or both non-positive); the bilipschitz comparison constant between two unitary coordinate-increasing norms on $\R^2$ may still be 
arbitrarily large, as evidenced by the unitary coordinate-increasing norms 
$\norm{(a, b)} = \eps \abs{a} + \eps \abs{b} + (1-\eps) \abs{a + b}$ for $0 < \eps \le 1$.

In most instances, unless otherwise specified, we let $\norm{\cdot}$ be the $\ell^1$-norm for our standard choice of product metric for the sake of computational convenience. That is, if $(X, d_X)$ and $(Y, d_Y)$ are metric spaces, we equip $X \times Y$ with the $\ell^1$-metric $d$ given by 
\begin{equation}\label{eq:l1-metric}
	d((x,y),(x',y'))=d_X(x,x')+d_Y(y,y').
\end{equation}

\subsection{Paths and length metrics}

Next, we proceed to recall some basic results and notions related to 
rectifiable and absolutely continuous paths. For a more in-depth treatment, we 
refer the reader to e.g.\ \cite[Section~5.1]{HKST}. Note, however, that the theory in \cite{HKST}
is presented in the setting of metric spaces, but we have to present the results
in the more general setting of pseudometric spaces.

Given a pseudometric space $(Z,d_Z)$ and a path $\gamma:[a,b]\to Z$, 
the length of $\gamma$
is given by 
\[
	\len_{d_Z}(\gamma):=\sup\bigg\lbrace
 	\sum_{i=1}^n d_Z(\gamma(t_{i-1}),\gamma(t_i))\, :\, n\in\N\text{ and }a=t_0<t_1<\cdots<t_n=b
	\bigg\rbrace.
\]
A path $\gamma$ with $\len_{d_Z}(\gamma) < \infty$ is said to be \emph{rectifiable}. Moreover, the \emph{length function} 
$s_\gamma \colon [a,b] \to [0, \infty]$ of a path $\gamma \colon [a,b] \to Z$ is given by 
$s_\gamma(t) = \len_{d_Z}(\gamma \vert_{[a, t]})$ for $t \in [a,b]$. A rectifiable path $\gamma:[a,b]\to Z$ is said to be 
\emph{absolutely continuous} if the length function $s_\gamma:[a,b]\to[0,\infty)$ is
absolutely continuous.
In addition to these, the \emph{metric speed} or the \emph{metric derivative}
of a path $\gamma \colon [a,b] \to Z$ at $t \in [a,b]$, if it exists, is given by the limit
\[
	\abs{\gamma'(t)} := \lim_{u \to t} \frac{d_Z(\gamma(u), \gamma(t))}{\abs{u - t}}.
\]

We note that if $(\tilde{Z}, d_{\tilde{Z}})$ is the induced metric space $\tilde{Z} = Z/{\sim}$ of the pseudometric space $(Z, d_Z)$, then we can project a path $\gamma \colon [a,b] \to Z$ to a path $\sigma \colon [a,b] \to \tilde{Z}$ by $\sigma(t) = [\gamma(t)]$ for all $t \in [a,b]$. Conversely, we can lift a path $\sigma \colon [a,b] \to \tilde{Z}$ to a path $\gamma \colon [a,b] \to Z$ by selecting $\gamma(t)$ to be an arbitrary element of the equivalence class $\sigma(t)$ for every $t \in [a,b]$, though this lift is potentially non-unique. Since $d_Z(z, w) = d_{\tilde{Z}}([z], [w])$ for all $z, w \in Z$, we conclude that in both the case of projecting $\gamma$ and the case of lifting $\sigma$, the path $\gamma$ is continuous with respect to the pseudometric $d_Z$ if and only if $\sigma$ is continuous with respect to 
the metric $d_{\tilde{Z}}$, and we have
\begin{equation}\label{eq:metric_proj_props}
	\len_{d_Z}(\gamma) = \len_{d_{\tilde{Z}}}(\sigma), \quad
	s_{\gamma} = s_{\sigma}, \quad \text{and} \quad
	\abs{\gamma'(t)} = \abs{\sigma'(t)}
\end{equation}
for all $t \in [a, b]$ such that $\abs{\gamma'(t)}$ or $\abs{\sigma'(t)}$ exists.

If $\gamma \colon [a,b] \to Z$ is absolutely continuous, then for a.e.\ $t \in [a,b]$, the metric speed $\abs{\gamma'(t)}$ exists and satisfies
\begin{equation}\label{eq:met_deriv_and_length_function}
	\abs{\gamma'(t)} = s_\gamma'(t).
\end{equation}
In the metric setting, this is proven in e.g.\ \cite[Theorem 4.4.8]{HKST}, and the pseudometric version follows from the metric version by \eqref{eq:metric_proj_props}. Consequently, we have
\[
	\len_{d_Z}(\gamma) = \int_a^b \abs{\gamma'(t)} \, dt.
\]

For the rest of this section we consider $(Z,d_Z)$ to be a metric space.
We recall that every rectifiable path $\gamma \colon [a,b] \to Z$ has an absolutely continuous
re-parametrization 
$\tilde{\gamma} \colon [\tilde{a}, \tilde{b}] \to Z$. 
Indeed, there is a re-parametrization $\gamma_s:[0,\len_{d_Z}(\gamma)]\to Z$,
called an \emph{arc-length parametrization} of $\gamma$,
such that $|\gamma_s^\prime(t)|=1$ for almost every $t$, see~\cite[Proposition 5.1.8]{HKST}.
The \emph{line integral} of a 
Borel function $\rho \colon Z \to [0, \infty]$ over a rectifiable curve $\gamma$ is then given by
\[
	\int_{\tilde{a}}^{\tilde{b}} \rho(\tilde{\gamma}(t)) \abs{\tilde{\gamma}'(t)} \, dt=:\int_\gamma\rho\, ds,
\]
where $\tilde{\gamma}$ is an absolutely continuous re-parametrization of $\gamma$; this integral is independent 
of the choice of re-parametrization $\tilde{\gamma}$.
  
A metric space $(Z, d_Z)$ is said to be \emph{rectifiably connected} if every pair of points $x, y \in Z$ 
can be connected with a rectifiable path. The space $(Z, d_Z)$ is said to be \emph{($M$-)quasiconvex} 
if every pair of points $x,y \in Z$ can be connected with a path $\gamma$ of length at most $Md_Z(x,y)$,
where $M \ge 1$ is a fixed constant.
Moreover, for each pair of points $x,y\in Z$ in a rectifiably connected space $(Z, d_Z)$, 
the \emph{inner length metric} $\delta_Z$ on $Z$ is defined by
\[
	\delta_Z(x,y)=\inf_\gamma \len_{d_Z}(\gamma)
\]
with the infimum taken over all rectifiable curves $\gamma$ in $Z$ with end points $x,y$. The space $(Z, d_Z)$ is then said to be a \emph{length space} if $d_Z = \delta_Z$.
Note that $(Z, d_Z)$ being a length space does not necessarily ensure the existence of a
rectifiable curve $\gamma$ from $x$ to $y$ with $\len_{d_Z}(\gamma) = d_Z(x, y)$ for all $x, y \in Z$; 
spaces with this stronger property are instead said to be \emph{geodesic spaces}. 

In a length space $(Z,d_Z)$, for each $\eps>0$ and each pair of points $x,y\in Z$ there is a curve
with end points $x,y$ in $Z$ with length at most $d_Z(x,y)+\eps$. We say that such curves are \emph{$\eps$-short}.

\begin{remark}\label{rem:metr_derivs_in_length_metrics}
If $(Z, d_Z)$ is a rectifiably connected space, then the inner length metric $\delta_Z$ is indeed a metric on $Z$, and the space $(Z, \delta_Z)$ is a length space. Moreover, if $(Z, d_Z)$ is a quasiconvex space, then the identity map $(Z, d_Z) \to (Z, \delta_Z)$ is bilipschitz. We recall that if $\gamma \colon [a,b] \to (Z, d_Z)$ is a path, then
	\begin{equation}\label{eq:length_in_inner_metric}
		\len_{d_Z}(\gamma) = \len_{\delta_Z}(\gamma).
	\end{equation}
	Indeed, $\len_{d_Z}(\gamma)\le \len_{\delta_Z}(\gamma)$ follows from the fact that $d_Z \le \delta_Z$, and the other direction is due to the computation
	\begin{align*}
		\len_{\delta_Z}(\gamma)&=\sup\bigg\lbrace \sum_{j=1}^n\delta_Z(\gamma(t_{j-1}),\gamma(t_j))\, :\, n\in\N, a=t_0<\cdots<t_n=b\bigg\rbrace\\
		&\le \sup\bigg\lbrace \sum_{j=1}^n\len_{d_Z}(\gamma\vert_{[t_{j-1},t_j]})\, :\, n\in\N, a=t_0<\cdots<t_n=b\bigg\rbrace\\
		&=\len_{d_Z}(\gamma).
	\end{align*}
	A notable consequence of 
\eqref{eq:met_deriv_and_length_function} and~\eqref{eq:length_in_inner_metric} is that, 
if $\gamma \colon [a,b] \to (Z, d_Z)$ is an absolutely continuous path, then $\gamma$ has the same metric 
derivative in both the metrics $d_Z$ and $\delta_Z$.
\end{remark}

If $(X, d_X)$ and $(Y, d_Y)$ are metric spaces and $\gamma:[a,b]\to X\times Y$ is a path, then we often decompose 
$\gamma$ into coordinate functions $\gamma = (\gamma_X, \gamma_Y)$, where $\gamma_X:[a,b]\to X$ and 
$\gamma_Y:[a,b]\to Y$ are paths. We note that $\gamma$ is rectifiable 
with respect to the $\ell^1$-metric $d$ on $X\times Y$
if and only if both $\gamma_X$ and 
$\gamma_Y$ are rectifiable, and similarly, $\gamma$ is absolutely continuous if and only if both $\gamma_X$ 
and $\gamma_Y$ are absolutely continuous. 
As absolute continuity of a curve is a bilipschitz invariant, and as each norm 
$\norm{\cdot}$ on $\R^2$ is bilipschitz to the $\ell^1$-norm there, it follows that the absolute continuity
of a curve $\gamma$, when $X\times Y$ is equipped with the metric $d_{\norm{\cdot}}$
corresponding to a coordinate-increasing norm $\norm{\cdot}$, is equivalent to its absolute continuity under the
$\ell^1$-metric.

We also point out that if both $X$ and $Y$ are $M$-quasiconvex, 
then $X \times Y$ is $M$-quasiconvex when equipped with the metric $d$ from \eqref{eq:l1-metric}.

\subsection{Gromov hyperbolic spaces}\label{gromovhyperbolic}

Gromov hyperbolicity, first proposed in the context of geometry of groups, embodies large-scale geometry
of hyperbolic spaces. 
See for example~\cite{BH99, BS07, BBI01, Kap01, Vai} for more details.

Let $(Z, d_Z)$ be a metric space. The \emph{Gromov product} $\gromp{x}{y}_w$ is defined for $x, y, w \in Z$ by
\[
	\gromp{x}{y}_w = \frac{1}{2} \left( d_Z(x,w) + d_Z(y,w) - d_Z(x,y) \right). 
\]
The space $(Z,d_Z)$ is said to be \emph{Gromov $\delta$-hyperbolic} for $\delta \ge 0$ if, for all $x, y, z, w \in Z$, we have
\begin{equation}\label{eq:Gromov_hyperbolicity}
	\gromp{x}{y}_w \ge \min \lbrace \gromp{x}{z}_w, \gromp{y}{z}_w \rbrace - \delta. 
\end{equation}
The classical example of Gromov hyperbolic spaces are the standard hyperbolic spaces $\H^n$, which are Gromov 
$\ln(3)$-hyperbolic for all $n \in \{2, 3, \dots\}$, see e.g.\ \cite[Proposition~4.3]{CDP}.

We recall that to prove Gromov hyperbolicity, it sufficies to verify \eqref{eq:Gromov_hyperbolicity} for a fixed $w$; see e.g.\ \cite[Proposition 1.2]{CDP}.

\begin{lemma}\label{lem:Gromov_hyp_basepoint_invariance}
	Let $(Z, d_Z)$ be a metric space and let $\delta > 0$. If there exists a point
$w \in Z$ such that \eqref{eq:Gromov_hyperbolicity} holds for all $x, y, z \in Z$, then $Z$ is Gromov $2\delta$-hyperbolic.
\end{lemma}

We also recall that Gromov hyperbolicity is a bilipschitz invariant on length spaces. The following statement is a special case of the 
stronger result \cite[Theorem~3.18]{Vai}, where the wider class of quasi-isometric maps is considered.
A loose claim that Gromov hyperbolicity
	is a bilipschitz invariant can be found in~\cite[Page~111]{Kap01}, but if at least one of the metric spaces in question is not a length
	space, this claim fails; see~\cite[Remark 3.19]{Vai}.

\begin{lemma}\label{lem:Gromov_hyp_bilipschitz}
	Let $(Z, d_Z)$ and $(W, d_W)$ be length spaces, and let $f \colon Z \to W$ be $L$-bilipschitz. If $W$ is Gromov $\delta$-hyperbolic, 
	then $Z$ is Gromov $\delta'$-hyperbolic with $\delta' = \delta'(\delta, L)$.
\end{lemma}

 Suppose then that $(Z,d_Z)$ is Gromov $\delta$-hyperbolic, and fix a base point $x_0 \in Z$. A sequence  $(x_i)$ in $Z$ is 
said to be a
 \emph{Gromov sequence} if $\lim_{i,j \to \infty} \gromp{x_i}{x_j}_{x_0} = \infty$. 
 The \emph{Gromov boundary} $\partial_{G} Z$ of $Z$ consists of all Gromov sequences of $Z$, under the 
 equivalence relation that $(x_i) \sim (y_i)$ if $\lim_{i \to \infty} \gromp{x_i}{y_i}_{x_0} = \infty$. 
 We extend the Gromov product to $\partial_G Z$ by setting
\begin{equation}\label{eq:GromProd-bdy}
\gromp{\overline{x}}{\overline{y}}_{x_0} = \inf 
\left\{ \liminf_{i,j \to \infty} \gromp{x_i}{y_j}_{x_0}
: (x_i) \in \overline{x}, (y_i) \in \overline{y}
\right\}
\end{equation}
for all $\overline{x}, \overline{y} \in \partial_{G} Z$. For every $\eps \in (0, \infty)$, we obtain a 
function $\tilde{d}_\eps \colon \partial_G Z \times \partial_G Z \to [0, 1]$ by setting
\[
\tilde{d}_\eps(\overline{x}, \overline{y}) = e^{-\eps \gromp{\overline{x}}{\overline{y}}_{x_0}}
\]
for all $\overline{x}, \overline{y} \in \partial_{G} Z$, with the convention that $e^{-\infty} = 0$. 
Using this, for every $\eps \in (0, \infty)$, we obtain a pseudometric $d_\eps$ on $\partial_G Z$ by
\begin{equation}\label{eq:gromov_boundary_metric}
	d_\eps(\overline{x}, \overline{y}) =
	\inf \left\{
	\sum_{i=1}^k \tilde{d}_\eps(\overline{x}_{i-1}, \overline{x}_i)
	: \overline{x}_0, \overline{x}_1, \dots, \overline{x}_k \in \partial_G Z,
	\overline{x}_0 = \overline{x}, \overline{x}_k = \overline{y}
	\right\}
\end{equation}
for all $\overline{x}, \overline{y} \in \partial_G Z$. By \cite[Proposition 5.16]{Vai}, if $\eps \le \min \{1, 1/(5\delta)\}$, then 
$(\partial_G Z, d_\eps)$ is a metric space and
\begin{equation}\label{eq:boundary_metric_premetric_comparison}
	2^{-1}\tilde{d}_\eps(\overline{x}, \overline{y}) \le d_\eps(\overline{x}, \overline{y}) 
	\le \tilde{d}_\eps(\overline{x}, \overline{y})
\end{equation}
for all $\overline{x}, \overline{y} \in \partial_G Z$. 

We recall that bilipschitz maps extend to the Gromov boundary as quasisymmetric maps. Here, we recall that 
if $(Z, d_Z)$ and $(W, d_W)$ are metric spaces and $\eta \colon [0, \infty) \to [0, \infty)$ is a homeomorphism, then a 
map $f \colon Z \to W$ is called \emph{$\eta$-quasisymmetric} if, for all triples of distinct points $x, y, z \in Z$, we have
\[
	\frac{d_W(f(x), f(y))}{d_W(f(x), f(z))} \le 
	\eta \left( \frac{d_Z(x,y)}{d_Z(x,z)}\right).
\]
The following statement follows from e.g.\ \cite[Theorem 5.35]{Vai}. 

\begin{lemma}\label{lem:gromov_hyp_QS_extension}
	Let $\delta \ge 0$, let $\eps \in (0, \min \{1, 1/(5\delta)\})$, let $(Z, d_Z)$ and $(W, d_W)$ be Gromov $\delta$-hyperbolic 
length spaces, and let $f \colon Z \to W$ be a bijective $L$-bilipschitz map. Then there exists an $\eta$-quasisymmetric 
homeomorphism $\partial f \colon (\partial_G Z, (d_Z)_\eps) \to (\partial_G W, (d_W)_\eps)$, 
where $\eta$ depends only on $\delta$ and $L$. 
\end{lemma}

We note that the requirement in Lemmas \ref{lem:Gromov_hyp_bilipschitz} and \ref{lem:gromov_hyp_QS_extension}  
that $Z$ must be a length space is the main reason 
why we consider general product metrics of the form \eqref{eq:general_product_metric} in this article. That is, for 
Theorems \ref{thm:solid_hyp_filling_is_PI} and \ref{thm:warped_product_PI}, it is relatively easy to see via bilipschitz 
invariance that a proof under one choice of product metric would imply the same result for all choices of product metric. 
However, for the Gromov hyperbolicity result of Theorem \ref{thm:solid_filling_is_Gromov_hyp}, in order to use 
Lemmas \ref{lem:Gromov_hyp_bilipschitz} and \ref{lem:gromov_hyp_QS_extension} to execute such an invariance argument, 
we have to first
show that $\H_{\alpha}(Y)$ and $\H^{\circ}_{\alpha}(Y)$ are length spaces for all choices of product metrics.
Apart from demonstrating that these are length spaces, it suffices to consider the more convenient $\ell^1$-metric.

\subsection{Function spaces and Poincar\'e inequalities}

Let $(Z,d_Z,\mu_Z)$ be a metric measure space
as in Standing Assumptions \ref{stassm:extra_assumptions}. 
Following the work of Heinonen and Koskela~\cite{HK}, 
given a measurable function $u:Z\to[-\infty,\infty]$, we say that a non-negative Borel function
$g:Z\to[0,\infty]$ is an \emph{upper gradient} of $u$ 
if for each non-constant absolutely continuous curve $\gamma:[a,b]\to Z$
we have
\[
	|u(\gamma(b))-u(\gamma(a))|	\le \int_\gamma g\, ds.
\]

We now fix $p \in [1, \infty)$.
We say that a measurable function $u:Z\to[-\infty,\infty]$ is \emph{integrable over balls} if 
$u \in L^1(B_Z, \mu_Z)$ for every ball $B_Z \subset Z$. 
The \emph{Dirichlet-Sobolev space} $D^{1,p}(Z,d_Z,\mu_Z)$ consists of measurable functions 
$u:Z\to[-\infty,\infty]$ which are integrable over balls  and have an upper gradient $g\in L^p(Z,\mu_Z)$. 
The \emph{Newtonian Sobolev space} $N^{1,p}(Z, d_Z,\mu_Z)$ consists of functions 
$u\in D^{1,p}(Z,d_Z,\mu_Z)$ for which $\int_Z|u|^p\, d\mu_Z$ is finite. We refer the interested reader to~\cite{BBbook, HKST} for more on these function spaces.

Given a function $u:Z\to\R$, we define the \emph{local Lipschitz constant function} $\Lip_{d_Z}u$ by 
\[
\Lip_{d_Z}u(z):=\limsup_{x\to z}\, \frac{|u(x)-u(z)|}{d_Z(x,z)}=\lim_{r\to 0^+}\, \sup_{x\in B_Z(z,r)}\,\frac{|u(x)-u(z)|}{d_Z(x,z)}.
\]
A function $u:Z\to\R$ is said to be \emph{locally Lipschitz continuous} if for each $z\in Z$ there is some $r_z>0$ and $C_z>0$
such that $|u(x)-u(y)|\le C_z\, d_Z(x,y)$ for each pair of points $x,y\in B_Z(z,r_z)$. When $u$ is locally Lipschitz continuous,
from~\cite[Lemma~6.2.5, Lemma~6.2.6]{HKST}
we know that $\Lip_{d_Z}u$ is an upper gradient of $u$ in $Z$.

\begin{remark}
Note that our definition of $\Lip_{d_Z}u$ is different from the one given in~\cite{HKST}, 
where the upper local Lipschitz constant of $u$ is defined to be the function $g_u$ given by
\[
g_u(z):=\limsup_{r\to 0^+}\, \sup_{y\in B_Z(z,r)}\, \frac{|u(z)-u(y)|}{r}.
\]
However,  we have that $\Lip_{d_Z}u=g_u$. Indeed, it is immediate that $0\le g_u\le \Lip_{d_Z}u$. Suppose now that $\Lip_{d_Z}u(z)>0$.
Then for each $\eps>0$ we can find $0<r_\eps<\eps$ and $y_\eps\in B_Z(z,r_\eps)$ such that
\[
\sup_{y\in B_Z(z,r_\eps)}\frac{|u(y)-u(z)|}{d_Z(y,z)}<\Lip_{d_Z}u(z)+\eps, \qquad \Lip_{d_Z}u(z)-\eps<\frac{|u(y_\eps)-u(z)|}{d_Z(y_\eps,z)}.
\]
It follows that for every $\eta>0$,
\[
\sup_{y\in B_Z(z,(1+\eta)d_Z(y_\eps,z))}\, \frac{|u(y)-u(z)|}{(1+\eta)d_Z(y_\eps,z)}\ge \frac{|u(y_\eps)-u(z)|}{(1+\eta)d_Z(y_\eps,z)}
>\frac{\Lip_{d_Z}u(z)-\eps}{1+\eta}.
\]
Therefore we see that 
\[
g_u(z)\ge \frac{ \Lip_{d_Z}u(z)-\eps}{1+\eta},
\]
the above inequality holding for all $\eps>0$ and $\eta>0$. Letting $\eps\to 0^+$ and $\eta\to 0^+$, we now see that $g_u\ge \Lip_{d_Z}u$ as
well, establishing the equality $g_u=\Lip_{d_Z}u$.
\end{remark}

Recall that, in Definition~\ref{defn:Sobolev-Poincare}, we defined a metric measure space $(Z, d_Z, \mu_Z)$ to be a 
$(p,q)$-Sobolev--Poincar\'e space (with constant $C$) if, 
for all $u \in D^{1,q}(Z)$ and all upper gradients $g \in L^q(Z)$ of $u$, there exists a constant $c_u \in \R$
for which
\begin{equation}\label{enum:infinite_PI_exists} 
	\norm{u-c_u}_{L^p(Z,\mu_{Z})} \le C \norm{g}_{L^q(Z,\mu_{Z})}.
\end{equation}
In the following lemma, we prove an alternate characterization of this property which also shows how to compute $c_u$.

\begin{lemma}\label{lem:inf_meas_PI_characterization}
	Let  $(Z, d_{Z}, \mu_{Z})$ be a metric measure space 
satisfying Standing Assumptions~\ref{stassm:extra_assumptions}, and suppose that
$\mu_{Z}(Z) = \infty$. Let $p, q \in [1, \infty)$, and 
fix a point $x_0 \in Z$. Then $Z$ is a $(p,q)$-Sobolev--Poincar\'e space if and only if, for 
all $u \in D^{1,q}(Z)$, and all upper gradients $g \in L^q(Z)$ of $u$, the limit $u_Z := \lim_{r \to \infty} u_{B_Z(x_0, r)}$ exists and satisfies
	\begin{equation}\label{enum:infinite_PI_explicit} 
		\norm{u - u_Z}_{L^p(Z,\mu_{Z})} \le C \norm{g}_{L^q(Z,\mu_{Z})},
	\end{equation}
where $C$ is independent of $u$, and $g$. Moreover, if $Z$ is a $(p,q)$-Sobolev--Poincar\'e space, 
then $u_Z$ is independent on the choice of $x_0$.
\end{lemma}

\begin{proof}	
	It is clear that \eqref{enum:infinite_PI_explicit} implies \eqref{enum:infinite_PI_exists} with $c_u = u_Z$. On the other hand, 
	suppose that \eqref{enum:infinite_PI_exists} holds, and let $u \in D^{1,q}(Z)$ with an upper gradient $g \in L^q(Z)$. 
	Fix any $x_0 \in Z$. Then by H\"older's inequality and the assumed condition~\eqref{enum:infinite_PI_exists},
	\begin{multline*}
		\abs{u_{B_Z(x_0, r)} - c_u} 
		\le  \frac{1}{\mu(B_Z(x_0, r))} \int_{B_Z(x_0, r)} |u - c_u| \, d\mu_{Z}\\
		\le \frac{\norm{u - c_u}_{L^p(B_Z(x_0, r), \mu_{Z})}}{\mu(B_Z(x_0, r))^{1/p}}
		\le \frac{\norm{u - c_u}_{L^p(Z,\mu_{Z})}}{\mu(B_Z(x_0, r))^{1/p}}
		\le \frac{C\norm{g}_{L^q(Z,\mu_{Z})}}{\mu(B_Z(x_0, r))^{1/p}}.
	\end{multline*}
	Since $\mu_{Z}(Z) = \infty$, we get that $\lim_{r \to \infty} u_{B_Z(x_0, r)} = c_u$, i.e.\ that $u_Z = c_u$, proving \eqref{enum:infinite_PI_explicit}. 
	Since $c_u$ is also independent of $x_0$, we furthermore get that $u_Z$ is independent of $x_0$.
\end{proof}

We also record the following immediate consequence of a space $Z$ being a $(p,p)$-Sobolev--Poincar\'e space.

\begin{lemma}
Let $p \in [1, \infty)$. If $Z$ is a $(p,p)$-Sobolev--Poincar\'e space, then $D^{1,p}(Z) = N^{1,p}(Z) + \R$.
\end{lemma}

\begin{remark}[Upper gradients and induced length metrics]\label{rem:QC2length}
	We note that, if the metric measure space $(Z,d_Z,\mu_Z)$ is such that $(Z, d_Z)$ is rectifiably connected, then for every measurable $u\colon Z\to[-\infty,\infty]$ and every Borel $g:Z\to[0,\infty]$, $g$ is an upper gradient of $u$ in $(Z, d_Z, \mu_Z)$ if and only if $g$ is an upper gradient of $u$ in the inner metric space $(Z, \delta_Z, \mu_Z)$. This is because, as noted in 
Remark~\ref{rem:metr_derivs_in_length_metrics}, the metric derivative $\abs{\gamma'}$ of an absolutely continuous path 
$\gamma$ is the same in both the $d_Z$-metric and the $\delta_Z$-metric.
\end{remark}

\begin{remark}\label{rem:Poincare_bilipschitz_invariance}
Suppose that $(Z, d_Z, \mu_Z)$ is an $M$-quasiconvex metric measure space, and let $p, q \in [1, \infty)$. If $Z$ is a $(p,q)$-Sobolev--Poincar\'e space as in Definition \ref{defn:Sobolev-Poincare}, then it is immediate 
from Remark~\ref{rem:QC2length} that $Z$ remains a $(p,q)$-Sobolev--Poincar\'e space with the same 
constant if the metric $d_Z$ is replaced with the inner length metric $\delta_Z$. 
	
Similarly, if $Z$ is an $\infty$-weak local $(p,q)$-Poincar\'e space as in Definition \ref{defn:weakPI}, then $Z$ remains an $\infty$-weak local $(p,q)$-Poincar\'e space if $d_Z$ is replaced by $\delta_Z$; however, 
this requires justification. Indeed, let $x \in Z$ and let $r_x$, $C_x$ be as in Definition~\ref{defn:weakPI}. 
Thus, if $u$ is integrable on balls in $(Z,d_Z,\mu_Z)$ and $c\in\R$, we have
\begin{align*}
&\int_{B_{\delta_Z}(x,r_x/M)}|u-u_{B_{\delta_Z}(x,r_x/M)}|^p\, d\mu_Z\\
&\le\int_{B_{\delta_Z}(x,r_x/M)}\, \frac{1}{\mu_Z(B_{\delta_Z}(x,r_x/M))}\int_{B_{\delta_Z}(x,r_x/M)}|u(y)-u(z)|^p\, d\mu_Z(z)\, d\mu_Z(y)\\
&\le \frac{2^{p-1}}{\mu_Z(B_{\delta_Z}(x,r_x/M))}\, \int_{B_{\delta_Z}(x,r_x/M)}\, \int_{B_{\delta_Z}(x,r_x/M)}\, \left(|u(y)-c|^p+|u(z)-c|^p\right)\, d\mu_Z(z)\, d\mu_Z(y)\\
&=2^{p}\, \int_{B_{\delta_Z}(x,r_x/M)}\, |u-c|^p\, d\mu_Z.
\end{align*}
		It follows that 
		\[
		\int_{B_{\delta_Z}(x,r_x/M)}|u-u_{B_{\delta_Z}(x,r_x/M)}|^p\, d\mu_Z\le 2^{p}\, \inf_{c\in\R}\, \int_{B_{\delta_Z}(x,r_x/M)}\, |u-c|^p\, d\mu_Z.
		\]
Now, combining this with the inclusion $B_{\delta_Z}(x,r_x/M)\subset B_Z(x,r_x)$, followed by the 
$\infty$-weak local $(p,q)$-Poincar\'e inequality of $(Z,d_Z)$, gives
\begin{align*}
\int_{B_{\delta_Z}(x,r_x/M)}|u-u_{B_{\delta_Z}(x,r_x/M)}|^p\, d\mu_Z&\le 2^p\int_{B_{\delta_Z}(x,r_x/M)}|u-u_{B_Z(x,r_x)}|\, d\mu_Z\\
  &\le 2^p\int_{B_Z(x,r_x)}|u-u_{B_Z(x,r_x)}|\, d\mu_Z\\
  &\le 2^p\, C_x^p\, \left(\int_Z g^q\, d\mu_Z\right)^{p/q},
\end{align*}
which
yields the validity of the 
		$\infty$-weak Sobolev-Poincar\'e inequality on $(Z,\delta_Z,\mu_Z)$. 
\end{remark}

\subsection{Other miscellaneous preliminaries}

We recall a simple lemma from metric topology which generalizes the fact that continuous functions are uniformly continuous on 
compact sets.
The proof is straightforward, but we provide it for convenience.  

\begin{lemma}\label{lem:improved_unif_cont_on_compacts}
	Let $(Z, d_Z)$ and $(W, d_W)$ be metric spaces, let $K \subset Z$ be compact, and let $f \colon Z \to W$ be continuous. 
	Then for every $\eps > 0$, there exists  $\delta > 0$ such that for all $x \in Z$ and $z \in K$, if $d_Z(x, z) < \delta$, then 
	$d_W(f(x), f(z)) < \eps$.
\end{lemma}

\begin{proof}
	As $f$ is continuous, for each $z \in K$ there exists $\delta_z > 0$ such that 
	$B_Z(z,\delta_z)\subset f^{-1}(B_W(f(z),\eps/2))$.	
	By compactness of $K$, we find a finite collection $z_1, \dots, z_n \in K$ such that 
	the balls 
	$B_Z(z_i, \delta_{z_i}/2)$ cover $K$. We let $\delta=\min\{\delta_{z_i}/2\, :\, i=1,\cdots, n\}$.
	Now, supposing that $z \in K$ and $x \in B_Z(z,\delta)$, there exists an index $i \in \{1, \dots, n\}$ 
	such that $d_Z(z, z_i) < \delta_{z_i}/2$. Since $\delta \le \delta_{z_i}/2$, the triangle inequality also yields that 
	$d_Z(x, z_i) < \delta_{z_i}$. Thus, $d_W(f(x), f(z)) \le d_W(f(x), f(z_i)) + d_W(f(z), f(z_i)) < \eps/2 + \eps/2 = \eps$, completing the proof.
\end{proof}

\section{Warped products of metric spaces}\label{sect:warped_products}

In this section, we discuss the construction of warped products in the metric setting. 
Let $(X,d_X), (Y,d_Y)$ be  length spaces, and let $\varphi \colon X \to [0, \infty)$ be continuous. 
We fix a unitary coordinate-increasing norm $\norm{\cdot}$ on $\R^2$. For every absolutely continuous path 
$\gamma = (\gamma_X, \gamma_Y) \colon [a,b] \to X \times Y$, we denote
\begin{equation}\label{eq:phi_length_def}
	\ell_\varphi(\gamma) := \int_a^b \norm{\left(\abs{\gamma_X'(t)}, \varphi(\gamma_X(t)) \abs{\gamma_Y'(t)}\right)} \, dt.
\end{equation}
Thanks to the unitary property of $\norm{\cdot}$, we have that
\[
\ell_\pip(\gamma)\le \int_a^b|\gamma^\prime_X(t)|+\pip(\gamma_X(t))\, |\gamma^\prime_Y(t)|\, dt,
\]
that is, the warped length is largest when the norm used is the $\ell^1$-norm.

Next, we define a pseudometric $d_\varphi$ on the product space $X \times Y$, given by 
\[
	d_\varphi((x,y), (x', y')) := \inf_{\gamma} \ell_{\varphi}(\gamma),
\]
where the infimum is over absolutely continuous paths $\gamma:[a,b]\to X\times Y$ with $\gamma(a)=(x,y)$ and 
$\gamma(b)=(x',y')$. The \emph{warped product} $X \times_\varphi Y$ is the quotient space
\[
	X \times_{\varphi} Y := X \times Y / \sim_{\varphi},
\]
where $(x, y) \sim_{\varphi} (x', y')$ if $d_{\varphi}((x,y), (x', y')) = 0$. We equip the warped product $X \times_\varphi Y$ 
with the quotient pseudometric induced by $d_{\varphi}$, which we, in a mild abuse of notation, also denote $d_{\varphi}$;
see the discussion in Subsection~\ref{sub:pseudobable}.
Moreover, we use $P_\varphi$ to denote the projection $X \times Y \to X \times_{\varphi} Y$ which takes $(x,y) \in X \times Y$ to its $\sim_{\varphi}$-equivalence class.

Next, we verify the following.

\begin{lemma}\label{lem:warped_product_metric}
	Let $(X, d_X)$ and $(Y, d_Y)$ be length spaces, and let 
	$\varphi \colon X \to [0, \infty)$ be continuous. Then $d_\varphi$ is 
	a pseudometric on $X \times Y$, and consequently $d_\varphi$ is a metric on $X \times_\varphi Y$. 
	Moreover, if additionally $\varphi(x) > 0$ for all $x \in X$, then $d_\varphi$ is a metric on $X \times Y$.
\end{lemma}

In particular, if $(X,d_X), (Y,d_Y)$ are length spaces and $\varphi \colon X \to (0, \infty)$ is a continuous positive-valued function, then by Lemma \ref{lem:warped_product_metric}, we may interpret $X \times_{\varphi} Y$ to be the product space $X \times Y$ equipped with the metric $d_\varphi$, as equivalence classes of $\sim_{\varphi}$ are singletons in this case. If, however, $\varphi$ assumes zero values, then some subsets of $X \times Y$ may be collapsed to points in $X \times_\varphi Y$.

Before beginning the proof of Lemma \ref{lem:warped_product_metric}, 
we establish some useful estimates for $d_\pip$.
Note that the estimates below induce corresponding estimates on pairs of points in $X\times_\pip Y$.

\begin{lemma}\label{lem:d_varphi_bounds}
	 Let $(X, d_X)$, $(Y, d_Y)$ be length spaces, let $(x,y), (x', y') \in X \times Y$, and let 
	$\varphi \colon X \to [0, \infty)$ be continuous. Then the following estimates hold.
	\begin{enumerate}
		\item For each $x_0 \in X$, we have 
		\begin{equation}\label{eq:d_varphi_upper_bound}
			d_\varphi((x,y), (x', y')) \le d_X(x, x_0) + d_X(x', x_0) + \varphi(x_0) d_Y(y, y').
		\end{equation}
		\item We have
		\begin{equation}\label{eq:d_varphi_x_lower_bound}
			d_\varphi((x,y), (x', y')) \ge d_X(x, x').
		\end{equation}
		\item For each $\eps \in (0, 1)$, there exists a function $r_\eps \colon X \to (0, \infty)$ such that, if $x_0 \in X$ satisfies $d_X(x, x_0) < r_\eps(x_0)$, then  
		\begin{equation}\label{eq:d_varphi_y_lower_bound}
			\min \left\{ r_\eps(x_0), (1-\eps)\,\varphi(x_0) d_Y(y, y') \right\} \le 
			d_\varphi((x,y), (x', y')).
		\end{equation}
	\end{enumerate}
\end{lemma}

\begin{proof}
 Let $x_0 \in X$ and $\eps > 0$. Recall that $X$ and $Y$ are length spaces. 
We select an absolutely continuous path $\widehat{\alpha}_1 \colon [0, 1] \to X$ from $x$ to $x_0$ 
with $\len_{d_X}(\widehat{\alpha}_1) < d_X(x, x_0) + \eps$, and let $\alpha_1\colon[0,1]\to X\times Y$ be given
by $\alpha_1(t)=(\widehat{\alpha}_1(t),y)$. Similarly, we select an absolutely continuous path $\widehat{\alpha}_2 \colon [2,3] \to X$ 
from $x_0$ to $x'$ with $\len_{d_X}(\widehat{\alpha}_2) < d_X(x_0, x') + \eps$,
and let $\alpha_2\colon[2,3]\to X\times Y$ be given by $\alpha_2(t)=(\widehat{\alpha}_2(t),y')$. Finally, 
we choose an absolutely continuous path
$\widehat{\beta} \colon [1, 2] \to Y$ from $y$ to $y'$ with $\len_{d_Y}(\widehat{\beta}) < d_Y(y, y') + \eps$,
and let $\beta\colon[1,2]\to X\times Y$ be given by $\beta(t)=(x_0,\widehat{\beta}(t))$. Let $\gamma\colon[0,3]\to X\times Y$
denote the concatenation of the three paths $\alpha_1,\beta$, and $\alpha_2$, we see using the
unitary property of $\norm{\cdot}$ that
	\[
	d_\varphi((x,y), (x', y')) \le \ell_\varphi(\gamma) \le d_X(x, x_0) + d_X(x', x_0) + \varphi(x_0) d_Y(y, y') + 3\eps.
	\]
	Letting $\eps \to 0$ proves \eqref{eq:d_varphi_upper_bound}.
	
	Now, for the duration of the rest of the proof, let $\gamma = (\gamma_X, \gamma_Y) \colon [a,b] \to X \times Y$ 
	be an absolutely continuous path from $(x,y)$ to $(x', y')$; the existence of such a $\gamma$ is guaranteed 
	since $X$ and $Y$ are length spaces. Since $\norm{\cdot}$ is unitary and coordinate-increasing, 
	\eqref{eq:d_varphi_x_lower_bound} follows from the estimate
	\[
	\ell_\varphi(\gamma) \ge \int_a^b \smallabs{\gamma'_X(t)} \, dt = \len_{d_X}(\gamma_X) \ge d_X(x, x').
	\]
	
	It remains to prove \eqref{eq:d_varphi_y_lower_bound}. For this, let $x_0 \in X$ and $0<\eps<1$. 
	If $\varphi(x_0) = 0$, then the claim 
	is trivially true for any choice of $r_\eps(x_0)$. 
	If on the other hand $\varphi(x_0) > 0$, we use the continuity of $\varphi$ to 
	select $r_\eps(x_0)>0$ so that 
	$\varphi>(1-\eps)\varphi(x_0)$ 
	on $B_X(x_0, 2 r_\eps(x_0))$. 
	Now, suppose that $d_X(x, x_0) < r_\eps(x_0)$.
	If the trajectory of $\gamma_X$ stays within $B_X(x_0, 2 r_\eps(x_0))$, we then have by the  
	unitary property and the coordinate-increasing property of $\norm{\cdot}$ that 
\[
(1-\eps)^{-1}\, \ell_{\varphi}(\gamma) \ge \varphi(x_0) \len_{d_Y}(\gamma_Y) \ge \varphi(x_0) d_Y(y, y').
\] 
	Otherwise, $\gamma_X$ must exit $B_X(x_0, 2 r_\eps(x_0))$, in which case 
	since its starting point $x$ is in $B_X(x_0, r_\eps(x_0))$, we may similarly estimate that 
	$\ell_{\varphi}(\gamma) \ge \len_{d_X}(\gamma_X) \ge r_\eps(x_0)$. Thus, we have 
	\[
	\ell_\varphi(\gamma) \ge \min(r_\eps(x_0), (1-\eps) \varphi(x_0) d_Y(y, y')),
	\]
	and infimizing over $\gamma$ yields \eqref{eq:d_varphi_y_lower_bound}.
\end{proof}

\begin{proof}[Proof of Lemma \ref{lem:warped_product_metric}]
	The only obstruction to $d_\varphi$ being a pseudometric on $X \times Y$ is if $d_\varphi((x,y), (x',y')) = \infty$ for some 
	$(x, y), (x', y') \in X \times Y$.
	However, this is impossible due to~\eqref{eq:d_varphi_upper_bound} of Lemma~\ref{lem:d_varphi_bounds}.
	Moreover, if $\varphi > 0$ on $X$, then for
$(x, y), (x', y') \in X \times Y$ with $(x, y) \ne (x', y')$, the claim follows 
if we can show that $d_\varphi((x,y), (x', y')) \ne 0$. If $x \ne x'$, then this follows immediately 
from~\eqref{eq:d_varphi_x_lower_bound} of Lemma~\ref{lem:d_varphi_bounds}, and if $y \ne y'$, then 
since $\varphi$ is positive-valued, this follows immediately from~\eqref{eq:d_varphi_y_lower_bound} of 
Lemma \ref{lem:d_varphi_bounds}. Thus, the proof is complete.
\end{proof}

Lemma \ref{lem:d_varphi_bounds} also has the following corollary, which clarifies the structure of $\sim_\varphi$-equivalence 
classes in $X \times Y$.

\begin{corollary}\label{cor:d_varphi_equiv_classes}
	Let $(X, d_X)$, $(Y, d_Y)$ be length spaces, let $(x,y) \in X \times Y$, and let 
	$\varphi \colon X \to [0, \infty)$ be continuous. If $\varphi(x) > 0$, then the $\sim_\varphi$-equivalence class 
of $(x,y)$ is a singleton $[(x,y)] = \{(x,y)\}$. On the other hand, if $\varphi(x) = 0$, then $[(x,y)] = \{x\} \times Y$.
\end{corollary}
\begin{proof}
Let $(x,y), (x',y') \in X \times Y$. If $x \ne x'$, then Lemma~\ref{lem:d_varphi_bounds} \eqref{eq:d_varphi_x_lower_bound} 
yields that $d_\varphi((x,y), (x', y')) > 0$. If $x = x'$ and $\varphi(x) > 0$, then 
Lemma~\ref{lem:d_varphi_bounds} \eqref{eq:d_varphi_y_lower_bound} with $x_0 = x$ 
yields that $d_\varphi((x,y), (x', y')) > 0$. On the other hand, if $x = x'$ and $\varphi(x) = 0$, then 
Lemma~\ref{lem:d_varphi_bounds} \eqref{eq:d_varphi_upper_bound} with $x_0 = x$ yields that $d_\varphi((x,y), (x', y')) = 0$.
\end{proof}

\subsection{Properties of warped products}

A key property of warped product spaces is that $\ell_\varphi$ computes the (pseudo)metric length of an absolutely 
continuous path with respect to $d_\varphi$. The case of the following 
proposition where $\varphi > 0$ and $\norm{\cdot}$ is the $\ell^2$-norm is stated in \cite[p.\ 4731]{Ch}, though without proof.

\begin{proposition}\label{prop:lengths_match_warped}
	Let $(X,d_X), (Y,d_Y)$ be length spaces, and let $\varphi \colon X \to [0, \infty)$ be continuous. 
If $\gamma \colon [a,b] \to X \times Y$ 
is an absolutely continuous path, then
	\[
	\len_{d_\varphi}(\gamma) = \ell_\varphi(\gamma).
	\]
Consequently, if $g$ is a non-negative Borel measurable function on $X\times_\pip Y$, and 
$\gamma:[a,b]\to X\times Y$ is an absolutely continuous path, then
\begin{equation}\label{eq:UG_in_skew_prod}
\int_{P_\pip\circ\gamma}g\, ds
=\int_a^b\, g(P_\pip(\gamma(t)))\, \norm{(|\gamma_X^\prime(t)|, \pip(\gamma_X(t))\, |\gamma_Y^\prime(t)|)}\, dt.
\end{equation}
\end{proposition}

The proof of Proposition~\ref{prop:lengths_match_warped} is long, and is thus
postponed to the next subsection. The remainder of this subsection will assume the 
validity of the above proposition. As the proof of the proposition in the next subsection will not
use any of the results discussed in the current subsection, there is no risk of a circular argument problem. 

Proposition~\ref{prop:lengths_match_warped} has the following 
corollary that is particularly important for our applications. 

\begin{corollary}\label{cor:warped_prod_length_space}
Let $(X,d_X), (Y,d_Y)$ be length spaces, and let $\varphi \colon X \to [0, \infty)$ be continuous. 
Then the metric space $(X \times_\varphi Y, d_\pip)$ is a length space.
\end{corollary}

\begin{proof}
	Let $(x_1,y_1), (x_2,y_2)\in X\times Y$. Then by the definition of $d_\pip$ and 
Proposition~\ref{prop:lengths_match_warped}, we have
	\[
	d_\pip((x_1,y_1), (x_2,y_2))=\inf_\gamma \ell_\pip(\gamma)=\inf_\gamma\, \len_{d_\pip}(\gamma),
	\]
	where the infimum is over all absolutely continuous curves in $X\times Y$ with end points $(x_1,y_1)$ and $(x_2,y_2)$.
	Thus, $(X\times Y, d_\pip)$ is a (pseudometric) length space. Since $P_\varphi$ preserves distances and lengths, it follows that $(X\times_\varphi Y, d_\pip)$ is also a (metric) length space.
\end{proof}

Next, we highlight a principle on warped product spaces that 
for all $(x,y), (x', y') \in X \times Y$, one can find a path $\gamma = (\gamma_X, \gamma_Y)$ 
of $d_\varphi$-shortest length from $(x,y)$ to $(x', y')$ for which $\gamma_Y$ is also a
path of $d_Y$-shortest length from $y$ to $y'$; see e.g.\ \cite[Theorem 3.1]{AB1}. 
In the setting of this paper, the two metric spaces are not necessarily  
geodesic spaces, so there might not even exist paths of 
shortest length. However, we provide an approximation-based version of this statement.

\begin{proposition}\label{prop:warped_product_optimal_paths}
	Let $(X, d_X)$, $(Y, d_Y)$ be length spaces, and let 
	$\varphi \colon X \to [0, \infty)$ be continuous. 
	Then for all pairs of points
	$(x,y), (x',y') \in X \times Y$ and every $\eps > 0$, there exist $a, b \in \R$ with 
	$a \le b$ and an absolutely continuous path 
	$\gamma = (\gamma_X, \gamma_Y) \colon [a,b] \to X \times Y$ from $(x,y)$ to $(x', y')$ for which
	\[
	\ell_\varphi(\gamma) \le d_\varphi((x,y), (x',y')) + \eps
	\quad \text{and} \quad
	\len_{d_Y}(\gamma_Y) \le d_Y(y, y') + \eps.
	\]
\end{proposition}

\begin{proof}
	By the definition of $d_\varphi$, we find an absolutely continuous path $\alpha \colon [a,b] \to X \times Y$ 
	from $(x,y)$ to $(x', y')$ for which $\ell_\varphi(\alpha) \le d_\varphi((x,y), (x',y')) + \eps$. Moreover, 
	we find a path $\beta_Y \colon [0, L] \to Y$ from $y$ to $y'$ for which 
	\[
	\len_{d_Y}(\beta_Y) \le \min\{\len_{d_Y}(\alpha_Y), d_Y(y, y') + \eps\}.
	\]
	By arc-length parametrizing $\beta_Y$, we may assume that $\beta_Y$ is 1-Lipschitz, 
	$\abs{\beta_Y'(t)} = 1$ for a.e.\ $t \in [0, L]$, and $L = \len_{d_Y}(\beta_Y)$.
	
	Let $s_{\alpha_Y} \colon [a,b] \to [0, \len_{d_Y}(\alpha_Y)]$ be the length function of $\alpha_Y$
	with respect to the metric $d_Y$. Then 
	$s_{\alpha_Y}$ is an absolutely continuous function with $s_{\alpha_Y}'(t) = \abs{\alpha_Y'(t)}$ 
	for a.e.\ $t \in [a,b]$. We set $b_0=\sup\{t\in[a,b]\, :\, s_{\alpha_Y}(t)\le \len_{d_Y}(\beta_Y)\}$.
	Note that the map $\beta_Y\circ s_{\alpha_Y}:[a,b_0]\to Y$
	is an absolutely continuous path that is a re-parametrization of $\beta_Y$. We extend
	this path to the path $\gamma_Y:[a,b]\to Y$ by setting 
	$\gamma_Y(t)= \beta_Y\circ s_{\alpha_Y}(t)$ when $t\le b_0$ and $\gamma_Y(t)=y'$ otherwise.
	Now we  define $\gamma:[a,b]\to X\times Y$ by $\gamma=(\alpha_X,\gamma_Y)$.
	
	As we have $\len_{d_Y}(\beta_Y) \le \len_{d_Y}(\alpha_Y)$, we observe 
	that $\gamma$ is a path from $(x,y)$ to $(x', y')$. 
	Since $\beta_Y$ is 1-Lipschitz, by the chain rule for Lipschitz functions and absolutely continuous functions, we have that
	$\gamma$ is absolutely continuous, $\abs{\gamma_Y'(t)} \le s_{\alpha_Y}'(t) = \abs{\alpha_Y'(t)}$ for a.e.\ $t \in [a,b_0]$ 
	and $\abs{\gamma_Y'(t)} = 0$ for a.e.\ $t \in [b_0,b]$.
	Thus, using the fact that $\norm{\cdot}$ is coordinate-increasing
	\begin{align*}
		\ell_\varphi(\gamma) 
		&= \int_a^b \norm{(\smallabs{\alpha_X'(t)}, 
			\varphi(\alpha_X(t)) \smallabs{\gamma_Y'(t)})} \, dt\\
		&\le \int_a^b \norm{(\smallabs{\alpha_X'(t)}, 
			\varphi(\alpha_X(t)) \smallabs{\alpha_Y'(t)})} \, dt\\
		&= \ell_{\varphi}(\alpha)
		\le d_\varphi((x,y), (x',y')) + \eps.
	\end{align*}
	Moreover, as $\gamma_Y$ is a re-parametrization of $\beta_Y$, it follows that
	$\len_{d_Y}(\gamma_Y)=\len_{d_Y}(\beta_Y)\le d_Y(y,y')+\eps$.
	Thus, $\gamma$ satisfies the desired conditions, completing the proof.
\end{proof}

Next, we highlight that for length spaces $(X,d_X), (Y,d_Y)$ and positive $\varphi$, the warped product $X \times_{\varphi} Y$ 
has the same topology as $X \times Y$. Here, recall that as stated in \eqref{eq:l1-metric}, we use $d$ to denote 
the $\ell^1$ product metric on $X \times Y$.

\begin{lemma}\label{lem:warped_product_topology}
	Let $(X, d_X)$ and $(Y, d_Y)$ be  length spaces, 
	and let $\varphi \colon X \to [0, \infty)$ be continuous. Then $P_\varphi \colon (X \times Y, d) 
	\to (X \times_\varphi Y, d_\varphi)$ is continuous, and $P_\varphi$ maps $d$-open subsets of 
	$(X \setminus \varphi^{-1}(\{0\})) \times Y$ to $d_\varphi$-open subsets of $X \times_\varphi Y$. 
	Note here that $(X \setminus \varphi^{-1}(\{0\})) \times Y$ is an open
	subset of $X\times Y$.
	
	In particular, if additionally $\varphi(x) > 0$ for all $x \in X$, then $P_\varphi$ is a homeomorphism, 
	i.e.\ the topology of $X \times_\varphi Y$ agrees with the usual product topology on $X\times Y$.
\end{lemma}

\begin{proof}
	Suppose first that $(x,y) \in X \times Y$, and that $(x_j, y_j)$ is a sequence in $X \times Y$ with 
$\lim_{j \to \infty} d((x,y), (x_j, y_j))=0$.  By \eqref{eq:d_varphi_upper_bound} of
Lemma~\ref{lem:d_varphi_bounds},  
we have 
\[
d_\varphi((x,y), (x_j, y_j)) \le d_X(x, x_j) + \varphi(x) d_Y(y, y_j), 
\]
and thus $\lim_{j \to \infty} d_\varphi ((x,y), (x_j, y_j))=0$. Hence, $P_\varphi$ is continuous.
	
	Since $P_\varphi$ is surjective, the openness claim reduces to proving the following statement: if 
$(x,y) \in (X \setminus \varphi^{-1}(\{0\})) \times Y$ and $(x_j, y_j)$ is a sequence of points in $X \times Y$ with 
$\lim_{j \to \infty} d_\varphi((x,y), (x_j, y_j))=0$, then $\lim_{j \to \infty} d((x,y), (x_j, y_j))=0$. This follows immediately from 
Lemma~\ref{lem:d_varphi_bounds}, where~\eqref{eq:d_varphi_x_lower_bound} 
yields $\lim_{j \to \infty} d_X(x, x_j)=0$ and~\eqref{eq:d_varphi_y_lower_bound} along with 
$x \notin \varphi^{-1}(\{0\})$ 
yields $\lim_{j \to \infty} d_Y(y, y_j)=0$. Finally, when $\varphi$ is everywhere 
positive, the claim that $P_\varphi$ is a homeomorphism follows from the previous parts, 
as well as the fact that $P_\varphi$ is bijective in this case by Lemma \ref{lem:warped_product_metric}.
\end{proof}

\subsection{Proof of Proposition \ref{prop:lengths_match_warped}}

We now proceed to prove Proposition~\ref{prop:lengths_match_warped}. One of the two required inequalities is simple.

\begin{lemma}\label{lem:lengths_match_part1}
Let $(X, d_X)$, $(Y, d_Y)$ be  length spaces, and let $\varphi \colon X \to [0, \infty)$ be continuous. 
If $\gamma \colon [a,b] \to X \times Y$ is an absolutely continuous path, then
	\[
	\len_{d_\varphi}(\gamma) \le \ell_\varphi(\gamma).
	\]
\end{lemma}
\begin{proof}
We select an arbitrary partition $a = t_0 < t_1 < \dots < t_k = b$ of $[a,b]$. It follows by definition of $d_\varphi$ that
	\[
	\sum_{i=1}^k d_\varphi(\gamma(t_{i-1}), \gamma(t_i)) \le 
	\sum_{i=1}^k \ell_\varphi(\gamma\vert_{[t_{i-1}, t_i]})
	= \ell_\varphi(\gamma).
	\]
Since $\len_{d_\varphi}(\gamma)$ is the supremum of the left hand side of the above estimate over all partitions, 
we obtain that $\len_{d_\varphi}(\gamma) \le \ell_\varphi(\gamma)$. 
\end{proof}

It now remains to prove the converse inequality.  
Our first step towards this is to prove Proposition \ref{prop:lengths_match_warped} in the case where $\varphi$ is positive and 
constant. In this case, the problem is far simpler, since we're able to easily write an explicit expression for $d_\varphi$.

\begin{lemma}\label{lem:lengths_match_constant_case}
	Let $(X, d_X)$, $(Y, d_Y)$ be  length spaces, 
	and let $\varphi \colon X \to (0, \infty)$ be constant, i.e.\ $\varphi \equiv A>0$. Then 
	\begin{equation}\label{eq:constant_warped_metric_formula}
		d_{\varphi}((x,y), (x',y')) = \smallnorm{(d_X(x,x'), Ad_{Y}(y, y'))}
	\end{equation}
	for all $(x,y), (x',y') \in X \times Y$.
Moreover, if $\gamma \colon [a,b] \to X \times Y$ is an absolutely continuous path, then
	\[
	\len_{d_\varphi}(\gamma) = \ell_\varphi(\gamma).
	\] 
\end{lemma}

\begin{proof}
Let $(x,y), (x',y') \in X \times Y$, and let $\eps > 0$. We may select paths $\alpha \colon [0, 1] \to X$ from $x$ to 
$x'$ and $\beta \colon [0, 1] \to Y$ from $y$ to $y'$ such that both $\alpha$ and $\beta$ have constant speed 
parametrizations, $\len_{d_X}(\alpha) \le d_X(x,x') + \eps$, and $\len_{d_Y}(\beta) \le d_Y(y, y') + \eps$. 
Now, both $\alpha$ and $\beta$ are absolutely continuous, 
so by letting $\gamma = (\alpha, \beta)$ and using the fact that $\norm{\cdot}$ is coordinate-increasing and unitary, we get
	\begin{multline*}
		d_{\varphi}((x,y), (x',y'))
		\le \ell_{\varphi}(\gamma) = \norm{(\len_{d_X}(\alpha), A\len_{d_Y}(\beta))}\\
		\le \norm{(d_X(x,x') + \eps, A(d_Y(y, y') + \eps))}
		\le \norm{(d_X(x,x'), Ad_Y(y, y'))}
		+ \eps\, (1+A).
	\end{multline*}
	Letting $\eps \to 0$ gives us that the left-hand side of~\eqref{eq:constant_warped_metric_formula}	
	is less than or equal to the right-hand side of that equation.
	
	Suppose next that $\gamma = (\gamma_X, \gamma_Y) \colon [a,b] \to X \times Y$ is an absolutely continuous 
path connecting $(x,y)$ and $(x',y')$. Then 
	\begin{align*}
		\ell_\varphi(\gamma) = \int_a^b \norm{(\abs{\gamma_X'(t)}, A\abs{\gamma_Y'(t)})} \, dt
		&\ge \norm{\int_a^b (\abs{\gamma_X'(t)}, A\abs{\gamma_Y'(t)}) \, dt} \\
		&= \norm{( \len_{d_X}(\gamma_X),A \len_{d_Y}(\gamma_Y))}.
	\end{align*} 
Taking the infimum over all absolutely continuous paths $\gamma$ from $(x,y)$ to $(x', y')$ and using the fact that $\norm{\cdot}$ is 
coordinate-increasing, we get
	\[
	d_{\varphi}((x,y), (x',y'))
	\ge \inf_{\gamma} \norm{(\len_{d_X}(\gamma_X), A\len_{d_Y}(\gamma_Y))}\\
	= \norm{(d_{X}(x,x'), Ad_{Y}(y, y'))},
	\]
	completing the proof of~\eqref{eq:constant_warped_metric_formula}.
	
	To prove the last claim, let $\gamma \colon [a,b] \to X \times Y$ be absolutely continuous.
Then it is absolutely continuous with respect to $d_\varphi$. 
Now, for a.e.\ $t \in [a,b]$, we can 
compute the metric speed of $\gamma$ with respect to $d_\varphi$ 
by using~\eqref{eq:constant_warped_metric_formula} 
to show that 
	\begin{align*}
		\abs{\gamma'(t)}_{d_\varphi} 
		&= \lim_{s \to t} \frac{d_\varphi(\gamma(s), \gamma(t))}{\abs{s-t}}\\
		&= \lim_{s \to t} \norm{\left(\frac{d_X(\gamma_X(s), \gamma_X(t))}{\abs{s-t}}, 
			A\frac{d_Y(\gamma_Y(s), \gamma_Y(t))}{\abs{s-t}}\right)}\\
		&= \norm{\left(\abs{\gamma_X'(t)}, A\abs{\gamma_Y'(t)}\right)}.
	\end{align*}
	Thus,
	\[
	\len_{d_\varphi}(\gamma) = \int_a^b \abs{\gamma'(t)}_{d_\varphi} \, dt = \ell_\varphi(\gamma).
	\]
\end{proof}

Next we prove Proposition \ref{prop:lengths_match_warped} in the case where the trajectory $\abs{\gamma_X}$ does not 
meet $\varphi^{-1}(\{0\})$. This step contains the majority of the proof of Proposition \ref{prop:lengths_match_warped}. 
The strategy we employ is inspired by the proof of~\cite[Proposition A.7]{BHK}, where Bonk, Heinonen, and Koskela 
show that the weighted length of a rectifiable curve agrees with its metric length in the corresponding weighted length metric.

\begin{lemma}\label{lem:lengths_match_warped_positive}
	Let $(X,d_X), (Y,d_Y)$ be length spaces, 
	and let $\varphi \colon X \to [0, \infty)$ be continuous. If $\gamma = (\gamma_X, \gamma_Y) \colon [a,b] \to X \times Y$ is 
	an absolutely continuous path with $\abs{\gamma_X} \cap \varphi^{-1}(\{0\}) = \emptyset$, then
	\[
	\len_{d_\varphi}(\gamma) = \ell_\varphi(\gamma).
	\]
\end{lemma}

\begin{proof}
	We first fix some notation. 
For $x_0 \in X$ with $\varphi(x_0) > 0$, we let $\delta_{\varphi(x_0)}$ be the metric given by
	\begin{align*}
		\delta_{\varphi(x_0)}((x,y), (x',y')) 
		&= \smallnorm{(d_{X}(x,x'), \varphi(x_0)d_{Y}(y, y'))}.
	\end{align*}
	As in Lemma \ref{lem:product_metrics}, $\delta_{\varphi(x_0)}$ is a metric on $X \times Y$.
As in~\eqref{eq:l1-metric}, $d$ denotes the $\ell^1$ product metric on $X \times Y$. 
	
By Lemma \ref{lem:lengths_match_part1}, it suffices to show that
$\len_{d_\varphi}(\gamma) \ge \ell_\varphi(\gamma)$. Since $\varphi$ is 
non-negative and continuous, and since $\abs{\gamma_X}$ is compact 
and does not meet $\varphi^{-1}(\{0\})$, we find $m, M \in (0, \infty)$ such that 
$m \le \varphi(x) \le M$ for all $x \in \abs{\gamma_X}$. We assume without loss of generality that 
$M \ge 1$.
	
Let $0<\eps<m/2$. Since the trajectory of $\gamma_X$ is a compact set in $X$ and $\pip$ is 
continuous, Lemma~\ref{lem:improved_unif_cont_on_compacts} yields a radius
$r > 0$ such that for all $t \in [a,b]$ and $x \in X$ with $d(x, \gamma_X(t)) < r$, we have 
$\abs{\varphi(x) - \varphi(\gamma_X(t))} < \eps$. We select a partition $a = t_0 < t_1 < \dots < t_k = b$ of $[a,b]$ such that 
	\begin{equation}\label{eq:partition_cond}
		\len_{d_X}(\gamma_X\vert_{[t_{i-1}, t_i]}) + \len_{d_Y}(\gamma_Y\vert_{[t_{i-1}, t_i]})
		< \frac{r}{M} 
	\end{equation}
for all $i \in \{1, \dots, k\}$. We then further refine this partition: for every $i \in \{1, \dots, k\}$, we select 
$t_{i-1} = t_{i,0} < t_{i,1} < \dots < t_{i, k_i} = t_i$ such that
	\begin{align*}
		\len_{\delta_{\varphi(\gamma(t_i))}}(\gamma\vert_{[t_{i-1}, t_i]})
		\le \frac{\eps}{k} 
		+ \sum_{j = 1}^{k_i} \delta_{\varphi(\gamma_X(t_i))}(\gamma(t_{i, j-1}), \gamma(t_{i, j})),
	\end{align*}
noting that
$\gamma$ has finite length with respect to the metric $\delta_{\varphi(\gamma(t_i))}$
since $\gamma$ is rectifiable with respect to the metric $d$.
We have for all $t \in [t_{i-1}, t_i]$, $i \in \{1, \dots, k\}$, that 
\[
d_X(\gamma_X(t), \gamma_X(t_i)) \le \len_{d_X}(\gamma_X\vert_{[t_{i-1}, t_i]}) < r/M \le r,
\] 
and consequently by our choices of $r$ and $m$, that
	\begin{equation}\label{eq:gamma_X_phi_value_upper_bound}
		\varphi(\gamma_X(t)) 
		\le \varphi(\gamma_X(t_i)) + \eps 
		\le (1+\eps/m) \varphi(\gamma_X(t_i)).
	\end{equation}
	Thus, using the fact that $\norm{\cdot}$ is coordinate-increasing, we may estimate that
	\begin{align*}
		\ell_\varphi(\gamma) 
		&= \sum_{i = 1}^k \int_{t_{i-1}}^{t_i}  \norm{\bigl(\smallabs{\gamma_X'(t)}, 
			\varphi(\gamma_X(t)) \smallabs{\gamma_Y'(t)}\bigr)} \, dt\\
		&\le (1+\eps/m) \sum_{i = 1}^k \int_{t_{i-1}}^{t_i}  
		\norm{\bigl(\smallabs{\gamma_X'(t)}, 
			\varphi(\gamma_X(t_i))\smallabs{\gamma_Y'(t)}\bigr)} \, dt
	\end{align*}
By Lemma~\ref{lem:lengths_match_constant_case}, it follows that
\begin{equation}\label{eq:component_1}\begin{aligned}
\ell_\varphi(\gamma) 
&\le (1+\eps/m) \sum_{i = 1}^k \len_{\delta_{\varphi(\gamma_X(t_i))}}(\gamma\vert_{[t_{i-1}, t_i]})\\
&\le (1+\eps/m)\left[ \eps+ \sum_{i=1}^k\sum_{j=1}^{k_i} \delta_{\varphi(\gamma_X(t_i))}(\gamma(t_{i, j-1}), \gamma(t_{i, j}))\right].
\end{aligned}\end{equation}
	
Next, let $K = \sum_{i=1}^k k_i$. By the definition of $d_\varphi$, there exists an absolutely continuous curve
$\beta = (\beta_X, \beta_Y) \colon [a,b] \to X \times Y$ such that for all 
$i \in \{1, \dots, k\}$ and $j \in \{1, \dots k_i\}$,  $\beta(t_{i, j-1}) = \gamma(t_{i, j-1})$, $\beta(t_{i,j}) = \gamma(t_{i,j})$, and
	\[
	d_\varphi(\gamma(t_{i, j-1}), \gamma(t_{i,j})) \ge \ell_\varphi(\beta\vert_{[t_{i,j-1}, t_{i,j}]}) - \frac{\eps}{K}.
	\]
We may also assume in this selection process that 
	\begin{equation}\label{eq:permitted_assumption}
		\ell_\varphi(\beta\vert_{[t_{i,j-1}, t_{i,j}]}) \le \ell_\varphi(\gamma\vert_{[t_{i,j-1}, t_{i,j}]})
	\end{equation}
for all pairs of indices $(i,j)$, as if this is not true for some pair $(i,j)$, then we may replace 
$\beta\vert_{[t_{i,j-1}, t_{i,j}]}$ with $\gamma\vert_{[t_{i,j-1}, t_{i,j}]}$ and all required properties of $\beta$ remain true.
If there is an index $1\le i\le k$ such that $\beta_X\vert_{[t_{i-1}, t_i]}$ leaves the
ball $B_X(\gamma_X(t_i),r)$, then from~\eqref{eq:permitted_assumption} and by the fact that $\beta(t_i)=\gamma(t_i)$,
necessarily
\[
\ell_\varphi(\gamma\vert_{[t_{i-1}, t_i]})\ge \ell_\pip(\beta\vert_{[t_{i-1}, t_i]})\ge \len_{d_X}(\beta_X\vert_{[t_{i-1}, t_i]})\ge r.
\]
This is not possible since by~\eqref{eq:partition_cond}, we have
\[
\ell_\varphi(\gamma\vert_{[t_{i-1}, t_i]})\le M\, \int_{t_{i-1}}^{t_i}\norm{\bigl(\smallabs{\gamma_X'(t)}, \smallabs{\gamma_Y'(t)}\bigr)} \, dt
\le M\, \int_{t_{i-1}}^{t_i} (|\gamma_X^\prime(t)|+|\gamma_Y^\prime(t)|)\, dt<r.
\]

We estimate that
	\begin{equation}\label{eq:component_2}
		\len_{d_\varphi}(\gamma) \ge \sum_{i=1}^k\sum_{j=1}^{k_i} d_\varphi(\gamma(t_{i, j-1}), \gamma(t_{i,j}))
		\ge \ell_{\varphi}(\beta) - \eps.
	\end{equation}
As observed above, $\beta_X([t_{i-1},t_i])\subset B_X(\gamma_X(t_i),r)$, and so
\[
\varphi(\beta_X(t)) \ge \varphi(\gamma_X(t_i)) - \eps\ge (1-\eps/m) \varphi(\gamma_X(t_i))
\]
	for all $t \in [t_{i-1}, t_i]$, $i \in \{1, \dots, k\}$. Thus, using the fact that $\norm{\cdot}$ is coordinate-increasing, it follows that
	\begin{align*}
		\ell_\varphi(\beta) 
		&= \sum_{i = 1}^k \int_{t_{i-1}}^{t_i}  \norm{\bigl(\smallabs{\beta_X'(t)}, 
			\varphi(\beta_X(t)) \smallabs{\beta_Y'(t)}\bigr)} \, dt\\
		&\ge (1-\eps/m) \sum_{i = 1}^k \int_{t_{i-1}}^{t_i}  
		\norm{\bigl(\smallabs{\beta_X'(t)}, 
			\varphi(\gamma(t_i))\smallabs{\beta_Y'(t)}\bigr)} \, dt.
	\end{align*}
	Another use of Lemma \ref{lem:lengths_match_constant_case}, combined with the fact that 
for $j \in \{0, \dots, k_i\}$
the points $\gamma(t_{i, j})$  lie along $\beta\vert_{[t_{i-1}, t_i]}$,  yields that
	\begin{align*}
		\ell_\varphi(\beta) 
		\ge (1-\eps/m) \sum_{i = 1}^k \len_{\delta_{\varphi(\gamma(t_i))}}(\beta\vert_{[t_{i-1}, t_i]})
		\ge (1-\eps/m) \sum_{i,j} \delta_{\varphi(\gamma(t_i))}(\gamma(t_{i, j-1}), \gamma(t_{i, j})).	
	\end{align*}
	Finally, combining this estimate with \eqref{eq:component_1} and \eqref{eq:component_2}, we get
	\[
	\ell_\varphi(\gamma) \le (1+\eps/m)\, \left[\eps+\frac{1}{1-\eps/m}\, \ell_\pip(\beta)\right]
	  \le (1+\eps/m)\, \left[\eps+\frac{1}{1-\eps/m}\, \left(\len_{d_\pip}(\gamma)+\eps\right)\right].
	\]
	Letting $\eps \to 0$, we obtain the claimed $\ell_\varphi(\gamma) \le \len_{d_\varphi}(\gamma)$, completing the proof.
\end{proof}

To complete the proof of Proposition \ref{prop:lengths_match_warped} , it only remains to eliminate the assumption 
$\abs{\gamma_X} \cap \varphi^{-1}(\{0\}) = \emptyset$ from Lemma \ref{lem:lengths_match_warped_positive}.

\begin{proof}[Proof of Proposition \ref{prop:lengths_match_warped}]
Let $\gamma = (\gamma_X, \gamma_Y) \colon [a,b] \to X \times Y$ be an absolutely continuous path
such that $\gamma_X$ intersects the set $\pip^{-1}(\{0\})$. 
By Lemma~\ref{lem:lengths_match_part1},  to obtain the claimed 
$\ell_\varphi(\gamma) = \len_{d_\varphi} (\gamma)$ it suffices to 
show that $\ell_\varphi(\gamma) \le \len_{d_\varphi} (\gamma)$. Since $\gamma_X$ and $\varphi$ are continuous, 
$K := \gamma_X^{-1}(\varphi^{-1}(\{0\}))$
is a closed subset of $[a, b]$. Thus, $(a, b) \setminus K$ consists of at most countably many pairwise 
disjoint open intervals $U_i$, $i \in I\subset\N$.
	
	Let $\eps > 0$. Then 
	\[
	\ell_\varphi(\gamma) = \int_a^b \smallabs{\gamma_X'(t)} \, dt 
	+ \int_a^b \bigl( \norm{(\smallabs{\gamma_X'(t)}, 
		\varphi(\gamma_X(t)) \smallabs{\gamma_Y'(t)})} 
	- \smallabs{\gamma_X'(t)}\bigr) \, dt.
	\]
Since $\norm{\cdot}$ is unital and coordinate-increasing, the integrand of the second integral is non-negative, 
and vanishes on $K$. Moreover, since $\gamma_X([a,b])$ is compact, $\varphi$ has a maximum value on 
$\gamma_X([a,b])$, and consequently $\ell_\varphi(\gamma) < \infty$. It follows that there exists a finite 
collection $U_{i_1}, \dots U_{i_k}$ of the intervals $U_i, i \in I$ such that
	\begin{align*}
		&\int_a^b \bigl( \norm{(\smallabs{\gamma_X'(t)}, 
			\varphi(\gamma_X(t)) \smallabs{\gamma_Y'(t)})} 
		- \smallabs{\gamma_X'(t)}\bigr) \, dt\\
		&\qquad\le \eps + \sum_{j=1}^k
		\int_{U_{i_j}} \bigl( \norm{(\smallabs{\gamma_X'(t)}, 
			\varphi(\gamma_X(t)) \smallabs{\gamma_Y'(t)})} 
		- \smallabs{\gamma_X'(t)}\bigr) \, dt.
	\end{align*}
The complement $[a,b] \setminus \bigcup_{j=1}^k U_{i_j}$ consists of finitely many pairwise disjoint closed 
intervals $S_1, \dots, S_l$, some of which may be singletons. 
Thus, 
	\[
	\ell_\varphi(\gamma)
	\le \eps + \sum_{j=1}^k \ell_{\varphi}(\gamma\vert_{U_{i_j}}) + \sum_{j=1}^l \int_{S_j} \smallabs{\gamma_X'(t)} \, dt.
	\]
For each $U_i$, $i \in I$, we take an increasing sequence of compact sub-intervals $V_{i,j}\subset U_i$, $j = 1, 2, \dots$, 
such that $\bigcup_jV_{i,j}=U_{i}$. On each of the intervals $V_{i,j}$, 
Lemma~\ref{lem:lengths_match_warped_positive} applies, and hence 
$\ell_\varphi(\gamma\vert_{V_{i,j}}) = \len_{d_\varphi}(\gamma\vert_{V_{i,j}})$. It follows that for every $i \in I$, 
	\[
	\ell_{\varphi}(\gamma\vert_{U_{i}}) = \lim_{j \to \infty} \ell_{\varphi}(\gamma\vert_{V_{i, j}})
	= \lim_{j \to \infty} \len_{d_\varphi}(\gamma\vert_{V_{i,j}}) \le \len_{d_\varphi}(\gamma\vert_{U_i}).
	\]
Moreover, on every $S_i$, $i \in \{1, \dots, l\}$, it follows from~\eqref{eq:d_varphi_x_lower_bound} that
	\[
	\int_{S_i} \smallabs{\gamma_X'(t)} \, dt = \len_{d_X} (\gamma_X \vert_{S_i}) \le \len_{d_\varphi} (\gamma \vert_{S_i}).
	\]
	In conclusion, we have that
	\[
	\ell_\varphi(\gamma)
	\le \eps + \sum_{j=1}^k \len_{d_\varphi}(\gamma\vert_{U_{i_j}}) 
	+ \sum_{j=1}^l \len_{d_\varphi} (\gamma \vert_{S_j})
	= \eps + \len_{d_\varphi} (\gamma).
	\]
	Letting $\eps \to 0$, the claim that $\ell_\varphi(\gamma) \le \len_{d_\varphi} (\gamma)$ follows.
	
We then briefly comment how \eqref{eq:UG_in_skew_prod} follows from the fact we have shown. For this, it suffices to show that for any absolutely continuous path $\gamma = (\gamma_X, \gamma_Y) \colon [a,b] \to X \times Y$, the metric derivative of $P_\varphi \circ \gamma$ is given by
	\[
		\abs{(P_\varphi \circ \gamma)'(t)} = \norm{(\abs{\gamma_X'(t)}, \varphi(\gamma_X(t)) \abs{\gamma_Y'(t)})}
	\]
	for a.e.\ $t \in [a,b]$. However, by \eqref{eq:met_deriv_and_length_function}, we have $\abs{(P_\varphi \circ \gamma)'(t)} = s_{P_\varphi \circ \gamma}'(t)$ for a.e.\ $t \in [a,b]$, and the previously shown fact that $\len_{d_\varphi} = \ell_\varphi$ yields that
	\[
		s_{P_\varphi \circ \gamma}(t)
		= \len_{d_\varphi} (\gamma\vert_{[a,t]})
		= \ell_\varphi(\gamma\vert_{[a,t]}) 
		= \int_a^t \norm{(\abs{\gamma_X'(r)}, \varphi(\gamma_X(r)) \abs{\gamma_Y'(r)})} \, dr
	\]
	for all $t \in [a,b]$. Thus, differentiating this identity on both sides proves the claim.
\end{proof}

\section{Warped products of metric measure spaces}

In preparation for the proofs of Theorems \ref{thm:solid_hyp_filling_is_PI} and \ref{thm:warped_product_PI}, we now proceed to 
discuss warped products of metric measure spaces, and also point out some basic properties of $(p,p)$-Sobolev--Poincar\'e spaces.

\subsection{Measures on warped products}\label{subsect:measures_on_warped_prod}

Let $(X, d_X, \mu_X)$ and $(Y, d_Y, \mu_Y)$ be metric measure spaces, and let $\varphi \colon X \to [0, \infty)$ be 
continuous. 
Let $\mu_{XY}$ denote the product outer measure
$\mu_X\times\mu_Y$ on $X\times Y$.
We define the \emph{product measure} $\mu_\varphi$ on the warped product space $X \times_\varphi Y$  
to be the push-forward of $\mu_{XY}$ 
under the projection 
$P_\varphi$ to $\sim_\varphi$-equivalence classes. 
That is, for sets $A\subset X\times_\pip Y$, we set
\[
\mu_\pip(A):=\mu_{XY}(P_\pip^{-1}(A)),
\]
and then define $\mu_\varphi$-measurable sets on $X \times_\varphi Y$ via the Carath\'eodory condition.
The warped product of the metric measure spaces 
$(X, d_X, \mu_X)$ and $(Y, d_Y, \mu_Y)$ is then the triple $(X \times_\varphi Y, d_\varphi, \mu_\varphi)$.

\begin{lemma}\label{lem:measurability_condition_for_pushforward}
	 Let $(X, d_X, \mu_X)$ and $(Y, d_Y, \mu_Y)$ be length 
	metric measure spaces as according to Standing Assumptions \ref{stassm:extra_assumptions}, and let 
	$\varphi \colon X \to [0, \infty)$ be continuous with $\varphi(x) > 0$ for $\mu_X$-a.e.\ $x \in X$.
	Then a set $S\subset X\times_\pip Y$ is $\mu_\pip$-measurable if and only if $P_\pip^{-1}(S)$ is $\mu_{XY}$-measurable.
\end{lemma}

\begin{proof}
	Let $S \subset X \times_{\varphi} Y$. If $P_\varphi^{-1} (S)$ is $\mu_{XY}$-measurable, then the $\mu_\varphi$-measurability of $S$ follows by computing that 
	\begin{align*}
		\mu_\pip(E)
		&= \mu_{XY}(P_\pip^{-1}(E))\\
		&=\mu_{XY}(P_\pip^{-1}(E)\cap P_\pip^{-1}(S))
		+\mu_{XY}(P_\pip^{-1}(E)\setminus P_\pip^{-1}(S))\\
		&= \mu_\pip(E\cap S)+\mu_\pip(E\setminus S)
	\end{align*}
	for every set $E \subset X \times_{\varphi} Y$.
	
	For the converse, let $S \subset X \times_\varphi Y$ be $\mu_\varphi$-measurable, and let $E \subset X \times Y$ be a set. We denote
	$K := \pip^{-1}(\{0\})\, \times Y$, noting that $\mu_{XY}(K) = 0$ by assumption. 
	We observe that $P_\varphi(K) \cap P_\varphi((X \times Y) \setminus K) = \emptyset$, and that $P_\varphi$ is injective in $(X \times Y) \setminus K$. Thus, we have $P^{-1}_\pip(P_\pip(E)) \setminus E \subset K$, and consequently $\mu_{XY}(P^{-1}_\pip(P_\pip(E)) \setminus E) = 0$. It follows that
	\begin{equation}\label{eq:measure_image_preimage_equality}
		\mu_{XY}(E) = \mu_{XY}(P^{-1}_\pip(P_\pip(E))) = \mu_\pip(P_\pip(E)).
	\end{equation} 
	Applying \eqref{eq:measure_image_preimage_equality} also on the sets $E \cap P_\pip^{-1}(S)$ and $E \setminus  P_\pip^{-1}(S)$, we have
	\begin{align*}
		\mu_{XY}(E \cap P_\pip^{-1}(S)) 
		&= \mu_\pip(P_\pip(E \cap P_\pip^{-1}(S))) = \mu_\pip(P_\pip(E) \cap S) 
		\qquad \text{and}\\
		\mu_{XY}(E \setminus P_\pip^{-1}(S)) 
		&= \mu_\pip(P_\pip(E \setminus P_\pip^{-1}(S))) = \mu_\pip(P_\pip(E) \setminus S). 
	\end{align*}
	Thus, we may compute using the $\mu_\pip$-measurability of $S$ that
	\begin{align*}
		\mu_{XY}(E) = \mu_\pip(P_\pip(E)) &= \mu_\pip(P_\pip(E) \cap S) + \mu_\pip(P_\pip(E) \setminus S)\\
		&= \mu_{XY}(E \cap P_\pip^{-1}(S)) + \mu_{XY}(E \setminus P_\pip^{-1}(S)),
	\end{align*}
	completing the proof that $P_\pip^{-1}(S)$ is $\mu_{XY}$-measurable.
\end{proof}

We now prove that under certain conditions, $(X \times_\varphi Y, d_\varphi, \mu_\varphi)$
is a metric measure space that satisfies our Standing Assumptions \ref{stassm:extra_assumptions}.

\begin{lemma}\label{lem:warped_produc_is_a_MMS}
Let $(X, d_X, \mu_X)$ and $(Y, d_Y, \mu_Y)$ be length 
metric measure spaces satisfying Standing Assumptions \ref{stassm:extra_assumptions}, and let 
$\varphi \colon X \to [0, \infty)$ be continuous.  
Suppose that at least one of the following two conditions holds:
\begin{enumerate}
\item $\inf_{x \in X} \varphi(x) > 0$,
\item $\mu_Y(Y)<\infty$ and $\varphi(x) > 0$ for $\mu_X$-a.e.\ $x \in X$. 
\end{enumerate}
Then $(X \times_\varphi Y, d_\varphi, \mu_\varphi)$ is a metric measure space that 
satisfies Standing Assumptions~\ref{stassm:extra_assumptions}.
\end{lemma}

\begin{proof}
Since $(X,d_X)$ and $(Y, d_Y)$ are separable spaces, so is $X\times Y$. 
As $X\times_\pip Y=P_\pip(X\times Y)$ and, by 
Lemma~\ref{lem:warped_product_topology}, $P_\pip$ is continuous, it follows that 
$X \times_\varphi Y$ is separable. Similarly,
if $E \subset X \times_\varphi Y$ is Borel, then $P_\varphi^{-1}(E)$ is Borel since
$P_\varphi$ is continuous by Lemma \ref{lem:warped_product_topology}, and therefore $E$ is 
$\mu_\varphi$-measurable by Lemma \ref{lem:measurability_condition_for_pushforward}. 
Thus, Borel sets are $\mu_\varphi$-measurable.

Next, we show that $\mu_\varphi$ is Borel regular. Let $E \subset X \times_\varphi Y$ be measurable, and denote 
$K = \varphi^{-1}(\{0\})\times Y$. Note that since $\varphi(x) > 0$ for $\mu_X$-a.e.\ $x \in X$, we have 
$\mu_{XY}(K)=0$. 
As $\mu_{XY}$ is Borel regular, by e.g.~\cite[Proposition~3.3.37]{HKST}, we find a sequence 
of open sets $U_j \subset X \times Y$ for which $P_\varphi^{-1}(E) \subset U_j$ for all $j$, and
\[
	\mu_\pip(E) = \mu_{XY}(P_\varphi^{-1}(E)) 
	= \lim_{j \to \infty} \mu_{XY}(U_j)
	= \lim_{j \to \infty} \mu_{XY}(U_j \cup K)
\]
Since $K$ is closed, and since $P_\varphi$ is open in $(X \times Y) \setminus K$ by 
Lemma \ref{lem:warped_product_topology}, it follows that the sets $P_\varphi(U_j \setminus K)$ 
are open. On the other hand, $P_\varphi(K)$ is closed, since its complement 
$P_\varphi((X \times Y) \setminus K)$ is open. Thus, $P_\varphi(U_j \cup K)$ are Borel 
sets containing $E$. Moreover, we observe that $P_\varphi^{-1}(P_\varphi(U_j \cup K)) = U_j \cup K$ 
since all of the non-injectivity of $P_\varphi$ occurs in $K$, and consequently
\[
	\lim_{j \to \infty} \mu_{XY}(U_j \cup K)
	= \lim_{j \to \infty} \mu_{XY}(P_\varphi^{-1}(P_\varphi(U_j \cup K)))
	= \lim_{j \to \infty} \mu_\varphi(P_\varphi(U_j \cup K)).
\]
Thus, we may approximate $E$ in $\mu_\varphi$ from the outside with the Borel sets 
$P_\varphi(U_j \cup K)$, implying that $\mu_\varphi$ is Borel regular.

It now only remains to show that balls on $X \times_\varphi Y$ have finite and positive 
$\mu_\pip$-measure. 
For this, let $(x_0, y_0) \in X \times Y$ and $r > 0$. Since $P_\varphi$ is continuous by 
Lemma \ref{lem:warped_product_topology}, $P_\varphi^{-1} (B_{d_\varphi}(P_\varphi(x_0, y_0), r))$ is 
open in $X \times Y$, and hence has positive $\mu_{XY}$-measure; thus, balls in 
$X \times_\varphi Y$ have positive $\mu_\pip$-measure. By~\eqref{eq:d_varphi_x_lower_bound} of 
Lemma~\ref{lem:d_varphi_bounds}, we have
 \begin{equation}\label{eq:ball-box1}
 P_\varphi^{-1} (B_{d_\varphi}(P_\varphi(x_0, y_0), r)) \subset B_X(x_0, r) \times Y,
 \end{equation}
 so in the case 
 $\mu_Y(Y) < \infty$, it follows that $B_{d_\varphi}(P_\varphi(x_0, y_0), r)$ has finite $\mu_\pip$-measure. 

In the remaining case we have $c :=\inf_{x \in X} \varphi(x) > 0$. Now, supposing that 
$(x, y) \in  P_\varphi^{-1} (B_{d_\varphi}(P_\varphi(x_0, y_0), r))$, then by the definition of $d_\varphi$, there exists an absolutely continuous path $\gamma = (\gamma_X, \gamma_Y)$ from $(x_0, y_0)$ to $(x, y)$ with $\ell_\varphi(\gamma) < r$. We then estimate that 
\[
d_Y(y, y_0) \le \len_{d_Y}(\gamma_Y) \le c^{-1} \ell_\pip(\gamma) < c^{-1} r.
\] 
Combining this with~\eqref{eq:ball-box1}, we conclude that 
\begin{equation}\label{eq:ball-box2}
P_\varphi^{-1} (B_{d_\varphi}(P_\varphi(x_0, y_0), r)) \subset B_X(x_0, r) \times B_Y(y_0, c^{-1}r), 
\end{equation}
implying that 
$B_{d_\varphi}(P_\varphi(x_0, y_0), r)$ has finite $\mu_\pip$-measure.
\end{proof}

In particular, we now have a rigorous definition for the spaces 
$\H_{\alpha, \beta}(Y)$ and $\H_{\alpha, \beta}^\circ(Y)$ stated in the introduction. 
Indeed, if $(Y, d_{Y}, \mu_{Y})$ is a length space and $\alpha, \beta \in (0, \infty)$, we define
\begin{align*}
	\H_{\alpha, \beta}^\circ(Y) 
	&= \left([0, \infty), d_{\text{Eucl}}, e^{\beta t} \, dt\right) 
		\times_{\pip_\alpha} (Y, d_{Y}, \mu_{Y}) &&\text{ where }\pip_\alpha(t)=e^{\alpha t},\\
	\H_{\alpha, \beta}(Y) 
	&= \left([0, \infty), d_{\text{Eucl}}, \sinh^\beta(t) \, dt\right)
		\times_{\psi_\alpha} (Y, d_{Y}, \mu_{Y}) &&\text{ where }\psi_\alpha(t)=\sinh^\alpha(t).
\end{align*}
By Lemma \ref{lem:warped_produc_is_a_MMS}, $\H_{\alpha, \beta}^\circ(Y)$ is a metric measure space under no 
additional assumptions, and $\H_{\alpha, \beta}(Y)$ is a metric measure space when $\mu_{Y}(Y) < \infty$. 
Note that without $\mu_{Y}(Y) < \infty$, every neighborhood of the point $\{0\} \times Y$ in 
$\H_{\alpha, \beta}(Y)$ would have infinite measure.

\subsection{Tensorization and warped product spaces}
Our proofs of Theorems \ref{thm:solid_hyp_filling_is_PI} and \ref{thm:warped_product_PI} also need some 
simple tensorization-related facts regarding warped products. In general, tensorization of various forms of 
metric derivatives is a complicated topic. For a more comprehensive study of tensorization on warped 
products spaces in a setting similar to $N^{1,p}$-spaces with $p=2$, 
see~\cite[Section 4]{GGN}. However, as we require these results for other values of $p$ as well,
we elect to provide proofs of the simple statements we require instead of referring the reader to \cite{GGN}.

\begin{lemma}\label{lem:warped_prod_horiz_and_vert_UGs}
Let $(X, d_X, \mu_X)$ and $(Y, d_Y, \mu_Y)$ be length 
metric measure spaces satisfying Standing Assumptions~\ref{stassm:extra_assumptions}.
Suppose that $\varphi \colon X \to [0, \infty)$ is continuous and 
satisfies at least one of the conditions~(1), (2) 
listed in Lemma~\ref{lem:warped_produc_is_a_MMS}. 
Let $u \colon X \times_\varphi Y\to \mathbb{R}$ 
be measurable, and let $g$ be an upper gradient of $u$. Denote $\tilde{u} = u \circ P_\varphi$ and 
$\tilde{g} = g \circ P_\varphi$. Then
	\begin{enumerate}
		\item for $\mu_Y$-a.e.\ $y \in Y$, the function $\tilde{g}_y \colon X \to [0, \infty]$, $x \mapsto \tilde{g}(x,y)$ is an upper 
gradient of the function $\tilde{u}_y \colon X \to \R$, $x \mapsto \tilde{u}(x,y)$; 
		\item for $\mu_X$-a.e.\ $x \in X$, the function $\varphi(x)\tilde{g}_x \colon Y \to [0, \infty]$, 
$y \mapsto \varphi(x)\tilde{g}(x,y)$ is an upper gradient of the function $\tilde{u}_x \colon Y \to \R$, $y \mapsto \tilde{u}(x,y)$.
	\end{enumerate}
\end{lemma}

\begin{proof}
First, note that if $v:X\times_\pip Y\to\R$ is $\mu_\pip$-measurable, then the function
$v\circ P_\pip:X\times Y\to\R$ is $\mu_{XY}$-measurable. Indeed for $t\in\R$, we know that
$E_t:=\{[(x,y)]\in X\times_\pip Y\, :\, v([(x,y)])>t\}$ is $\mu_\pip$-measurable. Since $\mu_\pip$ is Borel
regular, we can find a Borel set $W\subset X\times_\pip Y$ with $E_t\subset W$ such that 
$\mu_\pip(W\setminus E_t)=0$. Since $P_\pip$ is continuous, it follows that $P_\pip^{-1}(W)$ 
is  also a Borel subset of $X\times Y$, and
\[
0=\mu_\pip(W\setminus E_t)=\mu_{XY}(P_\pip^{-1}(W)\setminus P_\pip^{-1}(E_t)),
\]
that is, $P_\pip^{-1}(W)\setminus P_\pip^{-1}(E_t)$ is $\mu_{XY}$-measurable.
Hence, $P_\pip^{-1}(E_t)=P_\pip^{-1}(W)\setminus P_\pip^{-1}(W\setminus E_t)$ is 
also $\mu_{XY}$-measurable. Thus, by definition, $v\circ P_\pip$ is $\mu_{XY}$-measurable.

Since $u$ is $\mu_\varphi$-measurable, by the above argument we know that
$\tilde{u}$ is $\mu_{XY}$-measurable. Also, since $g$ is Borel measurable, 
by the continuity of $P_\varphi$ shown in Lemma~\ref{lem:warped_product_topology} we know that 
$\tilde{g}$ is Borel. Consequently, all $\tilde{g}_x$ and $\tilde{g}_y$ are Borel, and 
Tonelli's theorem applied to the positive and negative parts of $\tilde{u}$
ensures the measurability of $\tilde{u}_x$ for $\mu_X$-a.e.\ $x \in X$ 
and the measurability of $\tilde{u}_y$ for $\mu_Y$-a.e.\ $y \in Y$.
	
Let then $y \in Y$ be such that $\tilde{u}_y$ is measurable, let $x, x' \in X$, and let 
$\gamma_X \colon [a,b] \to X$ be an absolutely continuous path from $x$ to $x'$. 
Setting $\gamma:[a,b]\to X\times_\pip Y$ and $\gamma_0:[a,b]\to X\times Y$ by
$\gamma_0(t)=(\gamma_X(t),y)$ and 
$\gamma(t)=P_\pip((\gamma_X(t),y))$, we have that $\gamma$ 
is an absolutely continuous path. 
We use Proposition~\ref{prop:lengths_match_warped}  
to see that 
\[
\int_\gamma g\, ds=\int_{P_\pip\circ\gamma_0}g\, ds=\int_a^b g(\gamma_X(t),y)\, |\gamma_X^\prime(t)|\, dt
  =\int_{\gamma_X}\tilde{g}_y\, ds,
\]
where we have used the standing assumption that $\norm{\cdot}$ is a unitary norm. Now, from
the fact that $g$ is an upper gradient of $u$ we see that 
\[ 
		\abs{\tilde{u}_y(x) - \tilde{u}_y(x')}
		= \abs{u(P_\varphi(x, y)) - u(P_\varphi(x', y))}
		\le \int_\gamma g\, ds=\int_{\gamma_X} \tilde{g}_y\, ds. 
\] 
	Thus, $\tilde{g}_y$ is indeed an upper gradient of $\tilde{u}_y$.
	
The other case is proven analogously. Indeed, let $x \in X$ be such that $\tilde{u}_x$ is measurable, let $y, y' \in Y$, and 
let $\gamma_Y \colon [a,b] \to Y$ be an absolutely continuous path from $y$ to $y'$.  Now setting
$\gamma_0\colon[a,b]\to X\times Y$ and $\gamma\colon [a,b]\to X\times_\pip Y$ by
$\gamma_0(t)=(x,\gamma_Y(t))$ and $\gamma=P_\pip\circ\gamma_0$,
from Proposition~\ref{prop:lengths_match_warped}  
again, we see that
\[
\int_\gamma g\, ds=\int_{P_\pip\circ\gamma_0} g\,ds=\int_a^b g(x,\gamma_Y(t))\, \pip(x)\, |\gamma_Y^\prime(t)|\, dt
=\int_{\gamma_Y} \pip(x)\, \tilde{g}_x\, ds.
\]
The rest of the proof is identical to the previous case. 
\end{proof}

The other result on tensorization we use is as follows.

\begin{lemma}\label{lem:easy_tensorization}
Let $(X, d_X, \mu_X)$ and $(Y, d_Y, \mu_Y)$ be length 
metric measure spaces satisfying Standing Assumptions~\ref{stassm:extra_assumptions}.
Suppose that $\varphi \colon X \to [0, \infty)$ is continuous and  
satisfies at least one of the conditions~(1), (2) 
listed in Lemma~\ref{lem:warped_produc_is_a_MMS}.
Let $\tilde{u} \colon X \times Y \to \R$ be of the form
	\[
	\tilde{u}(x,y) = u_X(x) u_Y(y),
	\]
where $u_X \colon X \to \R$ and $u_Y \colon Y \to \R$ are Lipschitz, and $u_X \equiv 0$ in a (possibly empty) neighborhood of 
$\varphi^{-1}(\{0\})$. 
Let $K=\pip^{-1}(\{0\})\times Y$.
Then there is a  function $u \colon X \times_\varphi Y \to \R$ 
satisfying $u \circ P_\varphi = \tilde{u}$, 
and $u$ has an upper gradient 
$g \colon X \times_\varphi Y \to [0, \infty]$ given by
\[
g([(x,y)])=\begin{cases}
 |u_Y(y)|\, \Lip_{d_X}u_X(x)+ \dfrac{|u_X(x)|}{\pip(x)}\, \Lip_{d_Y}u_Y(y)&\text{ if }[(x,y)]\not\in P_\pip(K),\\
    0 &\text{ if }[(x,y)]\in P_\pip(K).
 \end{cases}
 \]
\end{lemma}

\begin{proof}
Let $E:=\pip^{-1}(\{0\})$.  Note that
by Corollary \ref{cor:d_varphi_equiv_classes}, we have $P_\pip(x,y)=[(x,y)]=\{x\}\times Y\in X\times_\pip Y$
when $x\in E$ and $y\in Y$, and $P_\pip(x,y)=[(x,y)]=\{(x,y)\}$ when $x\in X\setminus E$ and $y\in Y$. 
Since $u_X$ vanishes on  $E$, we set 
$u([(x,y)])=0$ if $x\in E$ and $y\in Y$, and $u([x,y])=\tilde{u}(x,y)$ if $x\in X\setminus E$ and $y\in Y$.
This defines
$u$ on $X\times_\pip Y$ unambiguously.
 Moreover, $g$ is well-defined, since if $[(x,y)]\not\in P_\pip(K)$, then $[(x,y)]=\{(x,y)\}$.

We claim that $g$ is Borel. For this, let $U \subset [0, \infty]$ be open. Since $\Lip_{d_X} u_X$ and $\Lip_{d_Y} u_Y$ are 
Borel, see e.g.\ \cite[Lemma 6.2.5]{HKST}, it follows that $g \circ P_\varphi$ is Borel. As $K$ is closed, 
$P_\varphi^{-1}(g^{-1}(U)) \setminus K$ is hence a Borel set. Since $P_\varphi$ is an open injection in 
$(X \times Y) \setminus K$, it follows that $P_\varphi$ maps Borel sets contained in $(X \times Y) \setminus K$ 
to Borel sets; thus, the set
\[
S := P_\varphi( P_\varphi^{-1}(g^{-1}(U)) \setminus K ) = g^{-1}(U) \setminus P_\varphi(K)
\]
is Borel. 
If $0\not\in U$, then $g^{-1}(U)=S$, and if $0\in U$, then $g^{-1}(U)=S\cup P_\pip(K)$ is also Borel as $P_\pip(K)$ is
closed in $X\times_\pip Y$. Hence $g$ is Borel function.

We next prove that $u$ is locally Lipschitz. For this, fix a point $(x_0, y_0) \in X \times Y$. 
In the case $\varphi(x_0) = 0$, there is a real number $r >0$ for which $u_X$ is identically zero in $B_X(x_0, r)$; since 
$B_{d_\varphi}((x_0, y_0), r) \subset B_X(x_0, r) \times Y$ by~\eqref{eq:d_varphi_x_lower_bound} of 
Lemma~\ref{lem:d_varphi_bounds}  we conclude that $u$ is identically zero in a neighborhood of $P_\varphi(x_0, y_0)$ in this case. 

In the other case $\varphi(x_0) > 0$, let $L$ be such that $u_X$ and $u_Y$ are both $L$-Lipschitz. 
We select $r = r_{1/2}(x_0)$, where $r_{1/2}(x_0)$ is as in 
Lemma~\ref{lem:d_varphi_bounds} \eqref{eq:d_varphi_y_lower_bound} with 
$\eps = 1/2$. Since $u_X$ and $u_Y$ are Lipschitz, we find a bound $C > 0$ for which 
$\abs{u_X(x)} \le C$ for all $x \in B_X(x_0, r/2)$ and $\abs{u_Y(y)} \le C$ for all 
$y \in B_Y(y_0, r/\varphi(x_0))$. Now, if $(x,y), (x', y') \in B_{d_\varphi}((x_0, y_0), r/2)$, 
then \eqref{eq:d_varphi_x_lower_bound} yields $d_X(x,x') \le d_{\varphi}((x,y),(x',y')) < r$ 
and $d_X(x, x_0) < r/2$. Moreover, by \eqref{eq:d_varphi_y_lower_bound}, 
\begin{align*}
	\min \{ r, \varphi(x_0) d_Y(y, y') / 2 \} &\le d_{\varphi}((x,y),(x',y')) < r 
	\qquad \text{and}\\
	\min \{ r, \varphi(x_0) d_Y(y, y_0) / 2 \} &\le d_{\varphi}((x,y),(x_0,y_0)) < r/2,
\end{align*} 
which implies that $\varphi(x_0) d_Y(y, y') \le 2d_{\varphi}((x,y),(x',y'))$ and $d_Y(y', y_0) < r/\varphi(x_0)$.
It follows that for all $(x,y), (x', y') \in B_{d_\varphi}((x_0, y_0), r/2)$, we have
\begin{align*}
	\abs{\tilde{u}(x,y) - \tilde{u}(x', y')}
	&\le \abs{u_X(x)} \abs{u_Y(y) - u_Y(y')} + \abs{u_Y(y')} \abs{u_X(x) - u_X(x')}\\
	&\le C L d_Y(y, y') + C L d_X(x, x')\\
	&\le \left(CL + \frac{2CL}{\varphi(x_0)} \right) d_\varphi((x,y), (x',y')).
\end{align*}

Thus, $u$ is indeed locally Lipschitz. Hence, by e.g. \cite[Lemma 6.2.6]{HKST}, in order to prove that $g$ is an 
upper gradient of $u$, it suffices to show that $\Lip_{d_\varphi} u \le g$. For this, suppose that $(x, y) \in X \times Y$. 
If $\varphi(x) = 0$, we then note that $(\Lip_{d_\varphi} u)(P_\varphi(x,y)) = 0$ since $u$ is identically zero in a 
neighborhood of $(x,y)$. If on the other hand $\varphi(x) \ne 0$, the fact that $P_\varphi$ is a local 
homeomorphism near $(x,y)$ by Lemma \ref{lem:warped_product_topology} yields that 
$d_\varphi((x,y), (x', y')) \to 0$ if and only if $(x', y') \to (x,y)$ in the usual product metric. Thus, we obtain that
\begin{align*}
	&(\Lip_{d_\varphi} u)(P_\varphi(x,y)) 
	= \limsup_{d_\varphi((x,y), (x', y')) \to 0} \frac{\abs{\tilde{u}(x,y) - \tilde{u}(x', y')}}{d_\varphi((x,y), (x', y'))}\\
	&\qquad\qquad\le \limsup_{(x', y') \to (x,y)} 
	\left( \abs{u_Y(y')} \frac{\abs{u_X(x) - u_X(x')}}{d_\varphi((x,y), (x', y'))} 
	+  \abs{u_X(x)} \frac{\abs{u_Y(y) - u_Y(y')}}{d_\varphi((x,y), (x', y'))} \right).
\end{align*}
However, by parts \eqref{eq:d_varphi_x_lower_bound} and \eqref{eq:d_varphi_y_lower_bound} of 
Lemma~\ref{lem:d_varphi_bounds} and the continuity of $u_X$ and $u_Y$, we may further estimate that
\begin{align*}
	(\Lip_{d_\varphi} u)(P_\varphi(x,y)) 
	&\le \limsup_{(x', y') \to (x,y)} 
	\left( \abs{u_Y(y')} \frac{\abs{u_X(x) - u_X(x')}}{d_X(x, x')} 
	+  \abs{u_X(x)} \frac{\abs{u_Y(y) - u_Y(y')}}{\varphi(x) d_Y(y, y')} \right)\\
	&\le \abs{u_Y(y)} \Lip_{u_X}(x) + \dfrac{\abs{u_X(x)}}{\varphi(x)} \Lip_{u_Y}(y).
\end{align*}
Thus, $\Lip_{d_\varphi} u \le g$, and the proof is complete.
\end{proof}

\section{Proof of Theorem~\ref{thm:warped_product_PI}}

Theorem~\ref{thm:warped_product_PI} is a special case 
of the following, more general, theorem.

\begin{theorem}\label{thm:warped_product_PI-ext}
Let $(X, d_X, \mu_X)$ and $(Y, d_Y, \mu_Y)$ be length metric measure spaces
which satisfy Standing Assumptions~\ref{stassm:extra_assumptions},
and let $1\le p <\infty$. Suppose that $\varphi \colon X \to [0, \infty)$ is a continuous function 
for which at least one of the following two conditions holds:
\begin{enumerate}
\item[(i)] $\inf_{x \in X} \varphi(x) > 0$,
\item[(ii)] $\mu_Y(Y)<\infty$ and $\varphi(x) > 0$ for $\mu_X$-a.e.\ $x \in X$. 
\end{enumerate} 
Moreover, suppose that all of the following conditions hold.
	\begin{enumerate} 
		\item $X$ is a $(p,p)$-Sobolev--Poincar\'e space with constant $C_X$, and $\mu_X(X)=\infty$.
		\item $Y$ is an $\infty$-weak local $(p,p)$-Poincar\'e space.
		\item There exists a measurable set $E \subset X$ with $\mu_X(E)>0$ and
		\[
			\varphi^{-p} \in L^\infty(E, \mu_X) \quad \text{and} \quad \varphi^{-p} \notin L^1(E, \mu_X).
		\]
	\end{enumerate}
	Then $(X\times_\pip Y, d_\pip, \mu_\pip)$ 
	is a $(p,p)$-Sobolev--Poincar\'e space with constant $C_X$, and therefore
	\[
		D^{1,p}(X \times_\varphi Y) = N^{1,p}(X \times_\varphi Y) + \R.
	\] 
\end{theorem}

\begin{proof} 
We will make use of the fact that when $v:X\times_\pip Y\to[0,\infty)$ is $\mu_\pip$-measurable, then 
$v\circ P_\pip$ is $\mu_{XY}$-measurable and 
\begin{equation}\label{eq:Chevalier}
\int_{X\times_\pip Y}v\, d\mu_\pip=\int_{X\times Y}v\circ P_\pip\, d\mu_{XY},
\end{equation}
where $\mu_{XY}$ is the product measure $\mu_X\times\mu_Y$ on $X\times Y$.
The above holds thanks to Lemma~\ref{lem:measurability_condition_for_pushforward} and Cavalieri's principle.

By Lemma~\ref{lem:warped_produc_is_a_MMS}, $(X \times_\varphi Y, d_\varphi, \mu_\varphi)$ is a 
metric measure space satisfying Standing Assumptions~\ref{stassm:extra_assumptions}. 
Let $u \in D^{1,p}(X \times_\varphi Y)$, and let 
$g \in L^p(X\times_\pip Y,\mu_\varphi)$ be an upper gradient of $u$, with intent of finding $c_u \in \R$ with
	\[
		\norm{u - c_u}_{L^p(X \times_\varphi Y, \mu_\varphi)} \le C_X \norm{g}_{L^p(X \times_\varphi Y, \mu_\varphi)}.
	\]
Following Lemma~\ref{lem:warped_prod_horiz_and_vert_UGs}, we
again adopt the notation $\tilde{u} = u \circ P_\varphi$ and $\tilde{g} = g \circ P_\varphi$, 
as well as the notation $\tilde{u}_x$, $\tilde{u}_y$, $\tilde{g}_x$, and $\tilde{g}_y$ for the functions 
$\tilde{u}(x, \cdot)$, $\tilde{u}(\cdot, y)$, $\tilde{g}(x, \cdot)$ and $\tilde{g}(\cdot, y)$ for all $x \in X$, 
$y \in Y$. By~\eqref{eq:Chevalier}, 
our claim hence amounts to showing that 
there is a constant $c_u$ for which
$\norm{\tilde{u} - c_u}_{L^p(X\times Y, \mu_{XY})} \le C_X \norm{\tilde{g}}_{L^p(X\times Y, \mu_{XY})}$.

As $u\in D^{1,p}(X\times_\pip Y)$ and therefore $u$ is integrable over balls in $X\times_\pip Y$,
it follows that $\tilde{u}$ is integrable on balls in $X\times Y$ by Lemma~\ref{lem:d_varphi_bounds}~(1) and~\eqref{eq:Chevalier}.
Since also $\tilde{g} \in L^p(X\times Y, \mu_{XY})$ by~\eqref{eq:Chevalier}, 
we can use Fubini's theorem to conclude that $\tilde{g}_y \in L^p(\mu_X)$ and 
$\tilde{u}_y$ is integrable over balls with respect to $\mu_X$ for $\mu_Y$-a.e.\ $y \in Y$. Moreover, by 
Lemma~\ref{lem:warped_prod_horiz_and_vert_UGs}, $\tilde{g}_y$ is an upper gradient 
of $\tilde{u}_y$ for $\mu_Y$-a.e.\ $y \in Y$. Thus, if we fix any $x_0 \in X$, and let
$\tilde{u}_X \colon Y \to \R$ be the function given by 
\[
\tilde{u}_X(y) = \lim_{r \to \infty} (\tilde{u}_y)_{B_X(x_0, r)},
\] 
then by Lemma~\ref{lem:inf_meas_PI_characterization},  
we have that $\tilde{u}_X(y)$ is well defined and
	\[
		\int_X \abs{\tilde{u}(x, y) - \tilde{u}_X(y)}^p \, d\mu_X(x) \le C_X^p \int_X \tilde{g}_y^p \, d\mu_X
	\]
for $\mu_Y$-a.e.\ $y \in Y$.
We also note that $\tilde{u}_X$ is measurable, since it is a pointwise limit of the functions 
$y\mapsto (\tilde{u}_y)_{B_X(x_0, r)}$ which are measurable by Fubini's theorem. Thus, we obtain that
	\begin{equation}\label{eq:PI_X_dir}
		\int_{X \times Y} \abs{\tilde{u}(x, y) - \tilde{u}_X(y)}^p \, d\mu_{XY}(x,y) 
		\le C_X^p \int_{X \times Y} \tilde{g}^p(x,y) \, d\mu_{XY}(x,y).
	\end{equation}
	
We now aim to show that $\tilde{u}_X$ is constant $\mu_Y$-a.e.\ in $Y$, as then the claim follows 
from~\eqref{eq:PI_X_dir} by selecting $c_u$ to be the constant value of $\tilde{u}_X$.  
Since $Y$ is connected, and non-empty open subsets of $Y$ have positive $\mu_Y$-measure,
it suffices to show that for each $y_0\in Y$ there is some $r_{y_0}>0$ and a real number $c_{y_0}$ such that 
$u_X(y)=c_{y_0}$ $\mu_Y$-a.e.~in $B_Y(y_0,r_{y_0})$.
For this purpose, 
let $y_0 \in Y$ and 
let $B \subset Y$ be a ball centered at $y_0$ where the 
$\infty$-weak local $(p,p)$-Poincar\'e inequality of $Y$ holds. 
	
Analogously to before, we note that for $\mu_X$-a.e.\ $x \in X$, we have by 
Lemma~\ref{lem:warped_prod_horiz_and_vert_UGs} that $\tilde{g}_x \in L^p(Y,\mu_Y)$, $\tilde{u}_x$ is integrable 
over balls with respect to $\mu_Y$, and $\varphi(x) \tilde{g}_x$ is an upper gradient of $\tilde{u}_x$. It follows from the $\infty$-weak 
$(p,p)$-Poincar\'e inequality on $B$ that 
	\begin{align*}
		\int_B \abs{\tilde{u}(x,y) - (\tilde{u}_x)_B}^p \, d\mu_Y(y) 
		&\le C_Y^p \int_{Y} (\varphi(x) \tilde{g}(x,y))^p \, d\mu_Y(y)\\
		&= C_Y^p \varphi^p(x) \int_Y \tilde{g}^p(x,y) \, d\mu_Y(y)
	\end{align*}
	for $\mu_X$-a.e.\ $x \in X$. Thus, if we denote $\tilde{u}_B(x) = (\tilde{u}_x)_B$ for $x \in X$, Tonelli's theorem 
and the fact that $\mu_X(\varphi^{-1}(\{0\})) = 0$ then yields
	\[
		\int_{X \times B} \abs{\tilde{u}(x,y) - \tilde{u}_B(x)}^p \varphi^{-p}(x) \, d\mu_X(x) d\mu_Y(y) 
		\le C_Y^p \int_{X \times Y} \tilde{g}^p(x,y) \, d\mu_X(x) d\mu_Y(y).
	\]
That is, $\tilde{u}(x,y) - \tilde{u}_B(x)$ is in $L^p(X \times B, (\varphi^{-p}\mu_X) \times \mu_Y)$. 
However, since $\varphi^{-p} \in L^\infty(E, \mu_X)$ by assumption, it follows from \eqref{eq:PI_X_dir} that 
$\tilde{u}(x,y) - \tilde{u}_X(y)$ is in $L^p(E \times B, (\varphi^{-p}\mu_X) \times \mu_Y)$. Thus, 
the function $(x,y)\mapsto \tilde{u}_B(x) - \tilde{u}_X(y)$ is in $L^p(E \times B, (\varphi^{-p}\mu_X) \times \mu_Y)$.
	
We claim this is only possible if $\tilde{u}_X$ is $\mu_Y$-a.e.\ constant in $B$. Indeed, suppose 
towards contradiction that $\tilde{u}_X$ is not $\mu_Y$-a.e.\ constant in $B$. Since $\mu_Y(B) > 0$, it follows that
	\[
		\inf_{c \in \R} \int_B \abs{\tilde{u}_X(y) - c}^p d\mu_Y(y) > 0.
	\]
	However, now by the fact that $\varphi^{-p} \notin L^1(E, \mu_X)$, we get by Tonelli's theorem that
	\begin{multline*}
		\int_{E \times B} \abs{\tilde{u}_B(x) - \tilde{u}_X(y)}^p \varphi^{-p}(x) d\mu_X(x) d\mu_Y(y)\\
		\ge \left( \inf_{c \in \R} \int_B \abs{\tilde{u}_X(y) - c}^p d \mu_Y(y) \right) 
		\int_E \varphi^{-p}(x) d\mu_X(x) 
		= \infty.
	\end{multline*}
	This contradicts $(x,y)\mapsto \tilde{u}_B(x) - \tilde{u}_X(y)$ 
being in $L^p(E \times B, (\varphi^{-p}\mu_X) \times \mu_Y)$. 
Thus, we must have that $\tilde{u}_X$ is $\mu_Y$-a.e.\ constant on $B$, completing the proof.
\end{proof}

\section{Proof of Theorem~\ref{thm:solid_hyp_filling_is_PI}}\label{Sec:Theorem-1.3}

Theorem~\ref{thm:solid_hyp_filling_is_PI} is a special case of the following theorem; the goal of this section is to prove this
general version.

\begin{theorem}\label{thm:solid_hyp_filling_is_PI-gen}
Let $(Y, d_{Y}, \mu_{Y})$ be a length metric measure space 
which satisfies Standing Assumptions~\ref{stassm:extra_assumptions},
and let $\alpha,\beta$ be two fixed positive real numbers. Suppose also that
$Y$ is an $\infty$-weak local $(p,p)$-Poincar\'e space for some $p$ with $1\le p<\infty$.
\begin{enumerate}
\item[(a)] If $p\le \beta/\alpha$, then the metric measure space $\H^{\circ}_{\alpha, \beta}(Y)$ is a $(p,p)$-Sobolev--Poincar\'e 
space with constant $C = C(\beta, p)$,  and consequently
\begin{equation}\label{eq:D_is_N+R_Hcirc}
	D^{1,p}(\H^{\circ}_{\alpha, \beta}(Y)) = N^{1,p}(\H^{\circ}_{\alpha, \beta}(Y)) + \R.
\end{equation}
\item[(b)] If $p \le \beta/\alpha$ and $\mu_{Y}(Y) < \infty$, then $\H_{\alpha, \beta}(Y)$ is a
$(p,p)$-Sobolev--Poincar\'e space with constant $C = C(\beta, p)$, and therefore
\begin{equation}\label{eq:D_is_N+R_H}
	D^{1,p}(\H_{\alpha, \beta}(Y)) = N^{1,p}(\H_{\alpha, \beta}(Y)) + \R.
\end{equation}
\item[(c)] If $p > \beta/\alpha$ and $Y$ has at least two points, 
then \eqref{eq:D_is_N+R_Hcirc} fails to hold, and if $p > \beta/\alpha$ and $\mu_{Y}(Y) < \infty$ 
then \eqref{eq:D_is_N+R_H} fails to hold.
\end{enumerate}
\end{theorem}

We first
give a simple proof covering part~(a) 
to demonstrate the main ideas of the argument. 
In the setting of Theorem~\ref{thm:warped_product_PI-ext},
this corresponds to the specific choice of 
$X=[0,\infty)$ equipped with the Euclidean metric $d_{\text{Eucl}}$ and 
$\pip(t)=e^{\alpha t}$, with $\mu_X$ given by $d\mu_X=e^{\beta t}\, dt$. 
The argument is based on the one found in~\cite[Lemma 5.6]{Str}.

\begin{proof}[Proof of part~(a) of Theorem~\ref{thm:solid_hyp_filling_is_PI-gen}]\label{simple-case}
Let $X$, $\pip$, $\mu_X$ be as above, with $\alpha>0$ and $\beta>0$.
To apply Theorem~\ref{thm:warped_product_PI-ext}, we need only to
verify is that $(X,d_{\text{Eucl}},\mu_X)$ is a $(p,p)$-Sobolev--Poincar\'e space. To do so, 
we argue as follows. Suppose that $u\in D^{1,p}(X,d_{\text{Eucl}},\mu_X)$, and let $g\in L^p(X,\mu_X)$ be an upper gradient
of $u$ on $X$. Then for $0\le x<y$ we have by H\"older's inequality that 
\begin{align}\label{eq:tag1}
|u(y)-u(x)|\le \int_x^y g(t)\, dt
  &\le \left(\int_x^y e^{\beta t}\, g(t)^p\, dt\right)^{1/p}\, \left(\int_x^y e^{-\beta t/(p-1)}\, dt\right)^{(p-1)/p}\\
  &\le \left(\frac{p-1}{\beta}\right)^{(p-1)/p}\, \left(\int_x^y g(t)^p\, d\mu_X(t)\right)^{1/p},\notag
\end{align}
and as $g\in L^p(X,\mu_X)$, it follows that $\lim_{x,y\to\infty}|u(y)-u(x)|=0$. That is, the limit $\lim_{y\to\infty}u(y)=:c_u$
exists as a real number. By replacing $u$ with $u-c_u$ if necessary, we may assume without loss of generality
that $\lim_{y\to\infty}u(y)=0$. Then we have from~\eqref{eq:tag1} that
\[
|u(x)|\le \left(\frac{p-1}{\beta}\right)^{(p-1)/p}\, \left(\int_x^y e^{\beta t}\, g(t)^p\, dt\right)^{1/p}\, \left(e^{-\beta x/(p-1)}\right)^{(p-1)/p}, 
\]
and so 
\begin{align}\label{eq:p-exp}
|u(x)|^p\, e^{\beta x}\le \left(\frac{p-1}{\beta}\right)^{(p-1)} \int_x^\infty g(t)^p\, d\mu_X(t).
\end{align}

On the other hand, using the fact that as $u\in N^{1,p}_{loc}([0,\infty), d_{\text{Eucl}},\mathcal{L}^1)$ with $\mathcal{L}^1$ the 
Lebesgue measure on $[0,\infty)$, we know that $u$ is locally absolutely continuous on $[0,\infty)$ and satisfies the Fundamental
Theorem of Calculus, as does $(u^2)^{p/2}$. Moreover, $\tfrac{d}{dx}e^{\beta x}=\beta e^{\beta x}$, and so, using
integration by parts, we obtain for $y>0$ that
\begin{align*}
\int_0^y|u(x)|^p\, d\mu_X(x)&=\frac{1}{\beta}\, \int_0^y(u(x)^2)^{p/2}\, \left(\frac{d}{dx} e^{\beta x}\right)\, dx\\
&=\frac{1}{\beta}\left[ |u(y)|^p\, e^{\beta y}-|u(0)|^p\, -p\int_0^y|u(x)|^{p-2}\, u(x)\, \left(\frac{d}{dx}u(x)\right)\, e^{\beta x}\, dx\right]\\
&\le \frac{1}{\beta}\left[ |u(y)|^p\, e^{\beta y}-|u(0)|^p+p \int_0^y|u(x)|^{p-1}\, g(x)\, d\mu_X(x)\right].
\end{align*}
In the above computation, we have used the fact that $|u^\prime(x)|\le g(x)$ for almost every (and hence $\mu_X$-a.e.) $x\in [0,\infty)$
because $g$ is an upper gradient of $u$.

When $p>1$, thanks to Young's inequality, we now obtain for $\eps>0$ that
\begin{align*}
\int_0^y|u(x)|^p\, d\mu_X(x)
&\le \frac{1}{\beta}\left[ |u(y)|^p\, e^{\beta y}-|u(0)|^p
    +p\int_0^y\left[\frac{(p-1)\eps\, |u(x)|^p}{p}+\frac{g(x)^p}{p\, \eps^{p-1}}\right]\, d\mu_X(x)\right].
\end{align*}
Choosing $\eps$ small enough so that $\tfrac{p-1}{\beta}\, \eps=\tfrac12$, we get 
\begin{align*}
\frac12\int_0^y|u(x)|^p\, d\mu_X(x)
&\le \frac{1}{\beta}\left[ |u(y)|^p\, e^{\beta y}\right]
    +\frac{1}{\beta\, \eps^{p-1}}\int_0^y g(x)^p\, d\mu_X(x).
\end{align*}
Applying~\eqref{eq:p-exp} with $x$ replaced by $y$ and noting that $g\in L^p(X,\mu_X)$, we must
also have $\lim_{y\to\infty}|u(y)|^p e^{\beta y}=0$. Thus, we get
\begin{align*}
\frac12\lim_{y\to\infty}\, \int_0^y|u(x)|^p\, d\mu_X(x)&\le 
\frac{1}{\beta\, \eps^{p-1}}\, \int_Xg^p\, d\mu_X
=\frac{1}{\beta}\, \left(\frac{p-1}{\beta}\right)^{p-1}\, \int_Xg^p\, d\mu_X,
\end{align*}
 that is, 
\[
\int_X|u|^p\, d\mu_X\le \frac{2}{\beta}\, \left(\frac{p-1}{\beta}\right)^{p-1}\, \int_Xg^p\, d\mu_X.\qedhere
\]
\end{proof}

Now we provide a more general proof that covers both parts~(a) and~(b) of Theorem~\ref{thm:solid_hyp_filling_is_PI-gen}.
Since this more general argument involves functions that might vanish on some subset of $[0,\infty)$, the proof is
more complicated. The most significant part of the argument is to verify the following 
Lemma~\ref{lem:weighted_reals_infty_meas_pp-PI}. Note that Lemma \ref{lem:psi_conds} below 
will also be used in the proof of Theorem~\ref{thm:solid_filling_is_Gromov_hyp}.

\begin{lemma}\label{lem:weighted_reals_infty_meas_pp-PI}
	Let $a \in \R$, $\alpha>0$, $p \in [1, \infty)$ and let
	$\psi\in C([a,\infty))\cap C^1((a,\infty))$ be a non-negative non-constant function 
	such that 
	\begin{equation}\label{eq:psi_deriv_condition}
		\psi(t) \le \alpha^{-1} \psi'(t)
	\end{equation}
	for all $t \in (a, \infty)$.  
	Let $X$ be the space $[a, \infty)$ equipped with the usual Euclidean metric and the weighted 
	Lebesgue measure $d \mu_X(t) = \psi(t) \, dt$. Then $\mu_{X}(X)=\infty$ and $X$ is a $(p,p)$-Sobolev--Poincar\'e space, 
	with constant $C = C(\alpha, p)$.
\end{lemma}

\begin{remark}\label{rem:deriv_condition_for_sinh_and_exp}
	We note that condition \eqref{eq:psi_deriv_condition} is satisfied on $[0, \infty)$ by the 
	functions $t \mapsto e^{\alpha t}$ and $t \mapsto \sinh^\alpha(t)$. Indeed, 
	if $\psi(t) = e^{\alpha t}$, then \eqref{eq:psi_deriv_condition} holds with equality, 
	and if instead $\psi(t) = \sinh^{\alpha}(t)$, then $\psi'(t) = \alpha \sinh^{\alpha - 1}(t) \cosh(t)$, 
	and \eqref{eq:psi_deriv_condition} follows for $t \ge 0$ via the inequality $\sinh(t) \le \cosh(t)$.
\end{remark}

Before beginning the proof proper, we point out a few relevant properties of solutions of \eqref{eq:psi_deriv_condition}.

\begin{lemma}\label{lem:psi_conds}
	Let $a \in \R$, $\alpha>0$, and $1\le p<\infty$. Suppose that 
	$\psi\in C([a,\infty))\cap C^1((a,\infty))$ is a non-negative function that
	satisfies \eqref{eq:psi_deriv_condition} for all $t \in (a, \infty)$. Then the following properties hold.
	\begin{enumerate}
		\item \label{enum:psi_nondecr} The function $\psi$ is non-decreasing.
		\item \label{enum:psi_exp_estimate} If $\psi(b) > 0$ for some $b\ge a$, then for every $r \in [b, \infty)$, we have
		\[
		\psi(r) \ge \psi(b) e^{\alpha(r-b)}.
		\]
		\item \label{enum:psi_minus_pow_int_est} If $\psi(b) > 0$ for some $b\ge a$, 
		then for all $r \in [b, \infty)$ and $s \in (0, \infty)$, we have
		\[
		\int_r^\infty \psi^{-s}(t) \, dt \le \frac{1}{s\alpha} \psi^{-s}(r).
		\]
		\item \label{enum:psi_minimizer_bound} Let $A \in [0, \infty)$, let $b \ge a$, and let $F\colon [a,b]\to\R$ be given by
		$F(t)=A\psi(t)-2t$. Suppose that $F$ attains its minimum at some $\tau>a$. Then,
		\[
			A \psi(\tau) \le A \alpha^{-1} \psi'(\tau) \le 2\alpha^{-1}.
		\]
	\end{enumerate} 
\end{lemma}
\begin{proof}
	Since $\psi \ge 0$, \eqref{eq:psi_deriv_condition} implies that $\psi' \ge 0$, and hence \eqref{enum:psi_nondecr} follows. Moreover, if 
	$\psi(b) > 0$ and consequently $\psi(t) > 0$ for all $t \in [b, \infty)$, then \eqref{eq:psi_deriv_condition} also implies that	
	\[
	\frac{d}{dt} \ln \psi(t) = \frac{\psi'(t)}{\psi(t)} \ge \alpha
	\]
	for all $t \in (b, \infty)$, which by integration over $[b, r]$ implies \eqref{enum:psi_exp_estimate}.
	
	To prove the third claim,  
	suppose that $\psi(b) > 0$, and let $r \in [b, \infty)$ and $s \in (0, \infty)$. For all $t \in (b, \infty)$, we have by \eqref{eq:psi_deriv_condition} that
	\[
	-\frac{d}{dt} \psi^{-s}(t) = s \psi^{-s-1}(t) \psi'(t) \ge s\alpha \psi^{-s}(t).
	\]
	By integrating this inequality over $[r, \infty)$, and noting that $\lim_{t \to \infty} \psi^{-s}(t) = 0$ by \eqref{enum:psi_exp_estimate}, we obtain
	\[
	\int_r^\infty \psi^{-s}(t) \, dt
	\le \frac{1}{s\alpha} \int_r^\infty \left( - \frac{d}{dt} \psi^{-s}(t) \right) \, dt
	= \frac{1}{s\alpha} \psi^{-s}(r), 
	\]
	completing the proof of \eqref{enum:psi_minus_pow_int_est}.
	
	We now verify the final claim.  
	Since $\tau > a$, $F$ is differentiable at $\tau$ with $F'(\tau) = A\psi'(\tau) - 2$. Since $F(\tau) \le F(t)$ for all $a \le t < \tau$ by the fact that $\tau$ minimizes $F$, we must have $F'(\tau) \le 0$, i.e.\ $A\psi'(\tau) \le 2$. By combining this with \eqref{eq:psi_deriv_condition}, the claim of part \eqref{enum:psi_minimizer_bound} follows.
\end{proof}

\begin{proof}[Proof of Lemma \ref{lem:weighted_reals_infty_meas_pp-PI}]
Since $\psi$ is not identically zero, there exists $b \in [a, \infty)$ for which $\psi(b) > 0$. It follows by Lemma \ref{lem:psi_conds} \eqref{enum:psi_nondecr} that $\psi \ge \psi(b) > 0$ on $[b, \infty)$, and thus $[b, \infty)$ has infinite $\mu_X$-measure, proving that $\mu_X(X)=\infty$. 
This leaves only the proof of the $(p,p)$--Sobolev--Poincar\'e inequality. 
	
	For the $(p,p)$-Sobolev--Poincar\'e inequality on $X$, let $u \in D^{1,p}(X, \mu_{X})$, and let $g \in L^p(X,\mu_{X})$ be an upper gradient of $u$. We first consider the case $\psi(a) > 0$. In this case, we have $\psi^{-1/p} \in L^q([a, \infty), \mathcal{L}^{1})$ for every $q \in [1, \infty]$ by Lemma \ref{lem:psi_conds} \eqref{enum:psi_exp_estimate}. Thus, we may estimate using H\"older's inequality that
	\[
	\int_a^\infty g(t) \, dt = \int_a^\infty g(t) \psi^\frac{1}{p}(t) \psi^{-\frac{1}{p}}(t) \, dt
	\le \norm{g}_{L^p(X, \mu_{X})} \norm{\psi^{-1/p}}_{L^{p^*}([a, \infty), \mathcal{L}^{1})} < \infty,
	\]
	where $p^* = p/(p-1)$ is the H\"older conjugate of $p$. This proves that $u \in D^{1,1}([a, \infty), d_{\text{Eucl}}, \mathcal{L}^{1})$, which in turn implies that $u$ is absolutely continuous and $c_u := \lim_{t \to \infty} u(t)$ exists and is finite.
	
	We then let $\tilde{u} = u - c_u$, with the objective of showing that $\tilde{u} \in L^p(X, \mu_{X})$. By absolute continuity, we may write
	\[
	\tilde{u}(r) = \int_r^\infty u'(t) \, dt.
	\]
	We then claim that
	\begin{equation}\label{eq:ut_psi_pointwise_estimate}
		\abs{\tilde{u}(r)}^p \psi(r) \le \frac{(p-1)^{p-1}}{\alpha^{p-1}} \int_r^\infty \abs{u'(t)}^p \psi(t) \, dt
	\end{equation}
	for a.e.\ $r \in [a, \infty)$, where we interpret $(p-1)^{p-1} = 1$ if $p = 1$. Indeed, if $p > 1$, then we use H\"older's inequality and Lemma \ref{lem:psi_conds} \eqref{enum:psi_minus_pow_int_est} to estimate that
	\begin{align*}
		\abs{\tilde{u}(r)}^p 
		&\le \left( \int_r^\infty \abs{u'(t)} \psi^{\frac{1}{p}}(t) 
		\psi^{-\frac{1}{p}}(t) \, dt \right)^p\\
		&\le \left( \int_r^\infty \psi^{-\frac{1}{p-1}}(t) \, dt \right)^{p-1} 
		\int_r^\infty \abs{u'(t)}^p \psi(t) \, dt\\
		&\le \left( \frac{(p-1)}{\alpha} \psi^{-\frac{1}{p-1}}(r) \right)^{p-1} 
		\int_r^\infty \abs{u'(t)}^p \psi(t) \, dt,
	\end{align*}
	which re-arranges to \eqref{eq:ut_psi_pointwise_estimate}. In the final remaining case $p = 1$, the fact that $\psi$ is non-decreasing on $[a, \infty)$ by Lemma \ref{lem:psi_conds} \eqref{enum:psi_nondecr} yields that
	\[
	\abs{\tilde{u}(r)} \psi(r) \le \int_r^\infty \abs{u'(t)} \psi(r) \, dt \le \int_r^\infty \abs{u'(t)} \psi(t) \, dt,
	\]
	completing the proof of \eqref{eq:ut_psi_pointwise_estimate}.
	
	Next, let $b \in [a, \infty)$. We observe that by \eqref{eq:psi_deriv_condition} and integration by parts, we have
	\begin{align*}
		&\int_a^b \abs{\tilde{u}(t)}^p \psi(t) \, dt
		\le \frac{1}{\alpha} \int_a^b \abs{\tilde{u}(t)}^p \psi'(t) \, dt\\
		&\qquad\le \frac{1}{\alpha} \left( \abs{\tilde{u}(b)}^p \psi(b)  -  \abs{\tilde{u}(a)}^p \psi(a)
		+ p \int_a^b \abs{\tilde{u}(t)}^{p-1} \abs{u'(t)} \psi(t) \, dt \right).
	\end{align*}
	By \eqref{eq:ut_psi_pointwise_estimate}, we have
	\[
	\abs{\tilde{u}(b)}^p \psi(b) \le \frac{(p-1)^{p-1}}{\alpha^{p-1}} \norm{u'}_{L^p(X, \mu_{X})}^p.
	\]
	On the other hand, we may apply Young's inequality for products on the final integral, obtaining
	\begin{align*}
		& \int_a^b \abs{\tilde{u}(t)}^{p-1} \psi^{\frac{p-1}{p}}(t) \cdot 
		\frac{p}{\alpha} \abs{u'(t)} \psi^\frac{1}{p}(t) \, dt\\
		&\qquad\le \frac{p-1}{p} \int_a^b \abs{\tilde{u}(t)}^p \psi(t) \, dt
		+ \frac{p^{p-1}}{\alpha^p} \norm{u'}_{L^p(X, \mu_{X})}^p.
	\end{align*}
	Since $\tilde{u}$ and $\psi$ are both continuous, the integral 
	of $|\tilde{u}|^p\psi$ over $[a,b]$ is finite, and we may subtract it on both sides, obtaining the estimate
	\[
	\int_a^b \abs{\tilde{u}(t)}^p \psi(t) \, dt \le \frac{p(p-1)^{p-1} + p^p}{\alpha^p} \norm{u'}_{L^p(X, \mu_{X})}^p.
	\]
	Letting $b \to \infty$ then yields the desired estimate $\norm{u - c_u}_{L^p(X, \mu_{X})}  \le C(\alpha, p) \norm{u'}_{L^p(X, \mu_{X})}$, completing the proof in the case where $\psi(a) > 0$.
	
	It remains to consider the case where $\psi(a) = 0$. In this case, since $\psi$ is continuous and non-decreasing by Lemma \ref{lem:psi_conds} \eqref{enum:psi_nondecr}, there exists a $b \in [a, \infty)$ for which $\psi \equiv 0$ on $[a, b]$ and $\psi > 0$ on $(b, \infty)$. Now, we apply the previous case on the sub-intervals $[b + \eps, \infty)$ for all $\eps > 0$; since the subtracted constant $c_u = \lim_{r \to \infty} u(r)$ is the same for all of these sub-intervals, and since the constant $C(\alpha, p)$ is independent of the starting point of the interval, we obtain that
	\begin{align*}
		\norm{u - c_u}_{L^p(X, \mu_{X})} &= \lim_{\eps \to 0} \norm{u - c_u}_{L^p([b+\eps, \infty), \mu_X)}\\
		&\le \lim_{\eps \to 0} C(\alpha, p) \norm{g}_{L^p([b+\eps, \infty), \mu_X)}\\
		&= C(\alpha, p) \norm{g}_{L^p(X, \mu_{X})},
	\end{align*}
	completing the proof in the general case.
\end{proof}

With the pieces we have assembled so far, we are now ready to prove parts~(a) and~(b) of 
Theorem~\ref{thm:solid_hyp_filling_is_PI-gen}.

\begin{proof}[Proof of Theorem \ref{thm:solid_hyp_filling_is_PI-gen} Parts (a) and (b)]
	We have $\H^{\circ}_{\alpha, \beta}(Y) = X_\beta^{\exp} \times_{\varphi_\alpha} Y$ with 
$\varphi_\alpha(t) =  e^{\alpha t}$ and $X_\beta^{\exp} = ([0, \infty), d_{\text{Eucl}}, e^{\beta t} \, dt)$, and 
similarly $\H_{\alpha, \beta}(Y) = X^{\sinh}_\beta \times_{\psi_\alpha} Y$ with 
$\psi_\alpha(t) = \sinh^\alpha(t)$ and $X_\beta^{\sinh} = ([0, \infty), d_{\text{Eucl}}, \sinh^\beta(t) \, dt)$. 
Note that $\varphi_{\alpha}$ satisfies condition (i) of Theorem \ref{thm:warped_product_PI-ext}, 
and when $\mu_{Y}(Y)<\infty$, the function $\psi_{\alpha}$ satisfies condition (ii) of 
Theorem~\ref{thm:warped_product_PI-ext}. By Lemma~\ref{lem:weighted_reals_infty_meas_pp-PI} and 
Remark~\ref{rem:deriv_condition_for_sinh_and_exp}, $X_\beta^{\exp}$ 
and $X_\beta^{\sinh}$ are infinite measured $(p,p)$-Sobolev--Poincar\'e spaces 
for all $p \in [1, \infty)$. By assumption, $Y$ is an $\infty$-weak $(p,p)$-Poincar\'e space for our given 
$p \in [1, \infty)$. Moreover, if $p \le \beta/\alpha$, then $\varphi_\alpha^{-p}$ satisfies the final hypothesis of Theorem~\ref{thm:warped_product_PI-ext} with $E = [0, \infty)$ and $\psi_\alpha^{-p}$ 
satisfies the final hypothesis with $E = [1, \infty)$. Thus, the claim follows from Theorem~\ref{thm:warped_product_PI-ext}.
\end{proof}

It remains to prove the last part of Theorem \ref{thm:solid_hyp_filling_is_PI-gen}.

\begin{proof}[Proof of Theorem \ref{thm:solid_hyp_filling_is_PI-gen} Part (c)]
 We prove the result for $\H_{\alpha, \beta}(Y)$; an essentially identical proof works for $\H_{\alpha, \beta}^{\circ}(Y)$. We note that the assumption $\mu_{Y}(Y) < \infty$ is just used to ensure that $\H_{\alpha, \beta}(Y)$ is a metric measure space as in Standing Assumptions \ref{stassm:extra_assumptions}, which is the generality at which we consider Dirichlet and Newtonian Sobolev spaces.
	
	By our assumptions, there is a ball $B_Y( y_0, r)$ in $Y$ for which $\mu_{Y}(B_Y(y_0, r/2)) > 0$ and $\mu_{Y}(Y\setminus B_Y( y_0, r)) > 0$. We define functions
	\begin{align*}
		u_{\R} &\colon [0, \infty) \to [0, \infty),&  u_{\R}(t) &= \min\{ 1, \max \{ 0, t-1 \}\},\\
		u_Y &\colon Y \to [0, \infty), & u_Y(y) &= \min\{ r/2, \max \{ 0, r-d_Y(y, y_0) \}\},
	\end{align*}
	and define $\tilde{u} \colon [0, \infty) \times Y \to \R$ by 
$\tilde{u}(t,y) = u_{\R}(t) u_Y(y)$. We note that $u_{\R}$ and $u_{Y}$ are 1-Lipschitz, and $u_{\R}$ vanishes in a neighborhood of $0$. Thus, by Lemma \ref{lem:easy_tensorization}, $\tilde{u}$ descends to a function $u \colon \H_{\alpha, \beta}(Y) \to \R$ with $u \circ P_{\psi_\alpha} = \tilde{u}$, and $u$ has an upper gradient $g \colon \H_{\alpha, \beta}(Y) \to \R$ given for $(t, y) \in [0, \infty) \times Y$ by
	\[
	g(P_{\psi_\alpha}(t, y)) = \abs{u_{Y} (y)} \Lip_{u_{\R}}(t) + \frac{\abs{u_{\R}(t)}}{\sinh^\alpha(t)} \Lip_{u_{Y}}(y)
	\]
	when $t \ne 0$, and $g(P_{\psi_\alpha}(t, y)) = 0$ when $t = 0$. We denote $\tilde{g} = g \circ P_{\psi_\alpha}$.
	
	We note that $\abs{u_Y} \le r/2$, $\abs{u_{\R}} \le 1$, $\Lip_{u_{Y}} \le 1$, $\Lip_{u_{\R}} \le 1$, $u_{Y} \equiv 0$ outside $B_Y( y_0, r)$, $u_{\R} \equiv 0$ outside $(1, \infty)$, $\Lip_{u_{Y}} \equiv 0$ outside $\overline{B_Y}( y_0, r)$, and $\Lip_{u_{\R}} \equiv 0$ outside $[1,2]$. Recalling that $(a+b)^p \le 2^{p-1}(a^p + b^p)$ for all $a,b \ge 0$, we thus estimate that
	\begin{align*}
		&\int_Y \int_0^\infty \tilde{g}^p(t, y) \sinh^\beta(t) \, dt \, d\mu_{Y}(y)\\
		&\qquad \le \frac{r^p}{2} \sinh^\beta(2) \mu_{Y}(B_Y(y_0, r)) + 2^{p-1} \mu_{Y}(\overline{B_Y}(y_0, r)) \int_1^\infty \sinh^{\beta - p\alpha}(t) \, dt. 
	\end{align*}
	Noting that balls have finite measure in $Y$, the right hand side is finite if $p > \beta/\alpha$. 
	Thus, for all $p \in (\beta/\alpha, \infty)$, we have $u \in D^{1,p}(\H_{\alpha, \beta}(Y))$. However, 
	it is impossible that $u - c$ is $L^p$-integrable over $\H_{\alpha, \beta}(Y)$ for any $c \in \R$, 
	since $\tilde{u}$ is identically $0$ in $[1, \infty) \times (Y \setminus B_Y(y_0, r))$ and 
	identically $r/2$ in $[1, \infty) \times B_Y( y_0, r/2)$, both of which have infinite 
	$\nu_\beta$-measure. Thus, $u \notin N^{1,p}(\H_{\alpha, \beta}(Y)) + \R$, completing the proof for $\H_{\alpha, \beta}(Y)$. 
\end{proof}

\section{Proof of Theorem \ref{thm:solid_filling_is_Gromov_hyp}}

Theorem~\ref{thm:solid_filling_is_Gromov_hyp} is a special case of the following more general theorem, with the 
choice of $\psi(t)=\pip_\alpha(t)=e^{\alpha t}$ or the choice of $\psi(t)=\psi_\alpha(t)=\sinh^\alpha(t)$.

\begin{theorem}\label{thm:solid_filling_is_Gromov_hyp-II}
Suppose that $(Y, d_Y)$ is a length space and that the following are satisfied:
\begin{enumerate}
\item $\alpha>0$ and $\psi\in C([0,\infty))\cap C^1((0,\infty))$ is a non-constant non-negative function  
such that \eqref{eq:psi_deriv_condition} holds at every $t \in (0, \infty)$;
\item  if $Y$ is unbounded, then $\psi(0) = 0$.
\end{enumerate}
	Then $[0,\infty)\times_\psi Y$ is Gromov $\delta$-hyperbolic with $\delta=\delta(\alpha,\psi(0)\diam(Y))$ if $\psi(0)\ne 0$
	and $\delta=\delta(\alpha)$ if $\psi(0)=0$. 
 	If additionally $Y$ is complete and bounded 
  	and there exists a positive real number $C$ such that $\psi(t) \le C e^{\alpha t}$ for all $t \in [0, \infty)$,
	then for $\eps > 0$ sufficiently small the Gromov 
	boundary $\partial_G ([0,\infty)\times_\psi Y)$ equipped with its
	induced metric $d_\eps$ 
 is quasisymmetrically homeomorphic to $Y$.
\end{theorem}

To prove Theorem~\ref{thm:solid_filling_is_Gromov_hyp-II}, we prove the
theorem in the case where the norm $\norm{\cdot}$ on $\R^2$, used to define the warped product
structure on $[0,\infty)\times_\psi Y$, is the $\ell^1$-norm. 
We note that by Lemma~\ref{lem:Gromov_hyp_bilipschitz}, Gromov hyperbolicity is a bilipschitz invariant on length spaces,
and by Lemma~\ref{lem:gromov_hyp_QS_extension}, two 
bilipschitz equivalent
Gromov hyperbolic length spaces have quasisymmetrically
equivalent Gromov boundaries. 
By Corollary~\ref{cor:warped_prod_length_space},  $[0,\infty)\times_\psi Y$ is a length space for every 
choice of coordinate-increasing unitary norm $\norm{\cdot}$ which induces the warped product structure. 
Moreover, as a consequence of Lemma~\ref{lem:unitary_coord_incr_norms}, if we change the 
coordinate-increasing unitary norm $\norm{\cdot}$, this only amounts to a $2$-bilipschitz change in the metric of $[0,\infty)\times_\psi Y$.

{We begin by introducing some helpful terminology. 
Let $(Y, d_Y)$ be a length space, and let $I$ denote the Euclidean interval $[0,\infty)$.
 Let $\gamma$ 
 be a curve in $I\times Y$, with initial point and end point denoted $(t_1,x_1)$ and $(t_2,x_2)$ respectively.   
 We say that $\gamma$ is 
 a \emph{\pa-curve} if it is a concatenation of curves 
 $\gamma^{\desc}$, $\gamma^{\hor}$, and $\gamma^{\asc}$ in $I\times Y$ such that
		\begin{enumerate}
			\item[(i)] $\gamma^{\desc} = \bigl(\gamma^{\desc}_{{{I}}}, \gamma^{\desc}_Y \bigr)$, 
			where $\gamma^{\desc}_{{{I}}}$ is a decreasing path from $t_1$ to some 
			$\tau \in [0, \min\{t_1,t_2\}]$ and $\gamma^{\desc}_Y$ is constant;
			\item[(ii)] $\gamma^{\hor} = \bigl(\gamma^{\hor}_{{I}}, \gamma^{\hor}_Y\bigr)$, 
			where $\gamma^{\hor}_I(t)=\tau$ for all $t$ 
			and $\gamma^{\hor}_Y$ is a path from $x_1$ to $x_2$;
			\item[(iii)] $\gamma^\asc = \bigl(\gamma^{\asc}_{{{I}}}, \gamma^{\asc}_{{{Y}}}\bigr)$, 
			where $\gamma^{\asc}_{{{I}}}$ is an increasing path from $\tau$ to $t_2$ and $\gamma^{\asc}_{{{Y}}}$ is constant.
		\end{enumerate}
		We call $\gamma^{\desc}$, $\gamma^{\hor}$, and $\gamma^{\asc}$ respectively the \emph{descending part}, \emph{horizontal part}, and \emph{ascending part of the \pa-curve $\gamma$}. We also call 
	the constant value $\tau$ of $\gamma^{\hor}_{{I}}$ the \emph{horizontal level} of 
		 $\gamma$. See Figure~\ref{fig:path_in_A1} below for an illustration of a \pa-curve.  
		\begin{figure}[h]
			\begin{tikzpicture}[yscale = 0.9]
				\draw[->] (0,0) -- (0,5) node[anchor=east] {$[0, \infty)$};
				\draw (3,0) ellipse (3 and 1);
				\draw (6,0) node[anchor=west] {$Y$};
				\draw[-] (2,3.5) node[anchor=south east]{$(t_1,x_1)$} -- node[midway, anchor=east] {$\gamma^\desc$}
				(2,2) node[anchor=east] {$(\tau,x_1)$} to[out=0, in=190] node[midway, anchor=north] {$\gamma^\hor$} 
				(4,1.5) node[anchor=west] {$(\tau,x_2)$} -- node[midway, anchor=west] {$\gamma^\asc$} 
				(4,4.5) node[anchor=west] {$(t_2,x_2)$};
			\end{tikzpicture}
			\caption{\small An illustration of a \pa-curve $\gamma$.}
			\label{fig:path_in_A1}
		\end{figure}
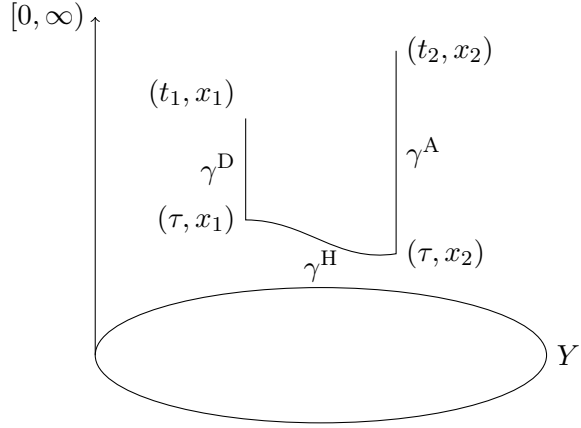
		
		Suppose that $\psi \colon I \to [0, \infty)$ is continuous and that $\gamma$ is an absolutely continuous \pa-curve. 
		By using Proposition \ref{prop:lengths_match_warped}, we have
		\begin{align}
			\label{eq:len_desc_part}
			\len_{d_\psi}(\gamma^\desc) &= \len_{d_{\text{Eucl}}}(\gamma^\desc_{{I}}) = t_1 - \tau\\
			\label{eq:len_hor_part}
			\len_{d_\psi}(\gamma^\hor) &= \psi(\tau)\len_{d_Y} (\gamma^\hor_Y)\\
			\label{eq:len_asc_part}
			\len_{d_\psi}(\gamma^\asc) &= \len_{d_{\text{Eucl}}}(\gamma^\asc_{{I}}) = t_2 - \tau.
		\end{align}\label{eq:pa-curve-length}
		Thus, the $d_\psi$-length of an absolutely continuous \pa-curve $\gamma$ is given by
		\begin{equation}\label{eq:pa-curve_length-total}
			\len_{d_\psi}(\gamma) = t_1 + t_2 - 2\tau + \psi(\tau)\len_{d_Y} (\gamma^\hor_Y).
		\end{equation}
		
		Our reason for introducing this terminology is that if $I \times_\psi Y$ is constructed using the $\ell^1$-norm with $Y$ 
	geodesic and $\psi \colon I \to [0, \infty)$ increasing, then it is relatively easy to see that geodesics in $I \times_\psi Y$ 
	are \pa-curves. In the following lemma, we prove an analogue of this fact for length spaces $Y$, by showing 
	that every curve $\gamma$ in $I \times Y$ can be replaced with a \pa-curve with the same endpoints but 
	shorter $d_\psi$-length. We also show an upper bound for the $d_\psi$-length of the horizontal part of the replacing \pa-curve.
	
	\begin{lemma}\label{lem:shape-geodesics-I}
	Let $(Y,d_Y)$ be a length space. Suppose that $(Y,d_Y)$, $\psi$, and $\alpha$ satisfy Condition~(1) of Theorem~\ref{thm:solid_filling_is_Gromov_hyp-II},
	and the pseudometric $d_\psi$ on $I\times Y$ is constructed using the $\ell^1$-norm on $\R^2$.
If $(t_1,x_1),(t_2,x_2)\in I\times Y$ and if $\beta_0$ is
an absolutely continuous path in $Y$ from $x_1$ to $x_2$,  then there exists an absolutely continuous \pa-curve $\beta$ 
from $(t_1, x_1)$ to $(t_2, x_2)$ with horizontal level $\tau \in [0, \min \{t_1, t_2\}]$, such that the following properties hold.
		\begin{enumerate}
			\item[(i)] The horizontal part $\beta^{\hor}$ of $\beta$ satisfies $\beta^\hor_Y = \beta_0$, and 
			its horizontal level $\tau$ depends only on $\psi$, $\min\{t_1, t_2\}$, and $\len_{d_Y}(\beta_0)$.
			\item[(ii)] Whenever $\gamma=(\gamma_I,\gamma_Y)$
			is an absolutely continuous path from $(t_1,x_1)$ to $(t_2,x_2)$ 
			with $\len_{d_Y}(\beta_0) \le \len_{d_Y}(\gamma_Y)$, we must have
			$\len_{d_\psi}(\beta)\le \len_{d_\psi}(\gamma)$. 
			\item[(iii)] We have that
				\begin{align*}
					\len_{d_\psi}(\beta^\hor) &= \psi(0) \len_{d_Y}(\beta_0) & \text{if } & \tau = 0,\\
					\len_{d_\psi}(\beta^\hor) &\le 2\alpha^{-1} & \text{if } & 0 < \tau \le \min \{t_1, t_2\}.
				\end{align*}
		\end{enumerate}
	\end{lemma}
	
	\begin{proof}
		We first prove (i). For every $\rho \in [0, \min \{t_1, t_2\}]$, there is a \pa-curve 
$\beta_\rho$
from $(t_1, x_1)$ to $(t_2, x_2)$ with horizontal level $\rho$ and horizontal part $\beta_\rho^{\hor}=(\beta^{\hor}_{\rho, I},\beta_0)$.
By \eqref{eq:pa-curve_length-total}, we have
			\begin{equation}\label{eq:beta_rho_length}
				\len_{d_\psi}(\beta_\rho) = F(\rho) := t_1 + t_2 - 2\rho + \psi(\rho) \len_{d_Y}(\beta_0).
			\end{equation}
			Since $F$ is continuous, $F$ has a minimal value on $[0, \min \{t_1, t_2\}]$. 
			Let $\tau$ be the largest $\rho \in [0, \min \{t_1, t_2\}]$ at which this minimal value of $F(\rho)$ occurs, 
and select $\beta = \beta_\tau$. 
Since $F$ is $C^1$-smooth, the optimal choice $\tau$ depends solely on $F^\prime=-2+\len_{d_Y}(\beta_0)\, \psi^\prime$
and the end points of the interval $[0,\min\{t_1,t_2\}]$. 
Thus, we conclude that (i) holds with the choice of $\beta=\beta_\tau$.
			
		We now show (ii). Let
		$\gamma =(\gamma_{{{I}}}, \gamma_Y):[a,b]\to I\times Y$ be  
		an absolutely continuous curve in $I\times Y$ from $(t_1,x_1)$ to $(t_2,x_2)$,
	  and suppose that  $\len_{d_Y}(\beta_0) \le \len_{d_Y}(\gamma_Y)$.
			Let $\tau_0=\min\{\gamma_I(t)\, :\, a\le t\le b\}$, noting that $\tau_0 \in [0, \min\{t_1, t_2\}]$.
		As $\psi$ is a monotone increasing function by \eqref{eq:psi_deriv_condition}, 
		we have 
		\begin{align*}
			\len_{d_\psi}(\gamma)=\ell_\psi(\gamma)&=\int_a^b|\gamma_{{{I}}}^\prime(t)|\, dt
			+\int_a^b\psi(\gamma_{{{I}}}(t))\, |\gamma_Y^\prime(t)|\, dt\\
			& \ge \len_{d_{\text{Eucl}}}(\gamma_{{{I}}})+\psi(\tau_0)\, \len_{d_Y}(\gamma_Y)\\
			& \ge (t_1 - \tau_0) + (t_2 - \tau_0) + \psi(\tau_0) \len_{d_Y}(\beta_0)\\
			& = \len_{d_\psi}(\beta_{\tau_0}) \ge \len_{d_\psi}(\beta).
		\end{align*}
		Thus, (ii) holds. For (iii), note that \eqref{eq:len_hor_part} yields 
		$\len_{d_\psi}(\beta^\hor) = \psi(\tau) \len_{d_Y}(\beta_0)$, and thus the 
		case $\tau = 0$ is clear. The case $\tau > 0$ of (iii) follows by applying 
		Lemma~\ref{lem:psi_conds} part~\eqref{enum:psi_minimizer_bound} with $A = \len_{d_Y}(\beta_0)$, which yields that
		\[
			\len_{d_\psi}(\beta^\hor) = \psi(\tau) \len_{d_Y}(\beta_0) \le 2 \alpha^{-1}.\qedhere
		\]
\end{proof}

We also prove an analogous statement for distances.

\begin{lemma}\label{lem:hyp_dist_computation}
 Let $(Y,d_Y)$ be a length space. Suppose that $(Y,d_Y)$, $\psi$, and $\alpha$ satisfy Condition~(1) of Theorem~\ref{thm:solid_filling_is_Gromov_hyp-II},
	and the pseudometric $d_\psi$ on $I\times Y$ is constructed using the $\ell^1$-norm on $\R^2$. 
	If $(t_1,x_1),(t_2,x_2)\in I\times Y$, then
	\begin{equation}\label{eq:d_psi_formula}
		d_\psi((t_1,x_1), (t_2,x_2)) = t_1+t_2+\min_{\rho \in [0, \min\{t_1,t_2\}]}  \left(\psi(\rho)d_Y(x_1,x_2) -2\rho\right).
	\end{equation}
	Moreover, if the minimum in~\eqref{eq:d_psi_formula} is achieved at some $\tau \in (0, \min \{t_1, t_2\}]$,
	then 
	\begin{equation}\label{eq:horiz_dist_est}
		\psi(\tau)d_Y(x_1,x_2) \le \alpha^{-1} \psi'(\tau)d_Y(x_1,x_2) \le 2\alpha^{-1}.
	\end{equation}
\end{lemma}

\begin{proof}
 If $\gamma = (\gamma_I, \gamma_Y)$ is any absolutely continuous curve from $(t_1,x_1)$ to $(t_2,x_2)$, then by applying Lemma \ref{lem:shape-geodesics-I} (ii) with $\beta_0 = \gamma_I$ and Proposition \ref{prop:lengths_match_warped}, we find an absolutely continuous \pa-curve $\beta$ from $(t_1,x_1)$ to $(t_2,x_2)$ with horizontal level $\tau \in [0, \min \{t_1, t_2\}]$ such that $\ell_\psi(\beta) \le \ell_\psi(\gamma)$. By \eqref{eq:pa-curve_length-total}, we thus have
	\begin{align*}
		\ell_\psi(\gamma) \ge \ell_\psi(\beta) &= t_1 + t_2 + \psi(\tau)\len_{d_Y}(\beta_0) - 2\tau\\
		&\ge t_1 + t_2 + \psi(\tau)d_Y(x_1, x_2) - 2\tau\\
		&\ge t_1+t_2+\min_{\rho \in [0, \min\{t_1,t_2\}]}  \left(\psi(\rho)d_Y(x_1,x_2) -2\rho\right).
	\end{align*}
	As this lower bound is independent of $\gamma$, taking the infimum over $\gamma$ yields that the left hand side of \eqref{eq:d_psi_formula} is greater than or equal to the right hand side. On the other hand, Lemma \ref{lem:d_varphi_bounds} part~(1) 
	proves that the left hand side of \eqref{eq:d_psi_formula} is less than or equal to the right hand side.
		
Finally, \eqref{eq:horiz_dist_est} follows from Lemma \ref{lem:psi_conds} part \eqref{enum:psi_minimizer_bound} with $A = d_Y(x_1,x_2)$.
\end{proof}

As an immediate consequence of Lemma \ref{lem:hyp_dist_computation}, we obtain an explicit formula for Gromov products in our setting.

\begin{corollary}\label{cor:gromp_computation}
	Let $(Y,d_Y)$ be a length space. Suppose that $(Y,d_Y)$, $\psi$, and $\alpha$ satisfy Condition~(1) of Theorem~\ref{thm:solid_filling_is_Gromov_hyp-II},
	and the pseudometric $d_\psi$ on $I\times Y$ is constructed using the $\ell^1$-norm on $\R^2$.
	If $z_0, z_1, z_2 \in I \times Y$ with $z_0 = (0, x_0)$, $z_1 = (t_1, x_1)$, and $z_2 = (t_2, x_2)$, then the Gromov product of the equivalence classes $[z_0], [z_1], [z_2]$ in $I \times_\psi Y$ is given by
	\[
		\gromp{[z_1]}{[z_2]}_{[z_0]} = \frac{1}{2} \left( \psi(0)(d_Y(x_1, x_0) + d_Y(x_2, x_0)) 
			+ \max_{\rho \in [0, \min\{t_1, t_2\}]} (2\rho - \psi(\rho) d_Y(x_1, x_2)) \right). 
	\]
\end{corollary}

We now have the tools needed to prove Theorem~\ref{thm:solid_filling_is_Gromov_hyp-II}. We split the proof into two parts. 
 Subsection~\ref{sub:hyper} is devoted to the proof of  the Gromov hyperbolicity
property, 
and Subsection~\ref{sub:bdryGromov} is devoted to identifying the Gromov boundary.}

\subsection{Proof of Gromov hyperbolicity}\label{sub:hyper}

As noted at the beginning of the section, it suffices to show 
the Gromov hyperbolicity part of Theorem \ref{thm:solid_filling_is_Gromov_hyp-II} 
when $\norm{\cdot}$ is the $\ell^1$-norm~on~$\R^2$.

\begin{lemma}\label{lem:warped_prod_is_gromov_hyp}
The metric space $I\times_\psi Y$ considered in Theorem~\ref{thm:solid_filling_is_Gromov_hyp-II},
with the pseudometric $d_\psi$ on $I\times Y$ constructed using the $\ell^1$-norm on $\R^2$,
	is Gromov $\delta$-hyperbolic, where $\delta = 2\alpha^{-1}$ if $\psi(0) = 0$, and $\delta = 2\alpha^{-1} + 3 \psi(0) \diam(Y)$ if $\psi(0) \ne 0$.
\end{lemma}

\begin{proof}
	We fix a point $z_0 = (0, x_0) \in I \times Y$. We also denote 
	$D = 0$ if $\psi(0) = 0$, and $D = \psi(0) \diam(Y)$ if $\psi(0) \ne 0$, noting that our assumptions 
	require $Y$ to be bounded if $\psi(0)\ne 0$.
	By Lemma \ref{lem:Gromov_hyp_basepoint_invariance}, it suffices to show that
	\begin{equation}\label{eq:hyperbolicity_claim_1}
		\gromp{[z_1]}{[z_2]}_{[z_0]} \ge \min \{ \gromp{[z_1]}{[z_3]}_{[z_0]}, \gromp{[z_2]}{[z_3]}_{[z_0]}\} - (\alpha^{-1} + 3D/2)
	\end{equation}
	for all $z_1, z_2, z_3 \in I \times Y$. For this purpose, let $z_1 = (t_1, x_1), z_2 = (t_2, x_2), z_3 = (t_3, x_3)$ be points in $I \times Y$.
	
	For all indices $i, j \in \{1, 2, 3\}$, we define $F_{ij} \colon [0, \min\{t_i, t_j\}] \to \R$ by 
	\[
		F_{ij}(\rho) = 2\rho - \psi(\rho)\, d_Y(x_i, x_j),
	\]
	and fix a point $\tau_{ij} \in [0, \min\{t_i, t_j\}]$ where $F_{ij}$ attains its maximum value. By the formula for the Gromov product shown in Corollary \ref{cor:gromp_computation}, it follows that for $i, j \in \{1, 2, 3\}$,
	\[
		\frac{1}{2} F_{ij}(\tau_{ij}) \le \gromp{[z_i]}{[z_j]}_{[z_0]} \le D + \frac{1}{2} F_{ij}(\tau_{ij}).
	\]
	Thus, \eqref{eq:hyperbolicity_claim_1} follows if we show that
	\begin{equation}\label{eq:hyperbolicity_claim_2}
		F_{12}(\tau_{12}) \ge \min \{ F_{13}(\tau_{13}), F_{23}(\tau_{23})\} - (2\alpha^{-1} + D).
	\end{equation}
	
	By symmetry, we may assume that $\tau_{13} \le \tau_{23}$, by switching the labeling 
of $z_1$ and $z_2$ if necessary. Now, since $\tau_{13} \in [0, \min\{t_1, t_3\}]$ and 
$\tau_{23} \in [0, \min\{t_2, t_3\}]$, we have $\tau_{13} \le t_1$ and 
$\tau_{13} \le \tau_{23} \le t_2$. Thus, $\tau_{13} \in [0, \min \{t_1, t_2\}]$. Since 
$F_{12}$ reaches its maximum value over $[0, \min \{t_1, t_2\}]$ at $\tau_{12}$, we thus have
	\[
		F_{12}(\tau_{13}) \le F_{12}(\tau_{12}).
	\]
	Now, by the triangle inequality, 
	\begin{align*}
		F_{12}(\tau_{12}) \ge F_{12}(\tau_{13})
		&= 2\tau_{13} - \psi(\tau_{13})d_Y(x_1, x_2)\\
		&\ge 2\tau_{13} - \psi(\tau_{13}) (d_Y(x_1, x_3) + d_Y(x_2, x_3))\\
		& = F_{13}(\tau_{13}) - \psi(\tau_{13}) d_Y(x_2, x_3)\\
		&\ge F_{13}(\tau_{13}) - \psi(\tau_{23}) d_Y(x_2, x_3),
	\end{align*}
where we used the monotonicity of $\psi$
and the assumption $\tau_{23}\ge \tau_{13}$ in the 
last step above. By~\eqref{eq:horiz_dist_est} of Lemma \ref{lem:hyp_dist_computation}, we have that
$\psi(\tau_{23})d_Y(x_2, x_3) \le \max \{2\alpha^{-1}, D\}$. Therefore,
	\begin{align*}
		F_{12}(\tau_{12})
		&\ge F_{13}(\tau_{13}) - \psi(\tau_{23}) d_Y(x_2, x_3)\\
		&\ge F_{13}(\tau_{13}) - \max \{2\alpha^{-1}, D\}\\
		&\ge \min \{ F_{13}(\tau_{13}), F_{23}(\tau_{23})\} - (2\alpha^{-1} + D).
	\end{align*}
	Hence~\eqref{eq:hyperbolicity_claim_2} holds, and the proof is complete.
\end{proof}

\begin{remark}
For readers who prefer the thin triangles definition of Gromov hyperbolicity, we point out that a family
 of \pa-curves $\beta$
constructed in Lemma~\ref{lem:shape-geodesics-I}
above can be shown to satisfy the hypotheses given in~\cite[Theorem 2.34]{Vai}.
\end{remark}

\subsection{Proof of the claim regarding the Gromov boundary}\label{sub:bdryGromov}

We first show that under the assumptions of Theorem~\ref{thm:solid_filling_is_Gromov_hyp-II}, 
there is 
a bijection $\Phi \colon \partial_{G}([0, \infty) \times_\psi Y) \to Y$.

\begin{lemma}\label{lem:map_from_Gromov_boundary_general}
 Let $(Y,d_Y)$ be a complete and bounded length space.
Suppose that $(Y,d_Y)$, $\alpha$, and $\psi$ satisfy the hypotheses~(1) and~(2) of Theorem~\ref{thm:solid_filling_is_Gromov_hyp-II}.
	Let $Z = I\times_\psi Y$, where the warped product is defined using the 
	$\ell^1$-norm on $\R^2$, and equip $Z$ with a base point  $[z_0]$ for a choice of 
	$z_0 = (0, x_0)\in I\times Y$. 
	Then the following statements hold.
	\begin{enumerate}
		\item \label{enum:H_gromov_seq_chara} 
		For every sequence of points $z_i = (t_{z_i}, x_{z_i}) \in I \times Y$, the sequence $([z_i])$
		is a Gromov sequence in $Z$ if and only if $(x_{z_i})$ is Cauchy in $Y$ and $\lim_{i \to \infty} t_{z_i} = \infty$. 
		Moreover, if $([z_i])$ is a Gromov sequence, then the
		limit $\lim_{i\to\infty}x_{z_i}$ is independent of the choice of representatives $(t_{z_i},x_{z_i})\in[z_i]$.
		\item \label{enum:H_gromov_seq_equiv_chara} If $z_i = (t_{z_i}, x_{z_i}) \in I \times Y$ and $w_i = (t_{w_i}, x_{w_i}) \in I\times Y$ 
		are such that $([z_i])$ and $([w_i])$ are Gromov sequences in $Z$, then $([z_i]) \sim ([w_i])$ if and only if 
		$\lim_{i \to \infty} x_{z_i} = \lim_{i \to \infty} x_{w_i}$.
		\item \label{enum:H_gromov_seq_product_computation} 
		The map $\Phi\colon \partial_{G} Z \to Y$, defined by $\Phi(\overline{z})=\lim_{i\to\infty}x_{z_i}$
		for each $\overline{z}\in\partial_G Z$ and each Gromov sequence $([z_i]) = ([(t_{z_i}, x_{z_i})])\in\overline{z}$, is a well-defined bijective map.	
		Moreover, the extended Gromov product on $\partial_G Z$ is given by
		\begin{equation}\label{eq:boundary_gromp_formula}\begin{aligned}
				2\gromp{\overline{z}}{\overline{w}}_{[z_0]} &= \psi(0)\bigl(d_Y(\Phi(\overline{z}), x_0) + d_Y(\Phi(\overline{w}), x_0)\bigr)\\
				&\qquad \qquad \qquad + \sup_{\rho \in [0, \infty)} 
				\left(2\rho - \psi(\rho) d_Y(\Phi(\overline{z}), \Phi(\overline{w}))\right)
		\end{aligned}\end{equation}
		for all $\overline{z}, \overline{w} \in \partial_{G} Z$.
	\end{enumerate}
\end{lemma}

\begin{proof} 
To prove the necessity part of~\eqref{enum:H_gromov_seq_chara} of the lemma, let $([z_i])$ be a sequence in $I\times_\psi Y$.
Choosing $(t_{z_i},x_{z_i})\in[z_i]$, by Corollary~\ref{cor:gromp_computation} and the
definition of Gromov sequences, $([z_i])$ is a Gromov sequence if and only if 
\[
\liminf_{i,j\to\infty}\left(\psi(0)[d_Y(x_{z_i},x_0)+d_Y(x_{z_j},x_0)]
+\sup_{\rho\in[0,\min\{t_{z_i},t_{z_j}\}]}\, (2\rho-\psi(\rho)\, d_Y(x_{z_i},x_{z_j}))\right)=\infty.
\]
As $Y$ is bounded, the above condition is equivalent to the condition
\begin{equation}\label{eq:t-infty}
\liminf_{i,j\to\infty}\sup_{\rho\in[0,\min\{t_{z_i},t_{z_j}\}]}\, (2\rho-\psi(\rho)\, d_Y(x_{z_i},x_{z_j}))=\infty.
\end{equation}
Suppose that $([z_i])$ is a Gromov sequence.
Then, by~\eqref{eq:t-infty}, $\liminf_{i,j\to\infty}\min\{t_{z_i},t_{z_j}\}=\infty$, that is, $\lim_{i\to\infty}t_{z_i}=\infty$. Moreover,
there must be some positive integer $k$ such that whenever $i\ge k$ and $j\ge k$, we have
\[
\sup_{\rho\in[0,\min\{t_{z_i},t_{z_j}\}]}\, (2\rho-\psi(\rho)\, d_Y(x_{z_i},x_{z_j}))\ge 0. 
\]
With a choice of $\rho_{i,j}\in [0,\min\{t_{z_i},t_{z_j}\}]$ for which $2\rho_{i,j}-\psi(\rho_{i,j})\, d_Y(x_{z_i},x_{z_j})$ is the above supremum, 
\[
2\rho_{i,j}-\psi(\rho_{i,j})\, d_Y(x_{z_i},x_{z_j})\ge 0.
\]
We
must have that $\lim_{i,j\to\infty}\rho_{i,j}=\infty$  by~\eqref{eq:t-infty}. Therefore, by Lemma~\ref{lem:psi_conds} 
part~\eqref{enum:psi_exp_estimate},
for sufficiently large positive integers $i,j$, we have
\[
d_Y(x_{z_i}, x_{z_j})\le  \frac{2\rho_{i,j}}{\psi(\rho_{i,j})}\le \frac{2\rho_{i,j}}{\psi(a)\, e^{\alpha (\rho_{i,j}-a)}},
\]
where $a>0$ is chosen so that $\psi>0$ on $[a,\infty)$. As $\lim_{r\to\infty}2r
/(\psi(a)\, e^{\alpha (r-a)})=0$, it follows
that $\lim_{i,j\to\infty}d_Y(x_{z_i},x_{z_j})=0$, that is, $(x_{z_i})$ is a Cauchy sequence in $Y$. As $Y$ is complete, it follows that
$\lim_{i\to\infty}x_{z_i}$ exists as a point in $Y$. Because $\lim_{i\to \infty}t_{z_i}=\infty$, if $(t_{z_i},x'_{z_i})\in[z_i]$, 
then for sufficiently large integer $i$ we must have $x'_{z_i}=x_{z_i}$, and so the limit of the Cauchy sequence in $Y$ is uniquely
determined by the Gromov sequence $([z_i])$. 

The necessity part of~\eqref{enum:H_gromov_seq_equiv_chara} is proved in a similar manner. Indeed, 
if we have two Gromov sequences $([z_i])$ and $([w_i])$ such that $\lim_{i\to\infty}\gromp{[z_i]}{[w_i]}_{[z_0]}=\infty$,
then a very similar argument to the above with $z_j$ replaced with $w_i=[(t_{w_i} ,x_{w_i})]$ and letting $i\to\infty$ gives that
$\lim_{i\to\infty}t_{z_i}=\infty$ and $\lim_{i\to\infty}t_{w_i}=\infty$. Moreover, as above, 
we know that $\lim_{i\to\infty}\rho_i=\infty$, with $\rho_i\in[0,\min\{t_{z_i},t_{w_i}\}]$ such that
\[
2\rho_i-\psi(\rho_i)d_Y(x_{z_i},x_{w_i})=\max_{\rho\in [0,\min\{t_{z_i},t_{w_i}\}]}(2\rho-\psi(\rho)d_Y(x_{z_i},x_{w_i})). 
\]
Hence
we have
\[
\limsup_{i\to\infty}d_Y(x_{z_i},x_{w_i})\le \limsup_{i\to\infty}\frac{2\rho_i}{\psi(a)\, e^{\alpha (\rho_{i}-a)}}.
\]
Now it follows from the above inequality that
$\lim_{i\to\infty}d_Y(x_{z_i},x_{w_i})=0$, and so necessarily we have that
$\lim_{i\to\infty}x_{z_i}=\lim_{i\to\infty}x_{w_i}$.

This proves the necessity parts of claim~\eqref{enum:H_gromov_seq_chara} and
of claim~\eqref{enum:H_gromov_seq_equiv_chara},
and also shows that the map $\Phi\colon \partial_GZ\to Y$ given by
\[
\Phi(\bar{z})=\lim_{i\to\infty}x_{z_i}\text{ for any choice of }([z_i])\in\bar{z}
\]
is well-defined and is independent of the choice of representative Gromov sequence $([z_i])\in\bar{z}$. 
On the other hand, for $x\in Y$, note that with $z_i=(i,x)$, the sequence $([z_i])\in I\times_\psi Y$ is a
Gromov sequence as 
\[
\gromp{[z_i]}{[z_j]}_{[z_0]}=
\psi(0)d_Y(x_0, x) + \frac{1}{2} \min\{i, j\}
\to\infty\text{ as }i,j\to\infty.
\]
Thus, with $\bar{x}$ denoting the Gromov equivalence class of the sequence $([z_i])$, we see that
$\Phi(\bar{x})=x$. Thus $\Phi$ is surjective.

To prove the sufficiency part of claim~\eqref{enum:H_gromov_seq_chara}, suppose now that $([z_i])$ is a sequence
in $I\times_\psi Y$ with $z_i=(t_{z_i},x_{z_i})$ such that $\lim_{i\to\infty} t_{z_i}=\infty$ and that 
$(x_{z_i})$ is a Cauchy sequence in $Y$. Let $z:=\lim_{i\to\infty}x_{z_i}$.
It follows that $\lim_{i,j\to\infty} d_Y(x_{z_i},x_{z_j})=0$. From the equation
\begin{align*}
2\gromp{[z_i]}{[z_j]}_{[z_0]}&=\psi(0)[d_Y(x_{z_i},x_0)+d_Y(x_{z_j},x_0)]
  +\sup_{\rho\in[0,\min\{t_{z_i},t_{z_j}\}]}\, (2\rho-\psi(\rho)\, d_Y(x_{z_i},x_{z_j})),
\end{align*}
to show that the sequence $([z_i])$ is a Gromov sequence, it suffices to show that
\[
\liminf_{i,j\to\infty} (\sup_{\rho\in[0,\min\{t_{z_i},t_{z_j}\}]}\, (2\rho-\psi(\rho)\, d_Y(x_{z_i},x_{z_j})))=\infty.
\]
Let $\rho_{i,j}\in [0,\min\{t_{z_i},t_{z_j}\}]$ such that 
\[
\sup_{\rho\in[0,\min\{t_{z_i},t_{z_j}\}]}\, (2\rho-\psi(\rho)\, d_Y(x_{z_i},x_{z_j}))=2\rho_{i,j}-\psi(\rho_{i,j})\, d_Y(x_{z_i},x_{z_j}).
\]
By Lemma \ref{lem:hyp_dist_computation} part \eqref{eq:horiz_dist_est}, we know that 
$\psi(\rho_{i,j})\, d_Y(x_{z_i},x_{z_j})\le \alpha/2$ if $\rho_{i,j} > 0$.
Therefore, to show that $([z_i])$ is a Gromov sequence, it suffices to show that
\begin{equation}\label{eq:rho-ij}
\liminf_{i,j\to\infty} \rho_{i,j}=\infty.
\end{equation}

For this, fix $C>0$ and suppose that $\liminf_{i,j\to\infty}\rho_{i,j}< C$. Then we can find large integers $i,j$ such that 
$\min\{t_{z_i},t_{z_j}\}>2C$, $\psi(2C)\, d_Y(x_{z_i},x_{z_j})<C$, and $\rho_{i,j}<C$. However, we have
$2C\in[0, \min\{t_{z_i},t_{z_j}\}]$ and $4C-\psi(2C)\, d_Y(x_{z_i},x_{z_j})>3C$, while 
$2\rho_{i,j}-\psi(\rho_{i,j})\, d_Y(x_{z_i},x_{z_j})\le 2\rho_{i,j}<2C$, violating the choice of $\rho_{i,j}$ as
maximizing the quantity $2\rho-\psi(\rho)\, d_Y(x_{z_i},x_{z_j})$ over all $\rho\in [0, \min\{t_{z_i},t_{z_j}\}]$.
Thus for each $C>0$ we must have $\liminf_{i,j\to\infty}\rho_{i,j}\ge C$, that is,~\eqref{eq:rho-ij} holds.
This completes the proof of the sufficiency part of claim~\eqref{enum:H_gromov_seq_chara}. 

This proof,
with necessary changes, also shows that if $([z_i])$ and $([w_i])$ are two Gromov sequences with 
$\lim_{i\to\infty}d_Y(x_{z_i},x_{w_i})=0$ (that is, the two Cauchy sequences $(x_{z_i})$ and $(x_{w_i})$ have the same
limit in $Y$), then $\liminf_{i\to\infty}\gromp{[z_i]}{[w_i]}_{[z_0]}=\infty$, completing the proof of 
the sufficiency part of claim~\eqref{enum:H_gromov_seq_equiv_chara}. From this, the injectivity of the map
$\Phi$ also follows. Therefore $\Phi\colon\partial_GZ\to Y$ is a well-defined bijection.

If we know that the sequences $([z_i])$ and $([w_i])$ are both a priori Gromov sequences in $I\times_\psi Y$,
then $([z_i])\sim([w_i])$ if and only if $\limsup_{i\to\infty} \rho_{i}=\infty$, where $\rho_{i}$ is the point at which the 
function $F_{x_{z_i},x_{w_i}}(\rho)=2\rho-\psi(\rho)\, d_Y(x_{z_i},x_{w_i})$ achieves its maximum in the interval $[0,\min\{t_{z_i},t_{w_i}\}]$. Indeed, if $\limsup_{i\to\infty} \rho_{i}=\infty$, then we can find
two subsequences $([z_{i_k}])$ and $([w_{i_k}])$ such that $\lim_{k\to\infty}\rho_{i_k}=\infty$,
at which point the argument preceding~\eqref{eq:rho-ij} tells us that $([z_{i_k}])\sim([w_{i_k}])$, and then
from claim~\eqref{enum:H_gromov_seq_chara} and claim~\eqref{enum:H_gromov_seq_equiv_chara}
we must have $([z_i])\sim([w_i])$. On the other hand, if $\limsup_{i\to\infty}\rho_i<\infty$, then
from the argument proving the necessity part of~\eqref{enum:H_gromov_seq_equiv_chara}
using~\eqref{eq:t-infty} we must have $([z_i])\not\sim([w_i])$.
With this knowledge, 
now we proceed to prove the final claim of the lemma.

Let $\overline{z},\overline{w}\in\partial_GZ$, and $([z_i])\in\overline{z}$, $([w_i])\in\overline{w}$. If $\overline{z}=\overline{w}$, then 
$\lim_{i\to\infty}\gromp{[z_i]}{[w_i]}_{[z_0]}=\infty$, and so from the definition of
$\gromp{\overline{z}}{\overline{w}}_{[z_0]}$ given in~\eqref{eq:GromProd-bdy}, we have that
$\gromp{\overline{z}}{\overline{w}}_{[z_0]}=\infty$, which agrees with~\eqref{eq:boundary_gromp_formula}.
Now suppose that $\overline{z}\ne \overline{w}$, in which case we also have that $([z_i])\not\sim([w_i])$.
Therefore $\limsup_{i\to\infty}\rho_i<\infty$. For $C>\limsup_{j\to\infty}\rho_j$, we have for sufficiently large $i$ 
 that $\min\{t_{z_i},t_{w_i}\}>C$ and $\rho_i\le C$. For such $i$ we have
\begin{align*}
2\gromp{[z_i]}{[w_i]}_{[z_0]}&=\psi(0)[d_Y(x_0,x_{z_i})+d_Y(x_0,x_{w_i})]
+\max_{\rho\in[0,\min\{t_{z_i},t_{w_i}\}]}\, (2\rho-\psi(\rho)d_Y(x_{z_i},x_{w_i}))\\
&=\psi(0)[d_Y(x_0,x_{z_i})+d_Y(x_0,x_{w_i})]
+\max_{\rho\in[0,C]}\, (2\rho-\psi(\rho)d_Y(x_{z_i},x_{w_i})).
\end{align*}
Since $\lim_{i\to\infty}x_{z_i}=\Phi(\overline{z})$ and $\lim_{i\to\infty}x_{w_i}=\Phi(\overline{w})$, we have that
\[
\lim_{i\to\infty}2\gromp{[z_i]}{[w_i]}_{[z_0]}=\psi(0)[d_Y(x_0,\Phi(\overline{z}))+d_Y(x_0,\Phi(\overline{w}))]
  +\max_{\rho\in[0,C]}\, (2\rho-\psi(\rho)d_Y(\Phi(\overline{z}),\Phi(\overline{w}))).
\]
Since the above expression holds true for all $C>\limsup_{i\to\infty}\rho_i$, taking the limit as $C\to\infty$ we obtain
the identity~\eqref{eq:boundary_gromp_formula}.
\end{proof}

It now remains to prove that the map $\Phi$ of Lemma \ref{lem:map_from_Gromov_boundary_general} is a quasisymmetry. In the proof, 
we use the following estimate, which we separate into its own lemma.

\begin{lemma}\label{lem:exponential_supremizer}
	Let $K$, $D$, and $\alpha$ be positive real numbers.
	Then there exists a constant $C = C(K, D, \alpha) > 0$ such that for all $d \in (0, D]$, we have 
	\begin{equation}\label{eq:exponential_supremizer_est}
		-C - \frac{2}{\alpha} \ln (d)
		\le \sup_{\rho \in [0, \infty)} \left( 2\rho - K d e^{\alpha \rho} \right)
		\le C - \frac{2}{\alpha} \ln (d)
	\end{equation}
\end{lemma}

\begin{proof}
	 Let $d \in (0, D]$, and let $F_{d} \colon \R \to \R$ be given by $F_{d}(\rho) = 2\rho - K d e^{\alpha \rho}$. We have $F_d'(\rho) = 2 - K \alpha d e^{\alpha \rho}$, which is decreasing with a single zero at $\rho_d := \alpha^{-1} \ln (2/(K\alpha d))$. Suppose first that $d \in (0, 2/(K\alpha))$. Then $\rho_d > 0$, and consequently
	\[
	\sup_{\rho \in [0, \infty)} F_d(\rho) = 2\rho_d - K d e^{\alpha \rho_d} = \frac{2}{\alpha} \ln \frac{2}{K\alpha d} - \frac{2}{\alpha}.
	\]
	This yields \eqref{eq:exponential_supremizer_est} for this range of $d$, with any choice of
	\[
	C \ge (2/\alpha) \abs{\ln(2/(K\alpha)) - 1}.  
	\]
	
	Now consider the case $d \in [2/(K\alpha), D]$. Then $\rho_d \le 0$ and hence $\sup_{\rho \in [0, \infty)} F_d(\rho) = F_d(0)$. Using the estimates $2/(K\alpha d) \le 1 \le D/d$ and $-D \le -d \le 0$, we obtain that
	\[
	\frac{2}{\alpha}\ln \frac{2}{K\alpha d} - KDe^{\alpha} \le - KDe^{\alpha} \le -K d e^{\alpha}
	= F_d(0) \le 0 \le \frac{2}{\alpha}\ln \frac{D}{d}.
	\]
	This yields \eqref{eq:exponential_supremizer_est} for this range of $d$, for any choice of
	with 
	\[
	C \ge  \max \{K D e^{\alpha} + (2/\alpha) \ln(K\alpha/2), (2/\alpha)\ln(D)\}.
	\]
	As all $d \in (0, D]$ are covered by these two cases, the proof is complete 
	by choosing $C$ large enough to satisfy both of the 
	above requirements.
\end{proof}

We now prove that $\Phi$ is a quasisymmetry.

\begin{lemma}
	Let $(Y,d_Y)$, $\alpha$, $\psi$, $Z$, $z_0$, and $\Phi \colon \partial_G Z \to Y$ be as in Lemma \ref{lem:map_from_Gromov_boundary_general}. Let $\delta \ge 0$ be such that $Z$ is Gromov $\delta$-hyperbolic, and 
	suppose that there exists a constant $C > 0$ such that $\psi(t) \le C e^{\alpha t}$ for all $t \in [0, \infty)$. 
	Then for all $\eps \in (0, \min \{1, 1/(5\delta)\}]$, the map $\Phi$ is an $\eta$-quasisymmetric homeomorphism from $(\partial_G Z, d_\eps)$ to $(Y, d_Y)$, where $d_\eps$ is as in \eqref{eq:gromov_boundary_metric} and $\eta$ depends only on $\alpha$, $\psi$, $\diam(Y)$, and $\eps$.
\end{lemma}

\begin{proof}
	 Let $0 < \eps \le \min\{1, 1/(5\delta)\}$, let $\overline{z}, \overline{w} \in \partial_{G} Z$ with 
	$\overline{z} \ne \overline{w}$, and denote $x = \Phi(\overline{z})$, $y = \Phi(\overline{w})$. We 
	prove the claim by finding a constant $C_0 = C_0(\alpha, \psi, \diam(Y), \eps) > 0$, for which
	\begin{equation}\label{eq:Gromov_bdry_QS_objective}
		C_0^{-1}[d_Y(x, y)]^{\eps/\alpha} \le d_\eps(\overline{z}, \overline{w}) \le C_0 [d_Y(x, y)]^{\eps/\alpha}.
	\end{equation}
	
	To prove \eqref{eq:Gromov_bdry_QS_objective}, we first recall that since $\eps \le \min \{1, 1/(5\delta)\}$, we have by \eqref{eq:boundary_metric_premetric_comparison} that
	\begin{equation}\label{eq:Gromov_bdry_QS_step_1}
		2^{-1} e^{-\eps\gromp{\overline{z}}{\overline{w}}_{[z_0]}} \le d_\eps(\overline{z}, \overline{w}) 
		\le e^{-\eps\gromp{\overline{z}}{\overline{w}}_{[z_0]}}.
	\end{equation}
	By part~\eqref{enum:H_gromov_seq_product_computation} of Lemma~\ref{lem:map_from_Gromov_boundary_general},
	we also have the two-sided estimate
	\begin{equation}\label{eq:Gromov_bdry_QS_step_2}\begin{aligned}
			\sup_{\rho \in [0, \infty)} \left(2\rho - \psi(\rho) d_Y(x, y)\right)
			&\le 2\gromp{\overline{z}}{\overline{w}}_{[z_0]}\\
			&\le 2\psi(0) \diam(Y) + \sup_{\rho \in [0, \infty)} \left(2\rho - \psi(\rho) d_Y(x, y)\right).
	\end{aligned}\end{equation}
	By assumption, $\psi(t) \le C e^{\alpha t}$ for all $t \in [0, \infty)$. 
	Moreover, since $\psi$ is non-constant, there exists $b \in [0, \infty)$ for which $\psi(b) > 0$, and hence by Lemma \ref{lem:psi_conds} \eqref{enum:psi_exp_estimate}, $\psi(t) \ge (\psi(b) e^{-\alpha b}) e^{\alpha t}$ for all $t \ge b$. It follows that there exist $c > 0$ and $A \ge 0$, depending only on $\psi$ and $\alpha$, such that
	\[
	c e^{\alpha t} - A \le \psi(t) \le C e^{\alpha t}
	\]
	for all $t \in [0, \infty)$. By combining this with \eqref{eq:Gromov_bdry_QS_step_2}, it follows that
	\begin{equation}\label{eq:Gromov_bdry_QS_step_3}\begin{aligned}
			\sup_{\rho \in [0, \infty)} \left(2\rho - Cd_Y(x, y)e^{\alpha \rho} \right)
			&\le 2\gromp{\overline{z}}{\overline{w}}_{[z_0]}\\
			&\le (2\psi(0) + A) \diam(Y) + \sup_{\rho \in [0, \infty)} 
			\left(2\rho - c\, d_Y(x, y) e^{\alpha \rho}\right).
	\end{aligned}\end{equation}
We now apply Lemma~\ref{lem:exponential_supremizer} with
$d = d_Y(x,y)$ and 
$D = \diam(Y)$ on the bounds of~\eqref{eq:Gromov_bdry_QS_step_3},
 with the choice of $K=C$ for the left-hand side bound and $K=c$ for the 
right-hand side bound. It follows that there exists a constant 
$C_1 = C_1(\alpha, \psi, \diam(Y)) \in \R$ such that
	\[
	- C_1 - \frac{1}{\alpha}\ln(d_Y(x, y))
	\le \gromp{\overline{z}}{\overline{w}}_{[z_0]} 
	\le C_1 - \frac{1}{\alpha}\ln(d_Y(x, y)).
	\]
Combining this estimate with \eqref{eq:Gromov_bdry_QS_step_1} yields that
	\[
	\frac{1}{2 e^{\eps C_1}} [d_Y(x, y)]^{\, \eps/\alpha}
	\le d_\eps(\overline{z}, \overline{w})
	\le e^{\eps C_1} [d_Y(x, y)]^{\, \eps/\alpha},
	\]
	completing the proof of \eqref{eq:Gromov_bdry_QS_objective}.
\end{proof}

Thus, with the results covered so far, the proof of Theorem \ref{thm:solid_filling_is_Gromov_hyp-II}
is complete.

\end{document}